\documentclass[reqno, 12pt]{amsart}
\pdfoutput=1
\makeatletter
\@namedef{subjclassname@2020}{\textup{2020} Mathematics Subject Classification}
\makeatother

\usepackage[T1]{fontenc}
\usepackage{amsmath}
\usepackage{amssymb}
\usepackage{bm}
\usepackage{mathrsfs}
\usepackage{dsfont}
\usepackage{accents}
\usepackage{stmaryrd}
\usepackage{mathtools}
\allowdisplaybreaks

\usepackage[rgb]{xcolor} 
\definecolor{mygreen}{rgb}{0,0.7,0.3}
\definecolor{myblue}{rgb}{0,0.50,1.20}
\definecolor{orange}{rgb}{2.55,1.65,0}
\definecolor{olive}{rgb}{0.65, 0.75, 0.2}

\definecolor{fillred}{rgb}{1,0.9,0.9}
\definecolor{fillgreen}{rgb}{0.9,1,0.9}

\definecolor{refkey}{rgb}{0,0.7,0.3}
\usepackage{comment}
\definecolor{labelkey}{rgb}{1,0,0}

\usepackage[colorlinks, linkcolor=blue, citecolor=red, urlcolor=cyan,pagebackref]{hyperref}

\usepackage{float}
\usepackage{graphicx,color}
\usepackage{ascmac}
\usepackage{harpoon}

\usepackage{tabularx}

\usepackage{tikz}
\usetikzlibrary{calc,decorations.markings,patterns}
\usetikzlibrary{arrows.meta}
\usetikzlibrary{positioning}
\usetikzlibrary{cd, intersections, decorations.pathmorphing, arrows, decorations.pathreplacing}
\usepackage{pgfplots}
\pgfplotsset{compat=1.17}
\usetikzlibrary{pgfplots.fillbetween}

\usepackage{cleveref}
\crefname{thm}{Theorem}{Theorems}
\crefname{cor}{Corollary}{Corollaries}
\crefname{lem}{Lemma}{Lemmas}
\crefname{prop}{Proposition}{Propositions}
\crefname{dfn}{Definition}{Definitions}
\crefname{ex}{Example}{Examples}
\crefname{claim}{Claim}{Claims}
\crefname{conj}{Conjecture}{Conjectures}
\crefname{conv}{Notation}{Notations}
\crefname{rem}{Remark}{Remarks}
\crefname{prob}{Problem}{Problems}
\crefname{figure}{Figure}{Figures}
\crefname{table}{Table}{Tables}
\crefname{section}{Section}{Sections}
\crefname{subsection}{Section}{Sections}
\crefname{appendix}{Appendix}{Appendices}
\crefname{lemdef}{Lemma-Definition}{Lemma-Definitions}
\crefname{conv}{Convention}{Conventions}
\crefname{introthm}{Theorem}{Theorems}
\crefname{introcor}{Corollary}{Corollaries}
\crefname{introconj}{Conjecture}{Conjectures}

\usepackage{amsthm}

\newtheorem{thm}{Theorem}
\newtheorem{lem}[thm]{Lemma}
\newtheorem{prop}[thm]{Proposition}
\newtheorem{cor}[thm]{Corollary}

\theoremstyle{definition}
\newtheorem{dfn}[thm]{Definition}
\newtheorem{rem}[thm]{Remark}
\newtheorem{ex}[thm]{Example}

\numberwithin{figure}{section}
\numberwithin{equation}{section}
\numberwithin{thm}{section}

\newcommand{\bZ}{\mathbb{Z}}

\newcommand{\bR}{\mathbb{R}}

\newcommand{\bL}{\mathbb{L}}

\newcommand{\bT}{\mathbb{T}}
\newcommand{\bM}{M} 

\newcommand{\bP}{\mathbb{P}}

\newcommand{\bW}{\mathbb{W}}
\newcommand{\bZP}{\mathbb{Z}_q\mathbb{P}}
\newcommand{\bp}{\mathbf{p}}

\newcommand{\sfs}{\mathsf{s}}

\newcommand{\sfC}{\mathsf{C}}

\newcommand{\uf}{\mathrm{uf}}
\newcommand{\f}{\mathrm{f}}

\newcommand{\spe}{M}
\newcommand{\punc}{{\bM_{\circ}}}
\newcommand{\tri}{\triangle}

\newcommand{\ve}{\varepsilon}

\newcommand{\A}{\mathcal{A}}

\newcommand{\cF}{\mathcal{F}}

\newcommand{\cI}{\mathcal{I}}

\newcommand{\cL}{\mathcal{L}}

\newcommand{\cO}{\mathcal{O}}

\newcommand{\cR}{\mathcal{R}}

\newcommand{\cT}{\mathcal{T}}

\newcommand{\X}{\mathcal{X}}
\newcommand{\cX}{\mathcal{X}}

\newcommand{\scS}{\mathscr{S}}
\newcommand{\SK}[1]{\mathscr{S}^{A}_{#1}}
\newcommand{\Skein}[1]{\mathscr{S}_{#1}^q}
\newcommand{\Tri}{\mathrm{Tri}}

\newcommand{\Bweb}[1]{\mathsf{BWeb}_{#1}}

\newcommand{\Teich}{Teichm\"uller}

\newcommand{\CA}{\mathscr{A}}

\DeclareMathOperator{\Spec}{\mathrm{Spec}}

\DeclareMathOperator{\Frac}{\mathrm{Frac}}

\usepackage{version}

\makeatletter
\newcommand{\oset}[3][0ex]{%
  \mathrel{\mathop{#3}\limits^{
    \vbox to#1{\kern-2\ex@
    \hbox{$\scriptstyle#2$}\vss}}}}
\makeatother

\makeatletter
\newcommand{\osetnear}[3][0ex]{%
  \mathrel{\mathop{#3}\limits^{
    \vbox to#1{\kern-.3\ex@
    \hbox{$\scriptstyle#2$}\vss}}}}
\makeatother
\newcommand{\overbarnear}[1]{\osetnear{#1}{-\!\!\!-\!\!\!-}}

\newcommand\qarrow[2]{\draw[->,shorten >=2pt,shorten <=2pt] (#1) -- (#2) [thick];} 

\def\centerarc(#1)(#2:#3:#4)
    {($(#1)+({#4*cos(#2)},{#4*sin(#2)})$) arc(#2:#3:#4) }

\tikzset{
  mid arrow/.style={postaction={decorate,decoration={
        markings,
        mark=at position .5 with {\arrow[#1]{stealth}}
      }}},
}

\tikzset{->-/.style 2 args={
	postaction={decorate},
	decoration={markings, mark=at position #1 with {\arrow[thick, #2]{>}}} 
    },
    ->-/.default={0.5}{}
}
\tikzset{-<-/.style 2 args={
	postaction={decorate},
	decoration={markings, mark=at position #1 with {\arrow[thick, #2]{<}}} 
    },
    -<-/.default={0.5}{}
}
\tikzset{->>-/.style 2 args={
	postaction={decorate},
	decoration={markings, mark=at position #1 with {\arrow[thick, #2]{>>}}} 
    },
    ->-/.default={0.5}{}
}
\tikzset{-<<-/.style 2 args={
	postaction={decorate},
	decoration={markings, mark=at position #1 with {\arrow[thick, #2]{<<}}} 
    },
    -<-/.default={0.5}{}
}
\tikzset{-/-/.style 2 args={
	postaction={decorate},
	decoration={markings, mark=at position #1 with {\arrow[thick,rotate=-30, #2]{|}}} 
    },
    -<-/.default={0.5}{}
}

\tikzset{
	overarc/.style={
		white, double=red, double distance=1.6pt, line width=2.4pt
	}
}

\tikzset{
	webline/.style={
		red, ultra thick
	}
}

\tikzset{
	wline/.style={
		olive, line width=2.4pt
	}
}

\newcommand{\bdryline}[3]{
    \coordinate (temp1) at #1;
    \coordinate (temp2) at #2;
    \coordinate (temp11) at ($(temp1)!#3!-90:(temp2)$);
    \coordinate (temp22) at ($(temp2)!#3!90:(temp1)$);
    \filldraw[gray!30] (temp1) -- (temp11) -- (temp22) -- (temp2) --cycle;
    \draw[very thick] (temp1) -- (temp2);
}

\title[Skein and cluster algebras with coefficients for unpunctured surfaces]{Skein and cluster algebras with coefficients for unpunctured surfaces}

\author{Tsukasa Ishibashi}
\address{Tsukasa Ishibashi, Mathematical Institute, Tohoku University, 
6-3 Aoba, Aramaki, Aoba-ku, Sendai, Miyagi 980-8578, Japan.}
\email{tsukasa.ishibashi.a6@tohoku.ac.jp}
\urladdr{https://sites.google.com/view/tsukasa-ishibashi/home} 

\author{Shunsuke Kano}
\address{Shunsuke Kano, Mathematical Science Center for Co-creative Society, Tohoku University, 
468-1 Aoba, Aramaki, Aoba-ku, Sendai, Miyagi 980-0845, Japan.}
\email{s.kano@tohoku.ac.jp}
\urladdr{https://sites.google.com/view/shunsuke-kano} 

\author[Wataru Yuasa]{Wataru Yuasa}
\address{Wataru Yuasa, Department of Mathematics, Kyoto University,
Kitashirakawa Oiwake-cho, Sakyo-ku, Kyoto 606-8502, Japan.}
\email{yuasa.wataru.6m@kyoto-u.ac.jp}
\urladdr{https://wataruyuasa.github.io/math/} 


\begin{document}

\begin{abstract}
We propose a skein model for the quantum cluster algebras of surface type with coefficients. We introduce a skein algebra $\mathscr{S}_{\Sigma,\mathbb{W}}^{A}$ of a \emph{walled surface} $(\Sigma,\mathbb{W})$, and prove that it has a quantum cluster structure. 
The walled surfaces naturally generalize the marked surfaces with multi-laminations, which have been used to describe the quantum cluster algebras of geometric type for marked surfaces by Fomin--Thurston \cite{FT18}.
Moreover, we give skein theoretic interpretation for some of quasi-homomorphisms \cite{Fra16} between these quantum cluster algebras.
\end{abstract}
\maketitle

\setcounter{tocdepth}{1}
\tableofcontents


\section{Introduction}\label{sec:intro}


\subsection{Skein realization of the quantum cluster algebra of marked surfaces}
For a compact oriented surface $\Sigma$ equipped with a finite number of marked points (called a \emph{marked surface}), we may associate several algebras of interest. One such algebra is the \emph{cluster algebra} introduced by Fomin--Zelevinsky \cite{FZ-CA1}, which is a commutative algebra equipped with distinguished generators called the \emph{cluster variables}. There is a cluster algebra $\CA_\Sigma$ defined from topological data of the marked surface $\Sigma$. There are also many instances of cluster algebras arising from various contexts including topology, representation theory, integrable system, mirror symmetry and so on, yielding fruitful connections among these areas. 
When $\Sigma$ has no interior marked points (called \emph{punctures}), we can further quantize $\CA_\Sigma$ to a \emph{quantum cluster algebra} $\CA_\Sigma^q$ by means of Berenstein--Zelevinsky \cite{BZ}, which is a non-commutative deformation of $\CA_\Sigma$ by a quantum parameter $q$. Since the cluster algebra $\CA_\Sigma$ describes the combinatorial structure of the decorated \Teich\ space of Penner \cite{Penner} and the moduli space of decorated twisted $SL_2$-local systems on $\Sigma$ \cite{FG06}, the non-commutative algebra $\CA_\Sigma^q$ is expected to encode the quantum geometry of these spaces. 

Another algebra associated with a topological surface $\Sigma$ is the \emph{Kauffman bracket skein algebra} of $\Sigma\times [0,1]$.
Przytycki~\cite{PrzytyckiSkein} and Turaev~\cite{TuraevSkein} introduce a Kauffman bracket skein module for an oriented $3$-manifold $M$, which consists of links in $M$ with skein relations (\cref{def:KBSR}).
Its relation to the $SL_2$-character variety of the fundamental group of the $3$-manifold is discovered by Bullock~\cite{Bul97,Bul98}.
In the case where $M=\Sigma\times [0,1]$, it has a natural algebra structure and also gives a deformation quantization of the $SL_2$-character variety, namely the moduli space of $SL_2$-local systems on $\Sigma$ (see, for example, Turaev~\cite{Turaev}). 
Muller \cite{Muller16} introduced a variant $\mathscr{S}^q_\Sigma$ of the Kauffman bracket skein algebra for a marked surface $\Sigma$ without punctures by assigning an appropriate \emph{boundary skein relations} at marked points. He obtained the following comparison result of this algebra with the quantum cluster algebra $\CA_\Sigma^q$:

\begin{thm}[{\cite{Muller16}}]\label{thm:Muller}
For any (triangulable) marked surface $\Sigma$ without punctures, we have $\mathscr{A}_\Sigma^q \subset \mathscr{S}^q_\Sigma$
in $\mathop{\mathrm{Frac}} \mathscr{S}^q_\Sigma$. 
Moreover, if $\Sigma$ has at least two marked points, we have 
\begin{align*} 
    \mathscr{A}_\Sigma^q= \mathscr{S}^q_\Sigma[\partial^{-1}].
\end{align*}
Here the right-hand side is the localized skein algebra along the boundary arcs, corresponding to the frozen variables.
\end{thm}
This result tells us that the two ``quantization'' approaches of the moduli space of $SL_2$-local systems agree with each other. Moreover, the topological/diagrammatic manipulations in the skein algebra make it much easier to compute Laurent expressions of an element of the quantum cluster algebra in a given cluster, and it leads to a natural construction of a linear basis \cite{Muller16} and quantum duality map \cite{Thu14,MQ23}. 

\subsection{Quantum cluster algebras with coefficients and skein algebras of walled surfaces}\label{subsec:intro_wskein_cluster}
In general, one may introduce \emph{coefficients} for a cluster algebra \cite{FZ-CA4}. A cluster algebra with coefficients in a semifield $\bP$ is an algebra over the group ring $\bZ\bP$, where the exchange relations among the cluster variables involve coefficients in $\bP$. 
Various important concepts in the theory of cluster algebras are defined by using a particular kind of coefficients, called the \emph{principal coefficients}.
For example, $F$-polynomials, $c$-/$g$-vectors and $Y$-patterns \cite{FZ-CA4} are such concepts. 
Also, Fomin--Shapiro--Thurston \cite{FST} exhibited various examples that are isomorphic to cluster algebras of marked surfaces with non-trivial coefficients, including the homogeneous coordinate rings of Grassmannians, and certain affine varieties related to the Lie theory. 
In the above mentioned cases, the coefficients are in a tropical semifield. Let us confine ourselves to this case. 

Our aim in this paper is to provide a skein realization of the quantum cluster algebras associated with a marked surface $\Sigma$ and with (not necessarily normalized) coefficients in a tropical semifield, generalizing the Muller's work ($\bP=\{1\}$). Here we remark that we are going to carefully distinguish the ``frozen variables'' from the ``coefficients'', following the geometric approach of \cite{BFMN20}. 
The quantum cluster algebras both having frozen variables and coefficients are formulated in \cite{CFMM}. 

%

We introduce the skein algebra $\Skein{\Sigma, \bW}$ of a ``walled surface'' $(\Sigma, \bW)$, which is a marked surface $\Sigma$ equipped with a certain datum $\bW=(\sfC,J,\ell)$ called a \emph{wall system}. Here $\sfC$ is a set of certain curves (which may have crossings), $J$ is a finite index set, and $\ell: \sfC \to J$ is a surjective map called a \emph{labeling}. The index set $J$ corresponds to that of the multi-lamination $\bL=(L_j)_{j \in J}$ in the Fomin--Thurston's approach \cite{FT18} as we will discuss in \cref{subsec:intro_comparison_FT}, and hence each subset $\sfC_j:=\ell^{-1}(j)$ with common labeling can be viewed as a generalization of a lamination. 
We introduce a new kind of skein relations which we call the \emph{wall-passing relations} (\cref{def:wall-pass}). 
Walls are allowed to cross each other. In the case where each subset $\sfC_j$ has no crossings for all $j \in J$, then it corresponds to a \emph{normalized coefficients}, and the wall system corresponds to a multi-lamination by certain ``shiftings'' of endpoints in several ways. See \cref{subsec:intro_comparison_FT} below. 

Before stating our results, we show our motivating example of the skein relation in $\Skein{\Sigma, \bW}$.
For an ideal triangulation $\tri$, we take the wall system $\bW_\tri$ which consists of the walls parallel to the ideal arcs in $\tri$.
We call $\bW_\tri$ the \emph{principal wall system} associated with $\tri$ (\cref{ex:wall_principal}), as it will correspond to the principal coefficients.
Then, we consider the multiplication of $\kappa \in \tri$ and the ideal arc $\kappa'$ obtained by flipping $\kappa$ in $\tri$:
\begin{align}
    \begin{tikzpicture}[scale=0.8, baseline=30]
    \draw [gray!30, line width=8pt] (-2.55,2.9) .. controls (-2.2,3) and (-2,3.2) .. (-1.9,3.55);
    \draw [gray!30, line width=8pt] (0.9,3.55) .. controls (1,3.25) and (1.25,3) .. (1.55,2.9);
    \draw [gray!30, line width=8pt] (-2.55,0.1) .. controls (-2.25,0) and (-2,-0.25) .. (-1.9,-0.55);
    \draw [gray!30, line width=8pt] (0.9,-0.55) .. controls (1.05,-0.25) and (1.3,0) .. (1.55,0.1);
    \draw (-2.5,2.75) .. controls (-2.1,2.9) and (-1.9,3.1) .. (-1.75,3.5);
    \draw (0.75,3.5) .. controls (0.9,3.15) and (1.15,2.9) .. (1.5,2.75);
    \draw (1.5,0.25) .. controls (1.15,0.1) and (0.95,-0.15) .. (0.75,-0.5);
    \draw (-1.75,-0.5) .. controls (-1.9,-0.15) and (-2.15,0.1) .. (-2.5,0.25);
    \node [fill, circle, inner sep=1.3] (v1) at (-2.05,3.05) {};
    \node [fill, circle, inner sep=1.3] (v2) at (-2.05,-0.05) {};
    \node [fill, circle, inner sep=1.3] (v3) at (1.05,-0.05) {};
    \node [fill, circle, inner sep=1.3] (v4) at (1.05,3.05) {};
    \draw [blue] (v1) edge (v2);
    \draw [blue] (v2) edge (v3);
    \draw [blue] (v3) edge (v4);
    \draw [blue] (v1) edge (v4);
    \draw [webline] (v1) -- (v3);
    \draw [webline, overarc] (v2) -- (v4);
    \node [blue] at (-0.5,3.35) {$\beta_1$};
    \node [blue] at (-2.4,1.5) {$\alpha_1$};
    \node [blue] at (-0.5,-0.4) {$\beta_2$};
    \node [blue] at (1.4,1.5) {$\alpha_2$};
    \node [red] at (0,2.4) {$\kappa$};
    \node [red] at (-1,2.45) {$\kappa'$};
    \draw[wline] (v2) to[bend right=15] node[midway,right]{$\gamma_\kappa$} (v4);
    \end{tikzpicture}
    =
    q
    \begin{tikzpicture}[scale=0.8, baseline=30]
    \draw [gray!30, line width=8pt] (-2.55,2.9) .. controls (-2.2,3) and (-2,3.2) .. (-1.9,3.55);
    \draw [gray!30, line width=8pt] (0.9,3.55) .. controls (1,3.25) and (1.25,3) .. (1.55,2.9);
    \draw [gray!30, line width=8pt] (-2.55,0.1) .. controls (-2.25,0) and (-2,-0.25) .. (-1.9,-0.55);
    \draw [gray!30, line width=8pt] (0.9,-0.55) .. controls (1.05,-0.25) and (1.3,0) .. (1.55,0.1);
    \draw (-2.5,2.75) .. controls (-2.1,2.9) and (-1.9,3.1) .. (-1.75,3.5);
    \draw (0.75,3.5) .. controls (0.9,3.15) and (1.15,2.9) .. (1.5,2.75);
    \draw (1.5,0.25) .. controls (1.15,0.1) and (0.95,-0.15) .. (0.75,-0.5);
    \draw (-1.75,-0.5) .. controls (-1.9,-0.15) and (-2.15,0.1) .. (-2.5,0.25);
    \node [fill, circle, inner sep=1.3] (v1) at (-2.05,3.05) {};
    \node [fill, circle, inner sep=1.3] (v2) at (-2.05,-0.05) {};
    \node [fill, circle, inner sep=1.3] (v3) at (1.05,-0.05) {};
    \node [fill, circle, inner sep=1.3] (v4) at (1.05,3.05) {};
    \draw [blue] (v1) edge (v2);
    \draw [blue] (v2) edge (v3);
    \draw [blue] (v3) edge (v4);
    \draw [blue] (v1) edge (v4);
    \draw [webline, rounded corners] (v1) -- (-0.55,1.55) -- (v2);
    \draw [webline, rounded corners] (v3) -- (-0.55,1.55) -- (v4);
    \node [blue] at (-0.5,3.35) {$\beta_1$};
    \node [blue] at (-2.4,1.5) {$\alpha_1$};
    \node [blue] at (-0.5,-0.4) {$\beta_2$};
    \node [blue] at (1.4,1.5) {$\alpha_2$};
    \draw[wline] (v2) to[bend right=15] node[midway,right]{$\gamma_\kappa$} (v4);
    \end{tikzpicture}
    +q^{-1}
    \begin{tikzpicture}[scale=0.8, baseline=30]
    \draw [gray!30, line width=8pt] (-2.55,2.9) .. controls (-2.2,3) and (-2,3.2) .. (-1.9,3.55);
    \draw [gray!30, line width=8pt] (0.9,3.55) .. controls (1,3.25) and (1.25,3) .. (1.55,2.9);
    \draw [gray!30, line width=8pt] (-2.55,0.1) .. controls (-2.25,0) and (-2,-0.25) .. (-1.9,-0.55);
    \draw [gray!30, line width=8pt] (0.9,-0.55) .. controls (1.05,-0.25) and (1.3,0) .. (1.55,0.1);
    \draw (-2.5,2.75) .. controls (-2.1,2.9) and (-1.9,3.1) .. (-1.75,3.5);
    \draw (0.75,3.5) .. controls (0.9,3.15) and (1.15,2.9) .. (1.5,2.75);
    \draw (1.5,0.25) .. controls (1.15,0.1) and (0.95,-0.15) .. (0.75,-0.5);
    \draw (-1.75,-0.5) .. controls (-1.9,-0.15) and (-2.15,0.1) .. (-2.5,0.25);
    \node [fill, circle, inner sep=1.3] (v1) at (-2.05,3.05) {};
    \node [fill, circle, inner sep=1.3] (v2) at (-2.05,-0.05) {};
    \node [fill, circle, inner sep=1.3] (v3) at (1.05,-0.05) {};
    \node [fill, circle, inner sep=1.3] (v4) at (1.05,3.05) {};
    \draw [blue] (v1) edge (v2);
    \draw [blue] (v2) edge (v3);
    \draw [blue] (v3) edge (v4);
    \draw [blue] (v1) edge (v4);
    \draw [webline, rounded corners] (v1) -- (-0.55,1.55) -- (v4);
    \draw [webline, rounded corners] (v3) -- (-0.55,1.55) -- (v2);
    \node [blue] at (-0.5,3.35) {$\beta_1$};
    \node [blue] at (-2.4,1.5) {$\alpha_1$};
    \node [blue] at (-0.5,-0.4) {$\beta_2$};
    \node [blue] at (1.4,1.5) {$\alpha_2$};
    \draw[wline] (v2) to[bend right=15] node[midway,right]{$\gamma_\kappa$} (v4);
    \end{tikzpicture}
    \label{eq:walled_skein_rel_elem}
\end{align}
By the wall passing relation \eqref{rel:wall-pass-ext}, the equation \eqref{eq:walled_skein_rel_elem} is rewritten as
\begin{align}
\kappa \kappa' = q z_{\kappa, +} \alpha_1 \alpha_2 + q^{-1} z_{\kappa, -} \beta_1 \beta_2.
\end{align}
Specializing $q=1$ and $z_{\kappa, -}=1$, we get the mutation relation in the cluster algebra with principal coefficients.
See the equation just before Example 3.4 in \cite{FZ-CA4}, for instance.


Our first result establishes a quantum cluster structure of the skein algebra $\Skein{\Sigma, \bW}$ for a general wall system $\bW$ (with taut walls, \cref{def:W-minimal}).
Namely, given an ideal triangulation $\tri$ of the underlying marked surface $\Sigma$, we construct a quantum seed with non-normalized coefficients in the fraction skew-field $\Frac\Skein{\Sigma, \bW}$ (\cref{lem:mut_in_skein}).
Let $\CA^q_{\Sigma, \bW}$ be the quantum cluster algebra associated with such a quantum seed. Then we have the following: 

\begin{thm}[\cref{thm:wskein_cluster}]\label{introthm:wskein_cluster}
For any (triangulable) walled surface $(\Sigma, \bW)$ (with taut walls), we have $\CA^q_{\Sigma, \bW} \subset \Skein{\Sigma, \bW}$ 
in $\mathop{\mathrm{Frac}}\Skein{\Sigma, \bW}$.
Moreover, the equality
\begin{align*}
    \CA^q_{\Sigma, \bW} =\Skein{\Sigma, \bW}[\partial^{-1}]
\end{align*}
holds if $\Sigma$ has at least two marked points, where $\Skein{\Sigma, \bW}[\partial^{-1}]$ is the boundary-localized skein algebra (\cref{def:boundary_localization}).
\end{thm}
We construct a linear basis $\Bweb{\Sigma,\bW}$ of $\Skein{\Sigma, \bW}$ in \cref{thm:basis-web}.  In \cref{ex:MW,ex:MWgeneral}, we demonstrate how to get a Laurent expression of an element of $\Skein{\Sigma, \bW}$ in a given cluster. Our computation by skein relations is much simpler and geometrically clearer, compared to the \emph{snake graph} method developed by \cite{MW13}.

\subsection{Comparison to the Fomin--Thurston's approach}\label{subsec:intro_comparison_FT}
A cluster algebra with coefficients is said to be of \emph{geometric type} if $\bP$ is a tropical semifield, and the coefficients are \emph{normalized} (see \cref{sect:cluster} for details). Fomin--Thurston \cite{FT18} observed that any 
cluster algebra of geometric type associated with a marked surface can be ``encoded'' in the (unbounded) integral laminations \cite{FG07}
(see also \cref{sec:lamination,sec:comparison}).
Let $\CA_{\Sigma, \bL}$ denote the cluster algebra of geometric type whose exchange matrices are induced from ideal triangulations of a marked surface $\Sigma$, and the normalized coefficients given by a multi-lamination $\bL$
on $\Sigma$.
If we specialize the frozen variables to be $1$, then it has a positive realization by \emph{laminated \Teich\ spaces} $\cT(\Sigma,\bL)$ \cite[Chapter 15]{FT18}:

\begin{thm}[\cite{FT18}]
For any (triangulable) marked surface $\Sigma$ and a multi-lamination $\bL$, we have a positive realization
\begin{align*}
    \CA_{\Sigma, \bL}\big|_{\mathrm{frozen}=1} \subset C^\infty(\cT(\Sigma,\bL)).
\end{align*}
Here the cluster variables are realized as laminated lambda length, while coefficients are certain weight parameters on $\cT(\Sigma,\bL)$. 
\end{thm}

When $\sfC_j$ has no crossings for all $j \in J$, we see that the quantum cluster algebra constructed in \cref{introthm:wskein_cluster} are of geometric type (in particular, normalized). 
In this case, we can canonically construct a multi-lamination $\bL(\bW)$ from the wall system $\bW$. For example, if $\bW_\tri$ is the {principal wall system} associated with an ideal triangulation $\tri$ (\cref{ex:wall_principal}), 
then $\bL(\bW_\tri)=\bL_\tri^\pm$ corresponds to the \emph{double principal coefficients} (\cref{ex:principal}). 
Generally, we have the following:
\begin{thm}[\cref{cor:wskein_cluster_lam}]
If $\sfC_j$ has no crossings for any $j \in J$, then
we have a canonical $\bZ_{q,\bW}$-algebra isomorphism
\begin{align*}
    \CA^q_{\Sigma, \bW} \xrightarrow{\sim} \CA^q_{\Sigma, \bL(\bW)}.
\end{align*}
Here, $\bZ_{q, \bW}$ is the base ring of the skein algebra $\Skein{\Sigma, \bW}$ and the right-hand side is the quantum cluster algebra with normalized coefficients given by the multi-lamination $\bL(\bW)$ (\cref{subsec:comparison}).
\end{thm}



The multi-lamination $\bL(\bW)$ constructed as above does not realize all possible normalized coefficients. Roughly speaking, they need to contain ``pairs'' of curves arising from each wall. This reflects the fact that the skein algebra $\CA^q_{\Sigma, \bW}$ contains two parameters $z_{j,+}$, $z_{j,-}$ for each $j \in J$. 

In order to realize the coefficients associated with any multi-lamination $\bL$, 
we conversely define a wall system $\bW(\bL)$ out of $\bL$. For example, if $\bL_\tri^+$ is the multi-lamination corresponding to the principal coefficients, then $\bW(\bL_\tri^+)=\bW_\tri$ is the principal wall system. 
Then we take the quotient $\Skein{\Sigma, \bW(\bL)}\big|_{\boldsymbol{z}_{-}=1}$ of $\Skein{\Sigma, \bW(\bL)}$ that specializes the redundant parameters $z_{j,-}$ to be $1$ (see \eqref{eq:quotient_z-} for a precise definition). Then we get the following:

\begin{thm}[\cref{thm:realization_any_lamination}]
For any multi-lamination $\bL$ on $\Sigma$, we have a canonical inclusion $\CA^q_{\Sigma, \bL} \subset \Skein{\Sigma, \bW(\bL)}\big|_{\boldsymbol{z}_{-}=1}$ of $\bZP_{\bL}$-algebras. Moreover, the equality 
\begin{align*}
    \CA^q_{\Sigma, \bL} =\Skein{\Sigma, \bW(\bL)}\big|_{\boldsymbol{z}_{-}=1}[\partial^{-1}]
\end{align*}
holds if $\Sigma$ has at least two marked points. 
\end{thm}
Thus we get a ``skein realization'' of the quantum cluster algebras $\CA^q_{\Sigma, \bL}$ with any normalized coefficients.
Again we emphasize that our general construction in \cref{introthm:wskein_cluster} also gives skein realizations in non-normalized cases. 



\subsection{Quasi-homomorphisms of skein algebras}
Fraser \cite{Fra16} introduced a map of cluster algebras of the same type but different coefficients, called \emph{quasi-homomorphism}, which rescales the cluster variables and coefficients in a way commuting with mutations. This notion is especially useful to extract certain structures of (quantum) cluster algebras that do not essentially depend on coefficients, and in particular ``rescaling'' known bases (for instance, see \cite{MQ23}). 
In this paper, we try to understand some quasi-homomorphisms between quantum cluster algebras of surface type in terms of skein algebras of walled surfaces.

\paragraph{\textbf{(1) Resolution of crossings.}}
Recall that a wall system $\bW=(\sfC,J,\ell)$ is allowed to have crossings. 
Let $\bW'$ be a wall system obtained from $\bW$ by resolving some of the crossings of walls with a common label in any possible directions as shown below, keeping their labels: 
 \begin{align}
            \ \tikz[baseline=-.6ex, scale=.1]{
                \draw[dashed] (0,0) circle(5cm);
                \draw[wline] (-45:5) -- (135:5) node[olive,left,scale=0.8]{$j$};
                \draw[wline] (-135:5) -- (45:5)node[olive,right,scale=0.8]{$j$};
            \ }
            \mapsto 
            \ \tikz[baseline=-.6ex, scale=.1]{
                \draw[dashed] (0,0) circle(5cm);
                \draw[wline] (45:5) node[olive,right,scale=0.8]{$j$} to[out=south west, in=north west] (-45:5);
                \draw[wline] (-135:5) to[out=north east, in=south east] (135:5) node[olive,left,scale=0.8]{$j$};
            \ }
            \mbox{ or }\ \tikz[baseline=-.6ex, scale=.1]{
                \draw[dashed] (0,0) circle(5cm);
                \draw[wline] (-45:5) node[olive,right,scale=0.8]{$j$} to[out=north west, in=north east] (-135:5);
                \draw[wline] (135:5) node[olive,left,scale=0.8]{$j$} to[out=south east, in=south west] (45:5);
            \ }
    \end{align}


\begin{thm}[\cref{thm:qhom_sm}]
For two wall systems $\bW$ and $\bW'$ related as above, we have a canonical $\bZ_{q, \bW}$-algebra homomorphism
\begin{align*}
    \Psi: \Skein{\Sigma, \bW} \to \Skein{\Sigma, \bW'}.
\end{align*}
Moreover, if $\Sigma$ has at least two marked points and $\bW'$ is still taut, then the map $\Psi$ is a quasi-homomorphism
with respect to the cluster structures described in \cref{introthm:wskein_cluster}.
\end{thm}
In the case where $\bW'$ has no crossings, the cluster structure of $\Skein{\Sigma, \bW'}$ is normalized. In this case, the map $\Psi$ gives a skein-theoretic interpretation of the ``normalization'' construction given in \cite[Proposition 4.3]{Fra16}.

\paragraph{\textbf{(2) Specialization from the principal wall.}}
We investigate another example of quasi-homomorphisms from the principal wall system
$\bW_\tri$ associated with an ideal triangulation $\tri$, which corresponds to the double principal coefficients.

\begin{thm}[\cref{thm:specialization_prin}]
For any taut wall system $\bW = (\mathsf{C}, J, \ell)$ on $\Sigma$, we have a $\bZ_{q,\bW_\tri}$-algebra homomorphism 
\begin{align*}
    \Psi_\tri: \Skein{\Sigma, \bW_\tri} \to \Skein{\Sigma, \bW}
\end{align*}
such that it is a quasi-homomorphism.
Here the $\bZ_{q,\bW_\tri}$-algebra structure of $\Skein{\Sigma, \bW}$ is defined by \eqref{eq:prin_psi}.
\end{thm}

For any wall system $\bW$, let $\bZ^+_{q,\bW} \subset \bZ_{q,\bW}$ be the sub-monoid consisting of elements of the form $\sum_{\lambda,\mu,\nu} c_{\lambda,\mu,\nu} q^{\lambda/2} \prod_{j \in J} z_{j,+}^{\mu_j} z_{j,-}^{\nu_j}$ with $c_{\lambda,\mu,\nu} \geq 0$.
A $\bZ_{q,\bW}$-basis $\mathbf{B}=\{B_\lambda \mid \lambda \in \Lambda\}$ of $\Skein{\Sigma, \bW}$ is said to be \emph{positive} if its structure constants belong to $\bZ^+_{q,\bW}$. 

\begin{cor}[\cref{cor:positive_bases}]
If $\mathbf{B}=\{B_\lambda \mid \lambda \in \Lambda\}$ is a positive $\bZ_{q,\bW_\tri}$-basis of $\Skein{\Sigma, \bW_\tri}$, then $(\Psi_\tri)_\ast\mathbf{B}:=\{\Psi_\tri(B_\lambda) \mid \lambda \in \Lambda\}$ is a positive $\bZ_{q,\bW}$-basis of $\Skein{\Sigma, \bW}$. 
\end{cor}
As a particular example, the \emph{bracelets basis} \cite{Thu14} is known to be positive for the principal coefficients by Mandel--Qin \cite{MQ23}. Therefore it gives a positive basis of $\Skein{\Sigma, \bW_\tri}$, which induces a positive basis of $\Skein{\Sigma, \bW}$ via the corollary above.

\subsection{Future topics of research}

\paragraph{\textbf{Stated skein algebras of walled surfaces.}}
In the forthcoming work \cite{IKY2}, we will investigate a generalization of the \emph{stated skein algebras} \cite{Le_triangular} for a walled surface, in particular the quantum trace maps associated with ideal triangulations. We will state a generalization of \cite[Theorem 5.2]{LY22}, giving an isomorphism between the reduced version of such stated skein algebra and $\Skein{\Sigma, \bW}[\partial^{-1}]$. The transition between the quantum trace maps associated with different triangulations would realize the quantum cluster $\X$-transformations with coefficients of \cite{CFMM}. 

\paragraph{\textbf{Betti model for the classical limit $q^{1/2}=1$.}}
In the case of empty wall $\bW=\emptyset$ (\emph{i.e.}, Muller's setting), the classical limit $\mathscr{S}_{\Sigma}^1[\partial^{-1}]$ coincides with the function algebra $\cO(\A_{SL_2,\Sigma}^\times)$ of (the generic part of) the moduli space of decorated twisted $SL_2$-local systems \cite{IOS}. So it would be interesting to seek for a moduli space $\A_{SL_2,(\Sigma,\bW)}$ for a walled surface $(\Sigma,\bW)$ that corresponds to the skein algebra $\mathscr{S}_{\Sigma,\bW}^1$. It should be fibered over $\Spec \bZ_{1,\bW}$. 

\subsection*{Organization of the paper}
Our notation on marked surfaces and quantum tori is summarized below in this section. 

In \cref{sect:skein}, we introduce the skein algebras of walled surfaces, and investigate their basic properties. Proofs for some technical statements on the minimal position are postponed until \cref{sec:minimal}. 
In \cref{subsec:skein_formula}, we obtain a formula for computing a skein relation in an embedded square in an arbitrary walled surface $(\Sigma,\bW)$, which is important for the comparison with the quantum exchange relations. Some generalizations of walled surfaces are proposed in \cref{sec:generalization}, which would be of independent interest from the skein point of view (and possible new connection to the cluster algebra). 

In \cref{sect:cluster}, we 
naturally extend the concepts of quantum cluster algebras for the setting with coefficients, following \cite{BZ,FZ-CA4,CFMM}.
We recall the (quantum) cluster algebras of surface type, and the realization of normalized coefficients by integral laminations \cite{FT18}. 
In \cref{sec:comparison}, we connect the two theories and prove all the statements in the introduction above. 

\cref{sec:minimal} is devoted to the discussion on the \emph{$\bW$-minimal positions} of curves (\cref{cor:minimal-curve}), which is necessary to unambiguously construct cluster variables and a graphical basis inside $\Skein{\Sigma,\bW}$. 

\subsection*{Acknowledgements.}
T. I. is supported by JSPS KAKENHI Grant Number~JP20K22304, JP24K16914.
S. K. is partially supported by scientific research support of Research Alliance Center for Mathematical Sciences and Mathematical Science Center for Co-creative Society, Tohoku University.
W. Y. is supported by Grant-in-Aid for Early-Career Scientists Grant Number JP19K14528, JP23K12972.
\bigskip

\subsection*{Notation on marked surfaces}
A marked surface $(\Sigma,M)$ is a connected compact oriented surface $\Sigma$ together with a fixed non-empty finite set $M \subset \partial\Sigma$ of \emph{marked points}. Notice that we do not consider interior marked points (``punctures'').  
When the choice of $M$ is clear from the context, we simply denote a marked surface by $\Sigma$. 
We always assume the following conditions:
\begin{enumerate}
    \item[(S1)] Each boundary component (if exists) has at least one marked point.
    \item[(S2)] $-2\chi(\Sigma)+|\spe| >0$.
\end{enumerate}
An \emph{ideal arc} on $\Sigma$ is an immersed arc on $\Sigma$ having endpoints on $\bM$, without self-crossings except for the endpoints. 
The conditions (S1) and (S2) ensure that the existence of an \emph{ideal triangulation}, which is 
a maximal collection $\tri$ of ideal arcs that are mutually non-isotopic and disjoint except for their endpoints. The complementary regions of $\Sigma \setminus \bigcup \tri$ are called \emph{ideal triangles}. 

For an ideal triangulation $\tri$, let $\tri_\f \subset \tri$ denote the subset consisting of boundary arcs and hence $\tri_\uf:=\tri\setminus\tri_\f$ is the subset consisting of interior ideal arcs. Then we have
\begin{align*}
    |\tri|= -3 \chi(\Sigma) + 2|\spe|, \quad 
    |\tri_\uf|= -3 \chi(\Sigma) + |\spe|. 
\end{align*}
For $\alpha \in \tri_\uf$, the \emph{flip} operation produces a new ideal triangulation $\tri'$ of $\Sigma$ obtained by replacing $\alpha$ with the other diagonal of the square containing $\alpha$ as a diagonal (see \cref{fig:tri_quivl}). Let $\Tri_\Sigma$ denote the $|\tri_\uf|$-regular graph with vertices given by the ideal triangulations of $\Sigma$, and edges given by flips among them. It is classically known to be connected.  


\smallskip

\paragraph{ \textbf{Notation on quantum tori.}} 
Let $\cR$ be a unital commutative ring with an invertible element $q \in \cR$. 
Our typical choice of $\cR$ is the ring $\bZ_q:=\bZ[q^{\pm 1/2}]$ of Laurent polynomials on a formal quantum parameter $q^{1/2}$, or its extensions by adding extra parameters. 

\begin{dfn}
Given a lattice $\Lambda$ equipped with a skew-symmetric form $(-,-):\Lambda \times \Lambda \to \bZ$, the associated \emph{based quantum torus} over $\cR$ is  is the associative $\cR$-algebra $\bT$ having
\begin{itemize}
    \item a free $\cR$-basis $B_\lambda$ parametrized by $\lambda \in \Lambda$, and
    \item the product given by $B_\lambda\cdot B_\mu = q^{(\lambda,\mu)/2} B_{\lambda+\mu}$. 
\end{itemize}
The \emph{rank} of $\bT_\Lambda$ is defined to be the rank of $\Lambda$. The map $B:\Lambda \to \bT$, $\lambda \mapsto B_\lambda$ is called the \emph{framing} of $\bT$. 
\end{dfn}
With this understanding, we will simply say that $B:\Lambda \to \bT$ is a based quantum torus. Observe that for a given a basis $(e_i)_{i=1,\dots,N}$ of the lattice $\Lambda$, the based quantum torus $B:\Lambda \to \bT$ is determined by its values on the basis $B_i:=B_{e_i}$ for $i=1,\dots,N$. Indeed, the element associated with a vector $\lambda=\sum_i \lambda_i e_i$ can be written as
\begin{align*}
    B_{\lambda} = q^{-\frac 1 2 \sum_{i < j}\lambda_i\lambda_j\omega_{ij}} B_{1}\cdots B_{N},
\end{align*}
where $\omega_{ij}:=(e_i,e_j)$. In view of this, we will also say that $B:\Lambda \to \bT$ is a based quantum torus generated by the elements $(B_i)_{i=1,\dots,N}$.

\section{Skein algebras of walled surfaces}\label{sect:skein}
\subsection{Walled surfaces and skein relations}
We will consider an additional structure on a marked surface, called a \emph{wall}, and define the skein algebra for such a marked surface with a wall.
Let $(\Sigma,\bM)$ be a marked surface. 
\begin{dfn}\label{dfn:wall}
    A \emph{wall system} $\bW=(\mathsf{C}, J, \ell)$ consists of the following data: 
    \begin{itemize}
        \item A set $\mathsf{C}=\mathsf{C}_{\mathrm{loop}}\sqcup\mathsf{C}_{\mathrm{arc}}$ of mutually distinct simple loops $\mathsf{C}_{\mathrm{loop}}$ and ideal arcs $\mathsf{C}_{\mathrm{arc}}$ on $\Sigma$ only with transverse double points except for their endpoints;
        \item A surjective map $\ell\colon \mathsf{C}\to J$ called a \emph{labeling} valued in a finite index set $J$. $\mathsf{C}_j:=\ell^{-1}(j)$ is the subset of curves with a label $j \in J$.
    \end{itemize} 
    A pair $\xi=(\gamma,j)$ with $\gamma \in \mathsf{C}$ and $j \in J$ with $\ell(\gamma)=j$ is called a \emph{wall}.
    The underlying curve $\gamma$ of a wall $\xi = (\gamma, j)$ is also sometimes called a wall.
    The pair $(\Sigma,\bW)$ is called a \emph{walled surface}. 
    A connected component of $\Sigma\setminus \bigcup\mathsf{C}$ is called a \emph{chamber} of $(\Sigma,\bW)$.
\end{dfn}

\begin{ex}\label{ex:wall_principal}
Given an ideal triangulation $\tri$ of $\Sigma$, consider the wall system $\bW_\tri=(\sfC_\tri,J,\ell)$ given by 
$\sfC_\tri:=\tri_\uf$, $J:=\{1,\dots,|\tri_\uf|\}$, together with a bijection $\ell:\sfC_\tri \to J$. We call $\bW_\tri$ the \emph{principal wall system} associated with $\tri$. One may also consider any smaller index set $J'$ together with a surjection $\ell': \sfC_\tri \to J'$, which will correspond to assigning a common coefficient variable to some of the walls. 
\end{ex}

See \cref{fig:walled-polygons} for more examples related to the examples in \cite{FT18}. 

\begin{figure}[ht]
    \begin{center}
        \begin{tikzpicture}[scale=.1]
            \coordinate (A) at (-135:15);
            \coordinate (B) at (135:15);
            \coordinate (C) at (45:15);
            \coordinate (D) at (-45:15);
            \draw[wline] (A) to[bend right] (C);
            \draw[wline] (A) to[bend left] (C);
            \draw[wline] (B) to[bend right] (D);
            \draw[wline] (B) to[bend left] (D);
            \bdryline{(A)}{(B)}{-2cm}
            \bdryline{(B)}{(C)}{-2cm}
            \bdryline{(C)}{(D)}{-2cm}
            \bdryline{(D)}{(A)}{-2cm}
            \draw[fill] (A) circle (20pt);
            \draw[fill] (B) circle (20pt);
            \draw[fill] (C) circle (20pt);
            \draw[fill] (D) circle (20pt);
            \node at (A) [xshift=15pt,yshift=10pt]{\scriptsize $1$};
            \node at (B) [xshift=10pt,yshift=-15pt]{\scriptsize $2$};
            \node at (A) [xshift=10pt,yshift=15pt]{\scriptsize $3$};
            \node at (B) [xshift=15pt,yshift=-10pt]{\scriptsize $4$};
        \end{tikzpicture}
        \hspace{1em}
        \begin{tikzpicture}[scale=.1]
            \foreach \i [evaluate=\i as \x using (\i+1)*60]in {0,1,...,5}
            {
                \coordinate (A\i) at (\x:15);
            }
            \foreach \i [evaluate=\i as \j using {mod(\i+2,6)}]in {0,1,...,5}
            {
                \draw[wline] (A\i) -- (A\j);
            }
            \foreach \i [evaluate=\i as \j using {mod(\i+1,6)}]in {0,1,...,5}
            {
                \bdryline{(A\i)}{(A\j)}{2cm}
            }
            \foreach \k in {0,1,...,5}
            {
                \draw[fill] (A\k) circle (20pt);
            }
        \end{tikzpicture}
        \hspace{1em}
        \begin{tikzpicture}[scale=.1, rotate=-60]
            \foreach \i [evaluate=\i as \x using (\i+1)*60]in {0,1,...,5}
            {
                \coordinate (A\i) at (\x:15);
            }
            \foreach \i [evaluate=\i as \j using {mod(\i+2,6)}]in {1,2,4,5}
            {
                \draw[wline] (A\i) -- (A\j);
            }
            \draw[wline] (A0) -- (A3);
            \draw[wline] (A2) -- (A5);
            \foreach \i [evaluate=\i as \j using {mod(\i+1,6)}]in {0,1,...,5}
            {
                \bdryline{(A\i)}{(A\j)}{2cm}
            }
            \foreach \k in {0,1,...,5}
            {
                \draw[fill] (A\k) circle (20pt);
            }
        \end{tikzpicture}
        \hspace{.6cm}
        \begin{tikzpicture}[scale=.9]
        \filldraw [gray!30]  (-1.7,2) rectangle (-1.5,-1);
        \draw [wline] (-1.5,-1) .. controls (-0.65,-0.7) and (0.25,-0.6) .. (0.75,-0.6);
        \draw [wline] (-1,2) .. controls (-0.5,1) and (0.5,0) .. (1,-0.25);
        \draw [wline] (-0.5,2) .. controls (0,1) and (0.5,0.5) .. (1.25,0.1);
        \draw [wline] (0,2) .. controls (0.35,1.4) and (0.75,0.95) .. (1.5,0.5);
        \draw [very thick] (-1.5,2) node [fill, circle, inner sep=1.5] (v5) {} -- (-1.5,0.5) node [fill, circle, inner sep=1.5] {} -- (-1.5,-1) node [fill, circle, inner sep=1.5] (v1) {};
        \node [fill, circle, inner sep=1.5] (v2) at (0.5,-1) {};
        \node [fill, circle, inner sep=1.5] (v3) at (1.5,0.5) {};
        \node [fill, circle, inner sep=1.5] (v4) at (0.5,2) {};
        \draw [->>-={.35}{}] (v1) -- (v2);
        \draw [->-={.4}{}] (v2) -- (v3);
        \draw [-<<-={.8}{}] (v3) -- (v4);
        \draw [-<-={.65}{}] (v4) -- (v5);
        \end{tikzpicture}
    \end{center}
    \caption{Examples of walled surfaces. The first three examples correspond to surfaces with laminations in \cite{FT18}. They are related to $\mathrm{Gr}_{2,4}$ and $\mathrm{Gr}_{2,6}$, and $\mathrm{SL}_4/N$. The fourth example is a torus with one boundary component equipped with a single wall.}
    \label{fig:walled-polygons}
\end{figure}
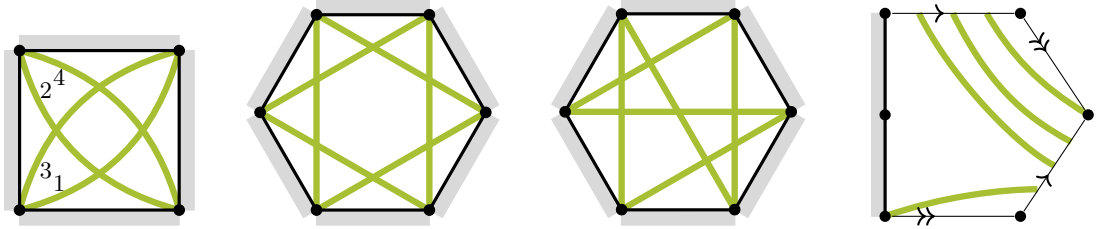

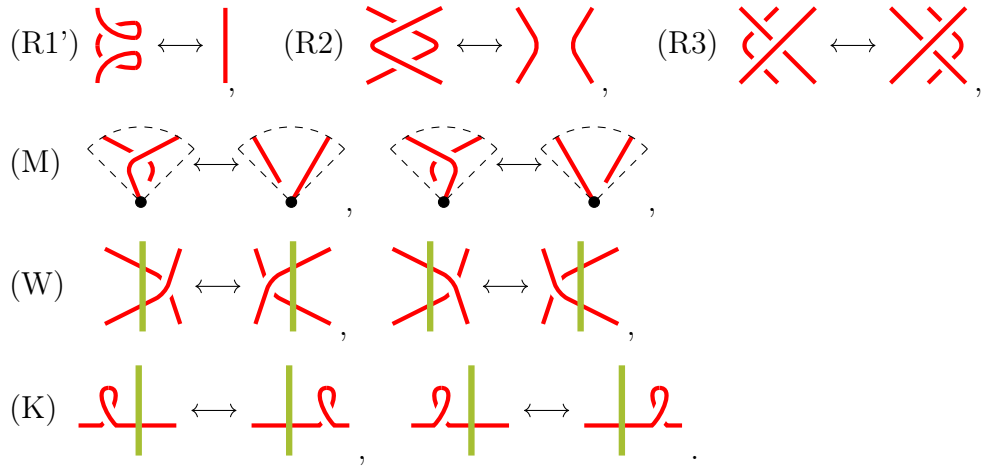
\begin{figure}[ht] 
    \centering
\begin{align*}
    &\begin{tikzpicture}[scale=.1]
        \begin{scope}[xshift=-6cm]
            \draw[webline] (-5,5) to[out=south, in=west] (0,2);
            \draw[webline, overarc] (-5,0) to[out=north, in=west] (0,4);
            \draw[webline] (0,4) to[out=east, in=east] (0,2);
            \draw[webline] (-5,5) to[out=north, in=west] (0,8);
            \draw[webline, overarc] (-5,10) to[out=south, in=west] (0,6);
            \draw[webline] (0,8) to[out=east, in=east] (0,6);
            \node at (-12,5) {(R1')};
        \end{scope}
        \begin{scope}
            \draw[<->] (-3,5) -- (3,5);
        \end{scope}
        \begin{scope}[xshift=6cm]
            \draw[webline] (0,0) -- (0,10);
        \end{scope}
    \end{tikzpicture}
    , \quad
    \begin{tikzpicture}[scale=.1]
        \begin{scope}[xshift=-10cm]
            \draw[webline, rounded corners] (-5,0) -- (5,5) -- (-5,10);
            \draw[webline, overarc, rounded corners] (5,0) -- (-5,5) -- (5,10);
            \node at (-12,5) {(R2)};
        \end{scope}
        \begin{scope}
            \draw[<->] (-3,5) -- (3,5);
        \end{scope}
        \begin{scope}[xshift=10cm]
            \draw[webline, rounded corners] (-5,0) -- (-2,5) -- (-5,10);
            \draw[webline, overarc, rounded corners] (5,0) -- (2,5) -- (5,10);
        \end{scope}
    \end{tikzpicture}
    , \quad
    \begin{tikzpicture}[scale=.1]
        \begin{scope}[xshift=-10cm]
            \draw[webline, rounded corners] (0,0) -- (-5,5) -- (0,10);
            \draw[webline, overarc] (5,0) -- (-5,10);
            \draw[webline, overarc] (-5,0) -- (5,10);
            \node at (-12,5) {(R3)};
        \end{scope}
        \begin{scope}
            \draw[<->] (-3,5) -- (3,5);
        \end{scope}
        \begin{scope}[xshift=10cm]
            \draw[webline, rounded corners] (0,0) -- (5,5) -- (0,10);
            \draw[webline, overarc] (5,0) -- (-5,10);
            \draw[webline, overarc] (-5,0) -- (5,10);
        \end{scope}
    \end{tikzpicture}
    ,\\ 
    &\begin{tikzpicture}[scale=.1]
        \begin{scope}[xshift=-10cm]
            \draw[webline, rounded corners, shorten <=.3cm] (0,0) -- (2,5) -- (120:10);
            \draw[webline, overarc, rounded corners] (0,0) -- (-2,5) -- (60:10);
            \draw[fill] (0,0) circle (20pt);
            \draw[dashed] (45:10) arc (45:135:10cm);
            \draw[dashed] (45:10) -- (0,0);
            \draw[dashed] (135:10) -- (0,0);
            \node at (-14,5) {(M)};
        \end{scope}
        \begin{scope}
            \draw[<->] (-3,5) -- (3,5);
        \end{scope}
        \begin{scope}[xshift=10cm]
            \draw[webline, shorten <=.3cm] (0,0) -- (120:10);
            \draw[webline] (0,0) -- (60:10);
            \draw[fill] (0,0) circle (20pt);
            \draw[dashed] (45:10) arc (45:135:10cm);
            \draw[dashed] (45:10) -- (0,0);
            \draw[dashed] (135:10) -- (0,0);
        \end{scope}
    \end{tikzpicture}
    , \quad
    \begin{tikzpicture}[scale=.1]
        \begin{scope}[xshift=-10cm]
            \draw[webline, overarc, rounded corners, shorten <=.3cm] (0,0) -- (-2,5) -- (60:10);
            \draw[webline, overarc, rounded corners] (0,0) -- (2,5) -- (120:10);
            \draw[fill] (0,0) circle (20pt);
            \draw[dashed] (45:10) arc (45:135:10cm);
            \draw[dashed] (45:10) -- (0,0);
            \draw[dashed] (135:10) -- (0,0);
        \end{scope}
        \begin{scope}
            \draw[<->] (-3,5) -- (3,5);
        \end{scope}
        \begin{scope}[xshift=10cm]
            \draw[webline] (0,0) -- (120:10);
            \draw[webline, shorten <=.3cm] (0,0) -- (60:10);
            \draw[fill] (0,0) circle (20pt);
            \draw[dashed] (45:10) arc (45:135:10cm);
            \draw[dashed] (45:10) -- (0,0);
            \draw[dashed] (135:10) -- (0,0);
        \end{scope}
    \end{tikzpicture}
    ,\\ 
    &\begin{tikzpicture}[scale=.1]
        \begin{scope}[xshift=-10cm]
            \draw[webline, overarc, rounded corners] (5,0) -- (3,6) -- (-5,10);
            \draw[webline, overarc, rounded corners] (-5,0) -- (3,4) -- (5,10);
            \draw[wline] (0,-1) -- (0,11);
            \node at (-14,5) {(W)};
        \end{scope}
        \begin{scope}
            \draw[<->] (-3,5) -- (3,5);
        \end{scope}
        \begin{scope}[xshift=10cm]
            \draw[webline, overarc, rounded corners] (5,0) -- (-3,4) -- (-5,10);
            \draw[webline, overarc, rounded corners] (-5,0) -- (-3,6) -- (5,10);
            \draw[wline] (0,-1) -- (0,11);
        \end{scope}
    \end{tikzpicture}
    , \quad
    \begin{tikzpicture}[scale=.1]
        \begin{scope}[xshift=-10cm]
            \draw[webline, overarc, rounded corners] (-5,0) -- (3,4) -- (5,10);
            \draw[webline, overarc, rounded corners] (5,0) -- (3,6) -- (-5,10);
            \draw[wline] (0,-1) -- (0,11);
        \end{scope}
        \begin{scope}
            \draw[<->] (-3,5) -- (3,5);
        \end{scope}
        \begin{scope}[xshift=10cm]
            \draw[webline, overarc, rounded corners] (-5,0) -- (-3,6) -- (5,10);
            \draw[webline, overarc, rounded corners] (5,0) -- (-3,4) -- (-5,10);
            \draw[wline] (0,-1) -- (0,11);
        \end{scope}
    \end{tikzpicture}
    ,\\ 
    &\begin{tikzpicture}[scale=.1]
        \begin{scope}[xshift=-10cm]
            \draw[webline, overarc] (-8,3) -- (-5,3) to[out=north east, in=east] (-4,8);
            \draw[webline, overarc] (-4,8) to[out=west, in=north west] (-3,3) -- (5,3);
            \draw[wline] (0,-1) -- (0,11);
            \node at (-14,5) {(K)};
        \end{scope}
        \begin{scope}
            \draw[<->] (-3,5) -- (3,5);
        \end{scope}
        \begin{scope}[xshift=10cm]
            \draw[webline, overarc] (-5,3) -- (4,3) to[out=north east, in=east] (5,8);
            \draw[webline, overarc] (5,8) to[out=west, in=north west] (6,3) -- (8,3);
            \draw[wline] (0,-1) -- (0,11);
        \end{scope}
    \end{tikzpicture}
    , \quad
    \begin{tikzpicture}[scale=.1]
        \begin{scope}[xshift=-10cm]
            \draw[webline, overarc] (-4,8) to[out=west, in=north west] (-3,3) -- (5,3);
            \draw[webline, overarc] (-8,3) -- (-5,3) to[out=north east, in=east] (-4,8);
            \draw[wline] (0,-1) -- (0,11);
        \end{scope}
        \begin{scope}
            \draw[<->] (-3,5) -- (3,5);
        \end{scope}
        \begin{scope}[xshift=10cm]
            \draw[webline, overarc] (5,8) to[out=west, in=north west] (6,3) -- (8,3);
            \draw[webline, overarc] (-5,3) -- (4,3) to[out=north east, in=east] (5,8);
            \draw[wline] (0,-1) -- (0,11);
        \end{scope}
    \end{tikzpicture}.
\end{align*}
    \vspace{-2mm}
    \caption{
        The Reidemeister moves with a wall. (R1'), (R2), and (R3) are the ordinary (framed) Reidemeister moves. (M) describes a local diagram in a chamber, where an over-passing arc passes through the under-passing arc at a marked point in the chamber. (W) is related to an ``isotopy of the thickened walled surface $(\Sigma, \bW)\times [0,1]$''. A kink can pass through a wall by (K).
    }
    \label{fig:Reidemeister}
\end{figure}

\begin{dfn}[Tangles in a walled surface]\label{def:tangle}
    Let $\bW=(\sfC,J,\ell)$ be a wall system on $\Sigma$.
    \begin{enumerate}
    \item 
    A \emph{tangle diagram $D$} in the walled surface $(\Sigma,\bW)$ is an immersion of a disjoint union of (unoriented) circles and intervals to $\Sigma$ satisfying the following condition:
    \begin{itemize}
        \item each intersection of $D\cup \bigcup\mathsf{C}$ in $\Sigma\setminus\partial\Sigma$ is a transverse double-point;
        \item each endpoint of arcs of $D$ lies in $\bM$;
        \item each intersection point of $D$ in $\Sigma\setminus\partial\Sigma$ (called \emph{internal crossings}) has over/under-passing information; 
        \item for each $p\in\bM$, there is a map from the set of half-edges incident to $p$ to a totally ordered set. The images of half-edges are called the \emph{elevations at $p$}, and the half-edges are said to have a \emph{simultaneous crossing} if they have the same elevation.
    \end{itemize}
    The first condition is rephrased that the diagram $D$ is \emph{$\bW$-transverse}.
    We denote the set of tangle diagrams in $(\Sigma,\bW)$ by $\mathsf{Diag}(\Sigma,\bW)$.
    \item Two tangle diagrams $D$ and $D'$ are said to be \emph{equivalent} if they are related by a sequence of the Reidemeister moves with walls defined in \cref{fig:Reidemeister} and $\bW$-isotopy. Here \emph{$\bW$-isotopy} means an isotopy on $\Sigma$ relative to $M \cup \bigcup \sfC$.
    The equivalence class of a tangle diagrams $D$ in $(\Sigma, \bW)$ is called a \emph{framed tangle} in $(\Sigma, \bW)$. 
    We denote the set of framed tangles in $(\Sigma, \bW)$ by $\mathsf{Tang}(\Sigma, \bW)$.
    \end{enumerate}
\end{dfn}

The over/under-passing information is indicated as
\ \tikz[baseline=-.6ex, scale=.1]{
    \draw[webline] (-45:5) -- (135:5);
    \draw[webline, overarc] (-135:5) -- (45:5);
    \draw[dashed] (0,0) circle [radius=5];
\ }.
Let $e_1$ and $e_2$ be adjacent half-edges sharing $p\in\bM$.
They are indicated (abbreviating walls) as 
\ \tikz[baseline=-.6ex, scale=.1, yshift=-2cm]{
    \coordinate (P) at (0,0);
    \draw[webline, shorten <=.2cm] (P) -- (110:6);
    \draw[webline] (P) -- (70:6);
    \draw[dashed] (30:6) arc (30:150:6cm);
    \draw[dashed] (P) -- (30:6);
    \draw[dashed] (P) -- (150:6);
    \draw[fill=black] (P) circle [radius=20pt];
    \node at (110:6) [left]{\scriptsize $e_{1}$};
    \node at (70:6) [right]{\scriptsize $e_{2}$};
    \node at (P) [below]{\scriptsize $p$};
\ } if the elevation of $e_1$ is lower than that of $e_2$, and as
\ \tikz[baseline=-.6ex, scale=.1, yshift=-2cm]{
    \coordinate (P) at (0,0);
    \draw[webline] (P) -- (110:6);
    \draw[webline] (P) -- (70:6);
    \draw[dashed] (30:6) arc (30:150:6cm);
    \draw[dashed] (P) -- (30:6);
    \draw[dashed] (P) -- (150:6);
    \draw[fill=black] (P) circle [radius=20pt];
    \node at (110:6) [left]{\scriptsize $e_{1}$};
    \node at (70:6) [right]{\scriptsize $e_{2}$};
    \node at (P) [below]{\scriptsize $p$};
\ } if $e_1$ and $e_2$ have the same elevation.
Similarly, the elevations of more than two half-edges at $p\in\bM$ are distinguished by their distance from $p$. 

\begin{rem}
Only the orders among the elevation are relevant in what follows. Local moves of a tangle shown in \cref{fig:forbidden} are forbidden.
\end{rem}

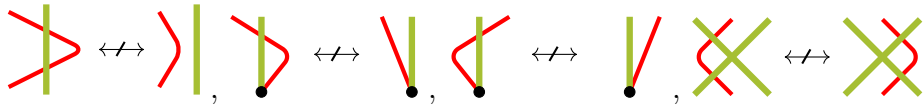
\begin{figure}[ht] 
    \begin{tikzpicture}[scale=.1]
        \begin{scope}[xshift=-10cm]
            \draw[webline, rounded corners] (-5,0) -- (5,5) -- (-5,10);
            \draw[wline] (0,-1) -- (0,11);
        \end{scope}
        \begin{scope}
            \draw[-/-={.5}{}, <->] (-3,5) -- (3,5);
        \end{scope}
        \begin{scope}[xshift=10cm]
            \draw[webline, rounded corners] (-5,0) -- (-2,5) -- (-5,10);
            \draw[wline] (0,-1) -- (0,11);
        \end{scope}
    \end{tikzpicture}
    ,
    \begin{tikzpicture}[scale=.1]
        \begin{scope}[xshift=-10cm]
            \draw[webline, rounded corners] (0,0) -- (4,5) -- (-4,10);
            \draw[wline] (0,0) -- (0,10);
            \draw[fill] (0,0) circle (20pt);
        \end{scope}
        \begin{scope}
            \draw[-/-={.5}{}, <->] (-3,5) -- (3,5);
        \end{scope}
        \begin{scope}[xshift=10cm]
            \draw[webline] (0,0) -- (-4,10);
            \draw[wline] (0,0) -- (0,10);
            \draw[fill] (0,0) circle (20pt);
        \end{scope}
    \end{tikzpicture}
    ,
    \begin{tikzpicture}[scale=.1]
        \begin{scope}[xshift=-10cm]
            \draw[webline, rounded corners] (0,0) -- (-4,5) -- (4,10);
            \draw[wline] (0,0) -- (0,10);
            \draw[fill] (0,0) circle (20pt);
        \end{scope}
        \begin{scope}
            \draw[-/-={.5}{}, <->] (-3,5) -- (3,5);
        \end{scope}
        \begin{scope}[xshift=10cm]
            \draw[webline] (0,0) -- (4,10);
            \draw[wline] (0,0) -- (0,10);
            \draw[fill] (0,0) circle (20pt);
        \end{scope}
    \end{tikzpicture}
    ,
    \begin{tikzpicture}[scale=.1]
        \begin{scope}[xshift=-10cm]
            \draw[webline, rounded corners] (0,0) -- (-5,5) -- (0,10);
            \draw[wline] (5,0) -- (-5,10);
            \draw[wline] (-5,0) -- (5,10);
        \end{scope}
        \begin{scope}
            \draw[-/-={.5}{}, <->] (-3,5) -- (3,5);
        \end{scope}
        \begin{scope}[xshift=10cm]
            \draw[webline, rounded corners] (0,0) -- (5,5) -- (0,10);
            \draw[wline] (5,0) -- (-5,10);
            \draw[wline] (-5,0) -- (5,10);
        \end{scope}
    \end{tikzpicture}
    \\ 

    \caption{Forbidden moves of tangles in a walled surface.}
    \label{fig:forbidden}
\end{figure}

\begin{dfn}[Algebra of framed tangles on a walled surface]
    Let $(\Sigma, \bW)$ be a walled surface, and $\cR$ a unital commutative ring.
    The \emph{algebra $\cR\mathsf{Tang}(\Sigma, \bW)$ of framed tangles in $(\Sigma, \bW)$} is the free $\cR$-module spanned by $\mathsf{Tang}(\Sigma, \bW)$.
    For two tangle diagrams $D_1$ and $D_2$, we define the product $D_1D_2$ by superposing $D_1$ over $D_2$, where $D_1$ has over-passing arcs than those in $D_2$ at any crossings $D_1\cap D_2$.
    The linear extension defines a multiplication in $\cR\mathsf{Tang}(\Sigma, \bW)$.
\end{dfn}

\begin{rem}
    We remark that the union of walls in $(\Sigma, \bW)$ is setwisely fixed by the Reidemeister moves and the $\bW$-isotopy. 
    Thus, the multiplication is well-defined.
\end{rem}

We will define the skein algebra $\SK{\Sigma,\bW}$ as a quotient of the free $\cR$-module $\cR\mathsf{Diag}(\Sigma,\bW)$ by skein relations. Then this quotient map factors through $\cR\mathsf{Tang}(\Sigma,\bW)$ as in the commutative diagram in below:
\begin{equation}\label{eq:def-skein}
\begin{tikzcd}
    \cR\mathsf{Diag}(\Sigma,\bW) \arrow[rd, two heads, "\text{skein rel.}"']\arrow[r,two heads, "\substack{\text{Reidemeister}\\ \text{moves}}"] & [1em]\cR\mathsf{Tang}(\Sigma,\bW) \arrow[d,dashed,two heads, "\exists\ \text{algebra hom.}"]\\
    &\SK{\Sigma,\bW}
\end{tikzcd}
\end{equation}

Before defining our skein relations, let us review some related skein relations defined in previous works.

The Kauffman bracket skein algebra of a surface $\Sigma$
is defined as a quotient of the $\bZ[A,A^{-1}]$-module spanned by link diagrams (tangle diagrams without endpoints) on $\Sigma$ modulo isotopy of $\Sigma$ and \emph{the Kauffman bracket skein relation}:

\begin{dfn}[Kauffman bracket skein relations]\label{def:KBSR}
    \begin{align}
            \ \tikz[baseline=-.6ex, scale=.1]{
                \draw[webline] (-45:5) -- (135:5);
                \draw[webline, overarc] (-135:5) -- (45:5);
            \ }
            &=A\ \tikz[baseline=-.6ex, scale=.1]{
                \draw[webline] (45:5) to[out=south west, in=north west] (-45:5);
                \draw[webline] (-135:5) to[out=north east, in=south east] (135:5);
            \ }
            +A^{-1}\ \tikz[baseline=-.6ex, scale=.1]{
                \draw[webline] (-45:5) to[out=north west, in=north east] (-135:5);
                \draw[webline] (135:5) to[out=south east, in=south west] (45:5);
            \ },\label{rel:Kauffman}\\
        \ \tikz[baseline=-.6ex, scale=.1]{
                \draw[webline] (0,0) circle [radius=4cm];
            \ }
            &=(-A^2-A^{-2}) \textcolor{red}{\bm{\varnothing}},\label{rel:trivial-loop}
    \end{align}
    where $\textcolor{red}{\bm{\varnothing}}$ means that there is no diagram in the neighboorhood.
\end{dfn}
These relations induce the following relation on the framing:
\begin{lem}[Kink relations]
    \begin{align}
        \ \tikz[baseline=-.6ex, scale=.1,yshift=-2cm]{
            \draw[webline, overarc] (-4,0) -- (-1,0) to[out=north east, in=east] (0,5);
            \draw[webline, overarc] (0,5) to[out=west, in=north west] (1,0) -- (4,0);
        \ }
        =-A^3
        \ \tikz[baseline=-.6ex, scale=.1,yshift=-2cm]{
            \draw[webline] (-4,0) -- (4,0);
        \ },
        \ \tikz[baseline=-.6ex, scale=.1,yshift=-2cm]{
            \draw[webline, overarc] (0,5) to[out=west, in=north west] (1,0) -- (4,0);
            \draw[webline, overarc] (-4,0) -- (-1,0) to[out=north east, in=east] (0,5);
        \ }
        =-A^{-3}
        \ \tikz[baseline=-.6ex, scale=.1,yshift=-2cm]{
            \draw[webline] (-4,0) -- (4,0);
        \ }
    \end{align}
\end{lem}

The skein algebra has many variants~\cite{BonahonWong11,Muller16,PS19,RogerYang14}, according to several purposes and types of surfaces.

Let us define skein relations related to marked points.
These relations will play a similar role as the Jones--Wenzl projector, also known as the $\mathfrak{sl}_2$-clasp.
The skein algebra $\SK{\Sigma}$ of a marked surface $\Sigma$ (with empty wall system) defined by Muller~\cite{Muller16} is a quotient of $\bZ[A^{\pm 1/2}]\mathsf{Diag}(\Sigma, \emptyset)$ modulo isotopy relative to $\partial\Sigma$, the Kauffman bracket skein relations, and the following relations at marked points:

\begin{dfn}[Clasped skein relations~\cite{Muller16}]\label{def:KBSR-mark}
    \begin{align}
        A^{-\frac{1}{2}}\ \tikz[baseline=-.6ex, scale=.1, yshift=-2cm]{
            \coordinate (A) at (-6,0);
            \coordinate (B) at (6,0);
            \coordinate (P) at (0,0);
            \draw[webline, shorten <=.2cm] (P) -- (110:6);
            \draw[webline] (P) -- (70:6);
            \draw[dashed] (30:6) arc (30:150:6cm);
            \draw[dashed] (P) -- (30:6);
            \draw[dashed] (P) -- (150:6);
            \bdryline{(A)}{(B)}{2cm}
            \draw[fill=black] (P) circle [radius=20pt];
        }\ 
        &=\ \tikz[baseline=-.6ex, scale=.1, yshift=-2cm]{
            \coordinate (A) at (-6,0);
            \coordinate (B) at (6,0);
            \coordinate (P) at (0,0);
            \draw[webline] (P) -- (110:6);
            \draw[webline] (P) -- (70:6);
            \draw[dashed] (30:6) arc (30:150:6cm);
            \draw[dashed] (P) -- (30:6);
            \draw[dashed] (P) -- (150:6);
            \bdryline{(A)}{(B)}{2cm}
            \draw[fill=black] (P) circle [radius=20pt];
        }\ 
        =A^{\frac{1}{2}}\ \tikz[baseline=-.6ex, scale=.1, yshift=-2cm]{
            \coordinate (A) at (-6,0);
            \coordinate (B) at (6,0);
            \coordinate (P) at (0,0);
            \draw[webline] (P) -- (110:6);
            \draw[webline, shorten <=.2cm] (P) -- (70:6);
            \draw[dashed] (30:6) arc (30:150:6cm);
            \draw[dashed] (P) -- (30:6);
            \draw[dashed] (P) -- (150:6);
            \bdryline{(A)}{(B)}{2cm}
            \draw[fill=black] (P) circle [radius=20pt];
        }\ ,\label{rel:elevation}\\
        \ \tikz[baseline=-.6ex, scale=.1, yshift=-2cm]{
            \coordinate (A) at (-6,0);
            \coordinate (B) at (6,0);
            \coordinate (P) at (0,0);
            \draw[webline, rounded corners] (P) -- (110:4) -- (70:4) -- (P);
            \draw[dashed] (30:6) arc (30:150:6cm);
            \draw[dashed] (P) -- (30:6);
            \draw[dashed] (P) -- (150:6);
            \bdryline{(A)}{(B)}{2cm}
            \draw[fill=black] (P) circle [radius=20pt];
        }\ 
        &=0.\label{rel:monogon}
    \end{align}
    Here, the two half-edges in \labelcref{rel:elevation} have consecutive elevations.
\end{dfn}

One can confirm that the skein relations in \cref{def:KBSR} and \cref{def:KBSR-mark} realize Reidemeister moves (R1)--(R3) in the first line of \cref{fig:Reidemeister}.
This fact induces a map $\bZ[A^{\pm^{1/2}}]\mathsf{Tang}(\Sigma,\emptyset)\twoheadrightarrow\scS_{\Sigma}^{A}$, and it defines the algebra structure on $\scS_{\Sigma}^{A}$.

We are going to introduce new skein relations for tangle diagrams in a walled surface $(\Sigma, \bW)$ and define a skein algebra $\SK{\Sigma, \bW}$
as a quotient of $\bZ_{A,\bW}\mathsf{Diag}(\Sigma, \bW)$.
The coefficient ring $\bZ_{A,\bW}$ is defined by
\begin{align*}
    \bZ_{A,\bW}:=\bZ[A^{\pm 1/2}, z_{j,+}^{\pm 1}, z_{j,-}^{\pm 1} \mid j \in J]
\end{align*}
for a wall system $\bW=(\sfC,J,\ell)$.
Let $a_j:=z_{j,+}z_{j,-}$ for $j \in J$. 

\begin{dfn}[Wall-passing relations]\label{def:wall-pass}
    For any walls, we introduce the relations
    \begin{gather}
    \ \tikz[baseline=-.6ex, scale=.1, yshift=-5cm]{
        \draw[webline, rounded corners] (-5,0) -- (5,5) -- (-5,10);
        \draw[wline] (0,-1) -- (0,11);
        \node at (0,11) [below right]{\scriptsize $j$};
    }\ 
    =a_j
    \ \tikz[baseline=-.6ex, scale=.1, yshift=-5cm]{
            \draw[webline, rounded corners] (-5,0) -- (-2,5) -- (-5,10);
            \draw[wline] (0,-1) -- (0,11);
            \node at (0,11) [below right]{\scriptsize $j$};
    }\ ,\label{rel:wall-pass-int}\\
    \ \tikz[baseline=-.6ex, scale=.1, yshift=-2cm]{
        \coordinate (A) at (-10,0);
        \coordinate (B) at (10,0);
        \coordinate (P) at (0,0);
        \draw[webline, rounded corners] (P) -- (60:6) -- (120:10);
        \draw[wline] (P) -- (90:10);
        \draw[dashed] (30:10) arc (30:150:10cm);
        \draw[dashed] (P) -- (30:10);
        \draw[dashed] (P) -- (150:10);
        \draw[fill=black] (P) circle [radius=20pt];
        \node at (90:10) [below right]{\scriptsize $j$};
    }\ 
    =z_{j,+}
    \ \tikz[baseline=-.6ex, scale=.1, yshift=-2cm]{
        \coordinate (A) at (-10,0);
        \coordinate (B) at (10,0);
        \coordinate (P) at (0,0);
        \draw[webline, rounded corners] (P) -- (120:10);
        \draw[wline] (P) -- (90:10);
        \draw[dashed] (30:10) arc (30:150:10cm);
        \draw[dashed] (P) -- (30:10);
        \draw[dashed] (P) -- (150:10);
        \draw[fill=black] (P) circle [radius=20pt];
        \node at (90:10) [below right]{\scriptsize $j$};
    }\ ,
    \ \tikz[baseline=-.6ex, scale=.1, yshift=-2cm]{
        \coordinate (A) at (-10,0);
        \coordinate (B) at (10,0);
        \coordinate (P) at (0,0);
        \draw[webline, rounded corners] (P) -- (120:6) -- (60:10);
        \draw[wline] (P) -- (90:10);
        \draw[dashed] (30:10) arc (30:150:10cm);
        \draw[dashed] (P) -- (30:10);
        \draw[dashed] (P) -- (150:10);
        \draw[fill=black] (P) circle [radius=20pt];
        \node at (90:10) [below left]{\scriptsize $j$};
    }\ 
    =z_{j,-}
    \ \tikz[baseline=-.6ex, scale=.1, yshift=-2cm]{
        \coordinate (A) at (-10,0);
        \coordinate (B) at (10,0);
        \coordinate (P) at (0,0);
        \draw[webline, rounded corners] (P) -- (60:10);
        \draw[wline] (P) -- (90:10);
        \draw[dashed] (30:10) arc (30:150:10cm);
        \draw[dashed] (P) -- (30:10);
        \draw[dashed] (P) -- (150:10);
        \draw[fill=black] (P) circle [radius=20pt];
        \node at (90:10) [below left]{\scriptsize $j$};
    }\ ,\label{rel:wall-pass-ext}\\
    \ \tikz[baseline=-.6ex, scale=.1, yshift=-5cm]{
        \draw[webline, rounded corners] (0,0) -- (-5,5) -- (0,10);
        \draw[wline] (5,0) -- (-5,10);
        \draw[wline] (-5,0) -- (5,10);
        \node at (5,10) [right]{\scriptsize $j_1$};
        \node at (-5,10) [left]{\scriptsize $j_2$};
    }\ 
    =
    \ \tikz[baseline=-.6ex, scale=.1, yshift=-5cm]{
        \draw[webline, rounded corners] (0,0) -- (5,5) -- (0,10);
        \draw[wline] (5,0) -- (-5,10);
        \draw[wline] (-5,0) -- (5,10);
        \node at (5,10) [right]{\scriptsize $j_1$};
        \node at (-5,10) [left]{\scriptsize $j_2$};
    }\ .\label{rel:wall-R3}
\end{gather}

    where $j,j_1,j_2 \in J$ are the labels of nearby walls.
\end{dfn}

\begin{lem}[The relation (W) in \cref{fig:Reidemeister}]\label{lem:realize-W3}
    \begin{align}
        \ \tikz[baseline=-.6ex, scale=.1, yshift=-5cm]{
            \draw[webline, overarc, rounded corners] (5,0) -- (3,6) -- (-5,10);
                \draw[webline, overarc, rounded corners] (-5,0) -- (3,4) -- (5,10);
                \draw[wline] (0,-1) -- (0,11);
        }\ 
        &=
        \ \tikz[baseline=-.6ex, scale=.1, yshift=-5cm]{
                \draw[webline, overarc, rounded corners] (5,0) -- (-3,4) -- (-5,10);
                \draw[webline, overarc, rounded corners] (-5,0) -- (-3,6) -- (5,10);
                \draw[wline] (0,-1) -- (0,11);
        }\ ,&
        \ \tikz[baseline=-.6ex, scale=.1, yshift=-5cm]{
                \draw[webline, overarc, rounded corners] (-5,0) -- (3,4) -- (5,10);
                \draw[webline, overarc, rounded corners] (5,0) -- (3,6) -- (-5,10);
                \draw[wline] (0,-1) -- (0,11);
        }\ 
        &=
        \ \tikz[baseline=-.6ex, scale=.1, yshift=-5cm]{
            \draw[webline, overarc, rounded corners] (-5,0) -- (-3,6) -- (5,10);
            \draw[webline, overarc, rounded corners] (5,0) -- (-3,4) -- (-5,10);
            \draw[wline] (0,-1) -- (0,11);
        }\ .
\end{align}
\end{lem}
\begin{proof}
    One can easily verify these relations by \labelcref{rel:Kauffman,rel:wall-pass-int}.
\end{proof}


\begin{dfn}[Skein algebra of a walled surface]
    For a walled surface $(\Sigma, \bW)$,
    the \emph{skein algebra $\SK{\Sigma, \bW}$ of a walled surface $(\Sigma,\bW)$} is the quotient of the free $\bZ_{A,\bW}$-module $\bZ_{A,\bW}\mathsf{Diag}(\Sigma, \bW)$ generated by framed tangles in $(\Sigma, \bW)$ modulo $\bW$-isotopy and the skein relations \labelcref{rel:Kauffman,rel:trivial-loop,rel:elevation,rel:monogon,rel:wall-pass-int,rel:wall-pass-ext,rel:wall-R3}.
\end{dfn}

The following lemma implies the commutative diagram \labelcref{eq:def-skein}. 
In particular, the algebra structure on $\SK{\Sigma,\bW}$ is inherited from $\bZ_{A,\bW}\mathsf{Tang}(\Sigma,\bW)$: 
\begin{lem}
    The Kauffman bracket skein relations~\labelcref{rel:Kauffman,rel:trivial-loop} in \cref{def:KBSR}, the clasped skein relations~\labelcref{rel:elevation,rel:monogon} in \cref{def:KBSR-mark}, and the wall-passing relations~\labelcref{rel:wall-pass-int,rel:wall-R3} in \cref{def:wall-pass} realize all the Reidemeister moves of tangle diagrams in a walled surface $(\Sigma, \bW)$.
\end{lem}
\begin{proof}
    It follows from \cref{lem:realize-W3} and Muller's clasped skein relations.
\end{proof}

Observe that we have an isomorphism
\begin{align}\label{eq:forgetful}
    \SK{\Sigma,\bW}\otimes_{\bZ_{A,\bW}} \bZ[A^{\pm 1}] \cong \SK{\Sigma},
\end{align}
where the ring homomorphism $\bZ_{A,\bW} \to \bZ[A^{\pm 1}]$ is given by $z_{j,\pm} \mapsto 1$ for all $j \in J$. In effect, this specialization amounts to forgetting the walls. 

\begin{figure}
    \centering
    \begin{tikzpicture}[scale=.15]
            \coordinate (M1) at (90:15);
            \coordinate (M2) at (90:5);
            \coordinate (M3) at (-90:5);
            \coordinate (M4) at (-90:15);
            \draw[very thick] (0,0) circle [radius=5];
            \draw[very thick] (0,0) circle [radius=15];
            \draw[wline] (M1) -- (M2);
            \draw[wline] (M3) -- (M4);
            \draw[wline] (M2) to[out=north west, in=north] (180:12) to[out=south, in=north west] (M4);
            \draw[wline] (M2) to[out=north east, in=north] (0:12) to[out=south, in=north east] (M4);
            \draw[webline] (0,0) circle [radius=7];
            \fill[black] (M1) circle [radius=.5];
            \fill[black] (M2) circle [radius=.5];
            \fill[black] (M3) circle [radius=.5];
            \fill[black] (M4) circle [radius=.5];
            \node at (-135:10) {$1$};
            \node at (90:10) [right]{$2$};
            \node at (-45:10) {$3$};
            \node at (-90:10) [right]{$4$};
            \node at (0:7) [right, red] {$\gamma$};
            \node at (180:15) [left]{$5$};
            \node at (180:5) [right]{$6$};
            \node at (0:5) [left]{$7$};
            \node at (0:15) [right]{$8$};
    \end{tikzpicture}
    \hspace{1cm}
    \begin{tikzpicture}[scale=.15]
            \coordinate (L1) at (-8,-15);
            \coordinate (L2) at (-8,0);
            \coordinate (L3) at (-8,15);
            \coordinate (R1) at (8,-15);
            \coordinate (R2) at (8,0);
            \coordinate (R3) at (8,15);
            \draw[very thick] (L1) -- (L3) -- (R3) -- (R1) -- cycle;
            \draw[wline] (L1) -- (R1);
            \draw[wline] (L2) -- (R2);
            \draw[wline] (L3) -- (R3);
            \draw[wline] (L1) -- (R2);
            \draw[wline] (L3) -- (R2);
            \draw[webline] ($(L1)!.5!(R1)$) -- ($(L3)!.5!(R3)$);
            \fill[black] (L1) circle [radius=.5];
            \fill[black] (L2) circle [radius=.5];
            \fill[black] (L3) circle [radius=.5];
            \fill[black] (R1) circle [radius=.5];
            \fill[black] (R2) circle [radius=.5];
            \fill[black] (R3) circle [radius=.5];
            \node at ($(L1)!.7!(R1)$) [above]{$4$};
            \node at ($(L1)!.7!(R2)$) [below]{$1$};
            \node at ($(L2)!.3!(R2)$) [above]{$2$};
            \node at ($(L3)!.7!(R2)$) [above]{$3$};
            \node at ($(L3)!.7!(R3)$) [below]{$4$};
            \node at ($(L1)!.5!(L2)$) [left]{$5$};
            \node at ($(L2)!.5!(L3)$) [left]{$8$};
            \node at ($(R1)!.5!(R2)$) [right]{$6$};
            \node at ($(R2)!.5!(R3)$) [right]{$7$};
            \node at (0,5) [left, red] {$\gamma$};
    \end{tikzpicture}
    \caption{A walled annulus $(\Sigma,\bW)$ and a simple loop $\gamma$. The principal wall system $\bW$ consists of $4$ walls labeled with the index set $J=\{1,2,3,4\}$ as indicated in the figure. The boundary intervals are labeled by $\{5,6,7,8\}$. The right picture is the left annulus cut by a wall labeled by $4$.}
    \label{fig:MWex}
\end{figure}
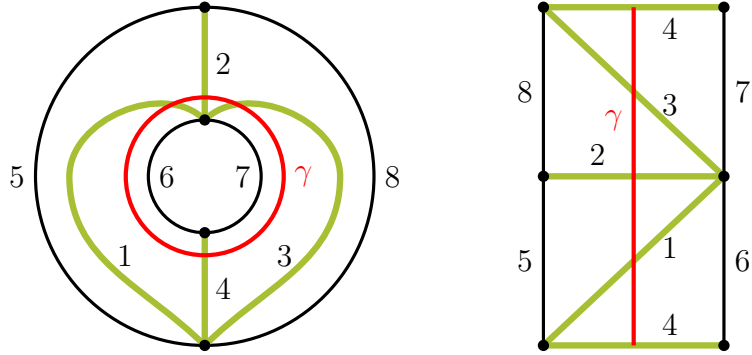

\begin{ex}[Laurent expansion of a simple loop in an annulus]\label{ex:MW}
    Let $\Sigma$ be an annulus with a principal wall system $\bW$ and a simple loop $\gamma$ as shown in \cref{fig:MWex}.
    Using skein relations at $A=1$, we expand $\gamma$ as a polynomial in simple arcs $\alpha_1,\alpha_2,\ldots,\alpha_8$ which corresponds to the walls in $\bW$ and boundary intervals of $\Sigma$.
    This is a skein theoretical calculation of the Laurent expansion of $\gamma$ by Musiker and Williams in \cite[Example~{3.24}]{MW13}, where they used snake graphs and perfect matching.
    We use the abbreviated notation $z_{i_1i_2\cdots i_k,\varepsilon}\coloneqq z_{i_1,\varepsilon}z_{i_2,\varepsilon}\cdots z_{i_k,\varepsilon}$. Since $\gamma$ intersects once with each of $\alpha_1,\alpha_2,\alpha_3,\alpha_4$, we are going to compute the product $\alpha_1\alpha_2\alpha_3\alpha_4 \gamma$. We compute 
    \begin{align*}
    \alpha_2\alpha_4\gamma
    &=
    \tikz[baseline=-.6ex, scale=.1]{
        \coordinate (L1) at (-8,-15);
        \coordinate (L2) at (-8,0);
        \coordinate (L3) at (-8,15);
        \coordinate (R1) at (8,-15);
        \coordinate (R2) at (8,0);
        \coordinate (R3) at (8,15);
        \coordinate (C1) at ($(L1)!.5!(R1)$);
        \coordinate (C2) at ($(L2)!.5!(R2)$);
        \coordinate (C3) at ($(L3)!.5!(R3)$);
        \draw[very thick] (L1) -- (L3) --(R3) -- (R1) -- cycle;
        \draw[wline] (L1) -- (R1);
        \draw[wline] (L2) -- (R2);
        \draw[wline] (L3) -- (R3);
        \draw[wline] (L1) -- (R2);
        \draw[wline] (L3) -- (R2);
        \draw[webline] (C1) -- (C3);
        \draw[webline] (L2) to[bend right] (R2);
        \draw[webline] (L3) to[bend right] (R3);
        \fill[black] (L1) circle[radius=1];
        \fill[black] (L2) circle[radius=1];
        \fill[black] (L3) circle[radius=1];
        \fill[black] (R1) circle[radius=1];
        \fill[black] (R2) circle[radius=1];
        \fill[black] (R3) circle[radius=1];
        \node at (0,5) [left, red]{\small $\gamma$};
        \node at (C3) [below right=5pt, red]{\small $\alpha_4$};
        \node at (C2) [below left=5pt, red]{\small $\alpha_2$};
    }
    =
    \tikz[baseline=-.6ex, scale=.1]{
        \coordinate (L1) at (-8,-15);
        \coordinate (L2) at (-8,0);
        \coordinate (L3) at (-8,15);
        \coordinate (R1) at (8,-15);
        \coordinate (R2) at (8,0);
        \coordinate (R3) at (8,15);
        \coordinate (C1) at ($(L1)!.5!(R1)$);
        \coordinate (C2) at ($(L2)!.5!(R2)$);
        \coordinate (C3) at ($(L3)!.5!(R3)$);
        \draw[very thick] (L1) -- (L3) --(R3) -- (R1) -- cycle;
        \draw[wline] (L1) -- (R1);
        \draw[wline] (L2) -- (R2);
        \draw[wline] (L3) -- (R3);
        \draw[wline] (L1) -- (R2);
        \draw[wline] (L3) -- (R2);
        \draw[webline, rounded corners] (L3) -- ($(C3)+(0,-2)$) -- (C3);
        \draw[webline, rounded corners] (L2) -- ($(C2)+(0,-2)$) -- ($(C3)+(0,-3)$) -- (R3);
        \draw[webline, rounded corners] (C1) -- ($(C2)+(0,-3)$) -- (R2);
        \fill[black] (L1) circle[radius=1];
        \fill[black] (L2) circle[radius=1];
        \fill[black] (L3) circle[radius=1];
        \fill[black] (R1) circle[radius=1];
        \fill[black] (R2) circle[radius=1];
        \fill[black] (R3) circle[radius=1];
    }
    +
    \tikz[baseline=-.6ex, scale=.1]{
        \coordinate (L1) at (-8,-15);
        \coordinate (L2) at (-8,0);
        \coordinate (L3) at (-8,15);
        \coordinate (R1) at (8,-15);
        \coordinate (R2) at (8,0);
        \coordinate (R3) at (8,15);
        \coordinate (C1) at ($(L1)!.5!(R1)$);
        \coordinate (C2) at ($(L2)!.5!(R2)$);
        \coordinate (C3) at ($(L3)!.5!(R3)$);
        \draw[very thick] (L1) -- (L3) --(R3) -- (R1) -- cycle;
        \draw[wline] (L1) -- (R1);
        \draw[wline] (L2) -- (R2);
        \draw[wline] (L3) -- (R3);
        \draw[wline] (L1) -- (R2);
        \draw[wline] (L3) -- (R2);
        \draw[webline, rounded corners] (L3) -- ($(C3)+(0,-2)$) -- (C3);
        \draw[webline, rounded corners] (R2) -- ($(C2)+(0,-2)$) -- ($(C3)+(0,-3)$) -- (R3);
        \draw[webline, rounded corners] (C1) -- ($(C2)+(0,-3)$) -- (L2);
        \fill[black] (L1) circle[radius=1];
        \fill[black] (L2) circle[radius=1];
        \fill[black] (L3) circle[radius=1];
        \fill[black] (R1) circle[radius=1];
        \fill[black] (R2) circle[radius=1];
        \fill[black] (R3) circle[radius=1];
    }
    +
    \tikz[baseline=-.6ex, scale=.1]{
        \coordinate (L1) at (-8,-15);
        \coordinate (L2) at (-8,0);
        \coordinate (L3) at (-8,15);
        \coordinate (R1) at (8,-15);
        \coordinate (R2) at (8,0);
        \coordinate (R3) at (8,15);
        \coordinate (C1) at ($(L1)!.5!(R1)$);
        \coordinate (C2) at ($(L2)!.5!(R2)$);
        \coordinate (C3) at ($(L3)!.5!(R3)$);
        \draw[very thick] (L1) -- (L3) --(R3) -- (R1) -- cycle;
        \draw[wline] (L1) -- (R1);
        \draw[wline] (L2) -- (R2);
        \draw[wline] (L3) -- (R3);
        \draw[wline] (L1) -- (R2);
        \draw[wline] (L3) -- (R2);
        \draw[webline, rounded corners] (R3) -- ($(C3)+(0,-2)$) -- (C3);
        \draw[webline, rounded corners] (L2) -- ($(C2)+(0,-2)$) -- ($(C3)+(0,-3)$) -- (L3);
        \draw[webline, rounded corners] (C1) -- ($(C2)+(0,-3)$) -- (R2);
        \fill[black] (L1) circle[radius=1];
        \fill[black] (L2) circle[radius=1];
        \fill[black] (L3) circle[radius=1];
        \fill[black] (R1) circle[radius=1];
        \fill[black] (R2) circle[radius=1];
        \fill[black] (R3) circle[radius=1];
    }
    +
    \tikz[baseline=-.6ex, scale=.1]{
        \coordinate (L1) at (-8,-15);
        \coordinate (L2) at (-8,0);
        \coordinate (L3) at (-8,15);
        \coordinate (R1) at (8,-15);
        \coordinate (R2) at (8,0);
        \coordinate (R3) at (8,15);
        \coordinate (C1) at ($(L1)!.5!(R1)$);
        \coordinate (C2) at ($(L2)!.5!(R2)$);
        \coordinate (C3) at ($(L3)!.5!(R3)$);
        \draw[very thick] (L1) -- (L3) --(R3) -- (R1) -- cycle;
        \draw[wline] (L1) -- (R1);
        \draw[wline] (L2) -- (R2);
        \draw[wline] (L3) -- (R3);
        \draw[wline] (L1) -- (R2);
        \draw[wline] (L3) -- (R2);
        \draw[webline, rounded corners] (R3) -- ($(C3)+(0,-2)$) -- (C3);
        \draw[webline, rounded corners] (R2) -- ($(C2)+(0,-2)$) -- ($(C3)+(0,-3)$) -- (L3);
        \draw[webline, rounded corners] (C1) -- ($(C2)+(0,-3)$) -- (L2);
        \fill[black] (L1) circle[radius=1];
        \fill[black] (L2) circle[radius=1];
        \fill[black] (L3) circle[radius=1];
        \fill[black] (R1) circle[radius=1];
        \fill[black] (R2) circle[radius=1];
        \fill[black] (R3) circle[radius=1];
    }\\
    &=z_{1234,+}
    \tikz[baseline=-.6ex, scale=.1]{
        \coordinate (L1) at (-8,-15);
        \coordinate (L2) at (-8,0);
        \coordinate (L3) at (-8,15);
        \coordinate (R1) at (8,-15);
        \coordinate (R2) at (8,0);
        \coordinate (R3) at (8,15);
        \coordinate (C1) at ($(L1)!.5!(R1)$);
        \coordinate (C2) at ($(L2)!.5!(R2)$);
        \coordinate (C3) at ($(L3)!.5!(R3)$);
        \draw[very thick] (L1) -- (L3) --(R3) -- (R1) -- cycle;
        \draw[wline] (L1) -- (R1);
        \draw[wline] (L2) -- (R2);
        \draw[wline] (L3) -- (R3);
        \draw[wline] (L1) -- (R2);
        \draw[wline] (L3) -- (R2);
        \draw[webline] (L2) -- (R3);
        \draw[webline] (L1) to[bend left] (R2);
        \fill[black] (L1) circle[radius=1];
        \fill[black] (L2) circle[radius=1];
        \fill[black] (L3) circle[radius=1];
        \fill[black] (R1) circle[radius=1];
        \fill[black] (R2) circle[radius=1];
        \fill[black] (R3) circle[radius=1];
    }
    +z_{14,+}z_{23,-}
    \tikz[baseline=-.6ex, scale=.1]{
        \coordinate (L1) at (-8,-15);
        \coordinate (L2) at (-8,0);
        \coordinate (L3) at (-8,15);
        \coordinate (R1) at (8,-15);
        \coordinate (R2) at (8,0);
        \coordinate (R3) at (8,15);
        \coordinate (C1) at ($(L1)!.5!(R1)$);
        \coordinate (C2) at ($(L2)!.5!(R2)$);
        \coordinate (C3) at ($(L3)!.5!(R3)$);
        \draw[very thick] (L1) -- (L3) --(R3) -- (R1) -- cycle;
        \draw[wline] (L1) -- (R1);
        \draw[wline] (L2) -- (R2);
        \draw[wline] (L3) -- (R3);
        \draw[wline] (L1) -- (R2);
        \draw[wline] (L3) -- (R2);
        \draw[webline] (R2) to[bend left] (R3);
        \draw[webline] (L1) to[bend right] (L2);
        \fill[black] (L1) circle[radius=1];
        \fill[black] (L2) circle[radius=1];
        \fill[black] (L3) circle[radius=1];
        \fill[black] (R1) circle[radius=1];
        \fill[black] (R2) circle[radius=1];
        \fill[black] (R3) circle[radius=1];
    }
    +z_{12,+}z_{34,-}
    \tikz[baseline=-.6ex, scale=.1]{
        \coordinate (L1) at (-8,-15);
        \coordinate (L2) at (-8,0);
        \coordinate (L3) at (-8,15);
        \coordinate (R1) at (8,-15);
        \coordinate (R2) at (8,0);
        \coordinate (R3) at (8,15);
        \coordinate (C1) at ($(L1)!.5!(R1)$);
        \coordinate (C2) at ($(L2)!.5!(R2)$);
        \coordinate (C3) at ($(L3)!.5!(R3)$);
        \draw[very thick] (L1) -- (L3) --(R3) -- (R1) -- cycle;
        \draw[wline] (L1) -- (R1);
        \draw[wline] (L2) -- (R2);
        \draw[wline] (L3) -- (R3);
        \draw[wline] (L1) -- (R2);
        \draw[wline] (L3) -- (R2);
        \draw[webline] (L2) to[bend right] (L3);
        \draw[webline] (R1) to[bend left] (R2);
        \fill[black] (L1) circle[radius=1];
        \fill[black] (L2) circle[radius=1];
        \fill[black] (L3) circle[radius=1];
        \fill[black] (R1) circle[radius=1];
        \fill[black] (R2) circle[radius=1];
        \fill[black] (R3) circle[radius=1];
    }
    +z_{234,-}
    \tikz[baseline=-.6ex, scale=.1]{
        \coordinate (L1) at (-8,-15);
        \coordinate (L2) at (-8,0);
        \coordinate (L3) at (-8,15);
        \coordinate (R1) at (8,-15);
        \coordinate (R2) at (8,0);
        \coordinate (R3) at (8,15);
        \coordinate (C1) at ($(L1)!.5!(R1)$);
        \coordinate (C2) at ($(L2)!.5!(R2)$);
        \coordinate (C3) at ($(L3)!.5!(R3)$);
        \draw[very thick] (L1) -- (L3) --(R3) -- (R1) -- cycle;
        \draw[wline] (L1) -- (R1);
        \draw[wline] (L2) -- (R2);
        \draw[wline] (L3) -- (R3);
        \draw[wline] (L1) -- (R2);
        \draw[wline] (L3) -- (R2);
        \draw[webline] (R2) to[bend left] (L3);
        \draw[webline] (R1) -- (L2);
        \fill[black] (L1) circle[radius=1];
        \fill[black] (L2) circle[radius=1];
        \fill[black] (L3) circle[radius=1];
        \fill[black] (R1) circle[radius=1];
        \fill[black] (R2) circle[radius=1];
        \fill[black] (R3) circle[radius=1];
    }\ ,
    \end{align*}
    \begin{align*}
    \alpha_3\ 
    \tikz[baseline=-.6ex, scale=.1]{
        \coordinate (L1) at (-8,-15);
        \coordinate (L2) at (-8,0);
        \coordinate (L3) at (-8,15);
        \coordinate (R1) at (8,-15);
        \coordinate (R2) at (8,0);
        \coordinate (R3) at (8,15);
        \coordinate (C1) at ($(L1)!.5!(R1)$);
        \coordinate (C2) at ($(L2)!.5!(R2)$);
        \coordinate (C3) at ($(L3)!.5!(R3)$);
        \draw[very thick] (L1) -- (L3) --(R3) -- (R1) -- cycle;
        \draw[wline] (L1) -- (R1);
        \draw[wline] (L2) -- (R2);
        \draw[wline] (L3) -- (R3);
        \draw[wline] (L1) -- (R2);
        \draw[wline] (L3) -- (R2);
        \draw[webline] (L2) -- (R3);
        \draw[webline] (L1) to[bend left] (R2);
        \fill[black] (L1) circle[radius=1];
        \fill[black] (L2) circle[radius=1];
        \fill[black] (L3) circle[radius=1];
        \fill[black] (R1) circle[radius=1];
        \fill[black] (R2) circle[radius=1];
        \fill[black] (R3) circle[radius=1];
    }
    =
    \tikz[baseline=-.6ex, scale=.1]{
        \coordinate (L1) at (-8,-15);
        \coordinate (L2) at (-8,0);
        \coordinate (L3) at (-8,15);
        \coordinate (R1) at (8,-15);
        \coordinate (R2) at (8,0);
        \coordinate (R3) at (8,15);
        \coordinate (C1) at ($(L1)!.5!(R1)$);
        \coordinate (C2) at ($(L2)!.5!(R2)$);
        \coordinate (C3) at ($(L3)!.5!(R3)$);
        \draw[very thick] (L1) -- (L3) --(R3) -- (R1) -- cycle;
        \draw[wline] (L1) -- (R1);
        \draw[wline] (L2) -- (R2);
        \draw[wline] (L3) -- (R3);
        \draw[wline] (L1) -- (R2);
        \draw[wline] (L3) -- (R2);
        \draw[webline] (L3) to[bend left] (R2);
        \draw[webline] (L2) -- (R3);
        \draw[webline] (L1) to[bend left] (R2);
        \fill[black] (L1) circle[radius=1];
        \fill[black] (L2) circle[radius=1];
        \fill[black] (L3) circle[radius=1];
        \fill[black] (R1) circle[radius=1];
        \fill[black] (R2) circle[radius=1];
        \fill[black] (R3) circle[radius=1];
    }
    =
    \tikz[baseline=-.6ex, scale=.1]{
        \coordinate (L1) at (-8,-15);
        \coordinate (L2) at (-8,0);
        \coordinate (L3) at (-8,15);
        \coordinate (R1) at (8,-15);
        \coordinate (R2) at (8,0);
        \coordinate (R3) at (8,15);
        \coordinate (C1) at ($(L1)!.5!(R1)$);
        \coordinate (C2) at ($(L2)!.5!(R2)$);
        \coordinate (C3) at ($(L3)!.5!(R3)$);
        \coordinate (C12) at ($(C1)!.5!(C2)$);
        \coordinate (C23) at ($(C2)!.5!(C3)$);
        \draw[very thick] (L1) -- (L3) --(R3) -- (R1) -- cycle;
        \draw[wline] (L1) -- (R1);
        \draw[wline] (L2) -- (R2);
        \draw[wline] (L3) -- (R3);
        \draw[wline] (L1) -- (R2);
        \draw[wline] (L3) -- (R2);
        \draw[webline, rounded corners] (L3) to[bend right] (R3);
        \draw[webline, rounded corners] (L2) -- ($(C23)+(0,4)$) -- (R2);
        \draw[webline] (L1) to[bend left] (R2);
        \fill[black] (L1) circle[radius=1];
        \fill[black] (L2) circle[radius=1];
        \fill[black] (L3) circle[radius=1];
        \fill[black] (R1) circle[radius=1];
        \fill[black] (R2) circle[radius=1];
        \fill[black] (R3) circle[radius=1];
    }
    +
    \tikz[baseline=-.6ex, scale=.1]{
        \coordinate (L1) at (-8,-15);
        \coordinate (L2) at (-8,0);
        \coordinate (L3) at (-8,15);
        \coordinate (R1) at (8,-15);
        \coordinate (R2) at (8,0);
        \coordinate (R3) at (8,15);
        \coordinate (C1) at ($(L1)!.5!(R1)$);
        \coordinate (C2) at ($(L2)!.5!(R2)$);
        \coordinate (C3) at ($(L3)!.5!(R3)$);
        \coordinate (C12) at ($(C1)!.5!(C2)$);
        \coordinate (C23) at ($(C2)!.5!(C3)$);
        \draw[very thick] (L1) -- (L3) --(R3) -- (R1) -- cycle;
        \draw[wline] (L1) -- (R1);
        \draw[wline] (L2) -- (R2);
        \draw[wline] (L3) -- (R3);
        \draw[wline] (L1) -- (R2);
        \draw[wline] (L3) -- (R2);
        \draw[webline, rounded corners] (L3) -- ($(C23)+(0,4)$) -- (L2);
        \draw[webline] (R2) to[bend left] (R3);
        \draw[webline] (L1) to[bend left] (R2);
        \fill[black] (L1) circle[radius=1];
        \fill[black] (L2) circle[radius=1];
        \fill[black] (L3) circle[radius=1];
        \fill[black] (R1) circle[radius=1];
        \fill[black] (R2) circle[radius=1];
        \fill[black] (R3) circle[radius=1];
    }
    =z_{3,+}
    \tikz[baseline=-.6ex, scale=.1]{
        \coordinate (L1) at (-8,-15);
        \coordinate (L2) at (-8,0);
        \coordinate (L3) at (-8,15);
        \coordinate (R1) at (8,-15);
        \coordinate (R2) at (8,0);
        \coordinate (R3) at (8,15);
        \coordinate (C1) at ($(L1)!.5!(R1)$);
        \coordinate (C2) at ($(L2)!.5!(R2)$);
        \coordinate (C3) at ($(L3)!.5!(R3)$);
        \coordinate (C12) at ($(C1)!.5!(C2)$);
        \coordinate (C23) at ($(C2)!.5!(C3)$);
        \draw[very thick] (L1) -- (L3) --(R3) -- (R1) -- cycle;
        \draw[wline] (L1) -- (R1);
        \draw[wline] (L2) -- (R2);
        \draw[wline] (L3) -- (R3);
        \draw[wline] (L1) -- (R2);
        \draw[wline] (L3) -- (R2);
        \draw[webline] (L3) to[bend right] (R3);
        \draw[webline] (L2) to[bend left] (R2);
        \draw[webline] (L1) to[bend left] (R2);
        \fill[black] (L1) circle[radius=1];
        \fill[black] (L2) circle[radius=1];
        \fill[black] (L3) circle[radius=1];
        \fill[black] (R1) circle[radius=1];
        \fill[black] (R2) circle[radius=1];
        \fill[black] (R3) circle[radius=1];
    }
    +z_{3,-}
    \tikz[baseline=-.6ex, scale=.1]{
        \coordinate (L1) at (-8,-15);
        \coordinate (L2) at (-8,0);
        \coordinate (L3) at (-8,15);
        \coordinate (R1) at (8,-15);
        \coordinate (R2) at (8,0);
        \coordinate (R3) at (8,15);
        \coordinate (C1) at ($(L1)!.5!(R1)$);
        \coordinate (C2) at ($(L2)!.5!(R2)$);
        \coordinate (C3) at ($(L3)!.5!(R3)$);
        \coordinate (C12) at ($(C1)!.5!(C2)$);
        \coordinate (C23) at ($(C2)!.5!(C3)$);
        \draw[very thick] (L1) -- (L3) --(R3) -- (R1) -- cycle;
        \draw[wline] (L1) -- (R1);
        \draw[wline] (L2) -- (R2);
        \draw[wline] (L3) -- (R3);
        \draw[wline] (L1) -- (R2);
        \draw[wline] (L3) -- (R2);
        \draw[webline] (L3) to[bend left] (L2);
        \draw[webline] (R2) to[bend left] (R3);
        \draw[webline] (L1) to[bend left] (R2);
        \fill[black] (L1) circle[radius=1];
        \fill[black] (L2) circle[radius=1];
        \fill[black] (L3) circle[radius=1];
        \fill[black] (R1) circle[radius=1];
        \fill[black] (R2) circle[radius=1];
        \fill[black] (R3) circle[radius=1];
    }
    \end{align*}
    , and
    \begin{align*}
    \alpha_1\ 
    \tikz[baseline=-.6ex, scale=.1]{
        \coordinate (L1) at (-8,-15);
        \coordinate (L2) at (-8,0);
        \coordinate (L3) at (-8,15);
        \coordinate (R1) at (8,-15);
        \coordinate (R2) at (8,0);
        \coordinate (R3) at (8,15);
        \coordinate (C1) at ($(L1)!.5!(R1)$);
        \coordinate (C2) at ($(L2)!.5!(R2)$);
        \coordinate (C3) at ($(L3)!.5!(R3)$);
        \draw[very thick] (L1) -- (L3) --(R3) -- (R1) -- cycle;
        \draw[wline] (L1) -- (R1);
        \draw[wline] (L2) -- (R2);
        \draw[wline] (L3) -- (R3);
        \draw[wline] (L1) -- (R2);
        \draw[wline] (L3) -- (R2);
        \draw[webline] (R2) to[bend left] (L3);
        \draw[webline] (R1) -- (L2);
        \fill[black] (L1) circle[radius=1];
        \fill[black] (L2) circle[radius=1];
        \fill[black] (L3) circle[radius=1];
        \fill[black] (R1) circle[radius=1];
        \fill[black] (R2) circle[radius=1];
        \fill[black] (R3) circle[radius=1];
    }
    =
    \tikz[baseline=-.6ex, scale=.1]{
        \coordinate (L1) at (-8,-15);
        \coordinate (L2) at (-8,0);
        \coordinate (L3) at (-8,15);
        \coordinate (R1) at (8,-15);
        \coordinate (R2) at (8,0);
        \coordinate (R3) at (8,15);
        \coordinate (C1) at ($(L1)!.5!(R1)$);
        \coordinate (C2) at ($(L2)!.5!(R2)$);
        \coordinate (C3) at ($(L3)!.5!(R3)$);
        \draw[very thick] (L1) -- (L3) --(R3) -- (R1) -- cycle;
        \draw[wline] (L1) -- (R1);
        \draw[wline] (L2) -- (R2);
        \draw[wline] (L3) -- (R3);
        \draw[wline] (L1) -- (R2);
        \draw[wline] (L3) -- (R2);
        \draw[webline] (L2) -- (R1);
        \draw[webline] (R2) to[bend right] (L1);
        \draw[webline] (R2) to[bend left] (L3);
        \fill[black] (L1) circle[radius=1];
        \fill[black] (L2) circle[radius=1];
        \fill[black] (L3) circle[radius=1];
        \fill[black] (R1) circle[radius=1];
        \fill[black] (R2) circle[radius=1];
        \fill[black] (R3) circle[radius=1];
    }
    =
    \tikz[baseline=-.6ex, scale=.1]{
        \coordinate (L1) at (-8,-15);
        \coordinate (L2) at (-8,0);
        \coordinate (L3) at (-8,15);
        \coordinate (R1) at (8,-15);
        \coordinate (R2) at (8,0);
        \coordinate (R3) at (8,15);
        \coordinate (C1) at ($(L1)!.5!(R1)$);
        \coordinate (C2) at ($(L2)!.5!(R2)$);
        \coordinate (C3) at ($(L3)!.5!(R3)$);
        \coordinate (C12) at ($(C1)!.5!(C2)$);
        \coordinate (C23) at ($(C2)!.5!(C3)$);
        \draw[very thick] (L1) -- (L3) --(R3) -- (R1) -- cycle;
        \draw[wline] (L1) -- (R1);
        \draw[wline] (L2) -- (R2);
        \draw[wline] (L3) -- (R3);
        \draw[wline] (L1) -- (R2);
        \draw[wline] (L3) -- (R2);
        \draw[webline] (L2) to[bend left] (L1);
        \draw[webline, rounded corners] (R2) -- ($(C12)+(-4,0)$) -- (R1);
        \draw[webline] (R2) to[bend left] (L3);
        \fill[black] (L1) circle[radius=1];
        \fill[black] (L2) circle[radius=1];
        \fill[black] (L3) circle[radius=1];
        \fill[black] (R1) circle[radius=1];
        \fill[black] (R2) circle[radius=1];
        \fill[black] (R3) circle[radius=1];
    }
    +
    \tikz[baseline=-.6ex, scale=.1]{
        \coordinate (L1) at (-8,-15);
        \coordinate (L2) at (-8,0);
        \coordinate (L3) at (-8,15);
        \coordinate (R1) at (8,-15);
        \coordinate (R2) at (8,0);
        \coordinate (R3) at (8,15);
        \coordinate (C1) at ($(L1)!.5!(R1)$);
        \coordinate (C2) at ($(L2)!.5!(R2)$);
        \coordinate (C3) at ($(L3)!.5!(R3)$);
        \coordinate (C12) at ($(C1)!.5!(C2)$);
        \coordinate (C23) at ($(C2)!.5!(C3)$);
        \draw[very thick] (L1) -- (L3) --(R3) -- (R1) -- cycle;
        \draw[wline] (L1) -- (R1);
        \draw[wline] (L2) -- (R2);
        \draw[wline] (L3) -- (R3);
        \draw[wline] (L1) -- (R2);
        \draw[wline] (L3) -- (R2);
        \draw[webline, rounded corners] (L1) -- ($(C12)+(-4,0)$) -- (R1);
        \draw[webline] (R2) to[bend left] (L2);
        \draw[webline] (R2) to[bend left] (L3);
        \fill[black] (L1) circle[radius=1];
        \fill[black] (L2) circle[radius=1];
        \fill[black] (L3) circle[radius=1];
        \fill[black] (R1) circle[radius=1];
        \fill[black] (R2) circle[radius=1];
        \fill[black] (R3) circle[radius=1];
    }
    =z_{1,+}
    \tikz[baseline=-.6ex, scale=.1]{
        \coordinate (L1) at (-8,-15);
        \coordinate (L2) at (-8,0);
        \coordinate (L3) at (-8,15);
        \coordinate (R1) at (8,-15);
        \coordinate (R2) at (8,0);
        \coordinate (R3) at (8,15);
        \coordinate (C1) at ($(L1)!.5!(R1)$);
        \coordinate (C2) at ($(L2)!.5!(R2)$);
        \coordinate (C3) at ($(L3)!.5!(R3)$);
        \coordinate (C12) at ($(C1)!.5!(C2)$);
        \coordinate (C23) at ($(C2)!.5!(C3)$);
        \draw[very thick] (L1) -- (L3) --(R3) -- (R1) -- cycle;
        \draw[wline] (L1) -- (R1);
        \draw[wline] (L2) -- (R2);
        \draw[wline] (L3) -- (R3);
        \draw[wline] (L1) -- (R2);
        \draw[wline] (L3) -- (R2);
        \draw[webline] (L2) to[bend left] (L1);
        \draw[webline] (R2) to[bend right] (R1);
        \draw[webline] (R2) to[bend left] (L3);
        \fill[black] (L1) circle[radius=1];
        \fill[black] (L2) circle[radius=1];
        \fill[black] (L3) circle[radius=1];
        \fill[black] (R1) circle[radius=1];
        \fill[black] (R2) circle[radius=1];
        \fill[black] (R3) circle[radius=1];
    }
    +z_{1,-}
    \tikz[baseline=-.6ex, scale=.1]{
        \coordinate (L1) at (-8,-15);
        \coordinate (L2) at (-8,0);
        \coordinate (L3) at (-8,15);
        \coordinate (R1) at (8,-15);
        \coordinate (R2) at (8,0);
        \coordinate (R3) at (8,15);
        \coordinate (C1) at ($(L1)!.5!(R1)$);
        \coordinate (C2) at ($(L2)!.5!(R2)$);
        \coordinate (C3) at ($(L3)!.5!(R3)$);
        \coordinate (C12) at ($(C1)!.5!(C2)$);
        \coordinate (C23) at ($(C2)!.5!(C3)$);
        \draw[very thick] (L1) -- (L3) --(R3) -- (R1) -- cycle;
        \draw[wline] (L1) -- (R1);
        \draw[wline] (L2) -- (R2);
        \draw[wline] (L3) -- (R3);
        \draw[wline] (L1) -- (R2);
        \draw[wline] (L3) -- (R2);
        \draw[webline] (L1) to[bend left] (R1);
        \draw[webline] (R2) to[bend left] (L2);
        \draw[webline] (R2) to[bend left] (L3);
        \fill[black] (L1) circle[radius=1];
        \fill[black] (L2) circle[radius=1];
        \fill[black] (L3) circle[radius=1];
        \fill[black] (R1) circle[radius=1];
        \fill[black] (R2) circle[radius=1];
        \fill[black] (R3) circle[radius=1];
    }\ .
    \end{align*}
    Thus we get 
    \begin{align*}
        \alpha_1\alpha_2\alpha_3\alpha_4\gamma=&z_{1234,+}\alpha_1^2\alpha_2\alpha_4+z_{124,+}z_{3,-}\alpha_1^2\alpha_7\alpha_8+z_{14,+}z_{23,-}\alpha_1\alpha_3\alpha_5\alpha_7+z_{12,+}z_{34,-}\alpha_1\alpha_3\alpha_6\alpha_8\\
        &+z_{1,+}z_{234,-}\alpha_3^2\alpha_5\alpha_6+z_{1234,-}\alpha_2\alpha_3^2\alpha_4.
    \end{align*}
    From this, we deduce
    \begin{align*}
        \alpha_1\alpha_2\alpha_3\alpha_4\gamma=\alpha_1^2\alpha_2\alpha_4+z_3\alpha_1^2+(z_2z_3+z_3z_4)\alpha_1\alpha_3+z_2z_3z_4\alpha_3^2+z_1z_2z_3z_4\alpha_2\alpha_3^2\alpha_4
    \end{align*} 
    by specializing $\alpha_5=\alpha_6=\alpha_7=\alpha_8=1$, $z_{i,+}=1$, and replacing $z_{i,-}$ with $z_{i}$ for $i=1,2,3,4$.
    The result is the same as the Laurent expansion of $\gamma$ in the cluster algebra of $\Sigma$ with principal coefficients in \cite[Example 3.24]{MW13}.\footnote{In our paper, the principal coefficients should correspond to the specialization $z_{\bullet,-}=1$, while here we specialize $z_{\bullet,+}=1$. This is due to the difference of convention: our exchange matrix is the transpose of that used in \cite{MW13}.}
\end{ex}

We can also easily calculate arcs with self-crossings in $\Sigma$ using a deformation of the recurrence relation for the Chebyshev polynomials. 
\begin{ex}[Laurent expansions of arcs with self-crossings in an annulus]\label{ex:MWgeneral}
    Let $(\Sigma,\bW)$ and $\gamma$ be in \cref{ex:MW}.
    For any positive integer $n$, we define $T_n(\gamma;\bW) \in \mathscr{S}_{\Sigma,\bW}$ as follows:
    \begin{align*}
    T_n(\gamma;\bW)=
    
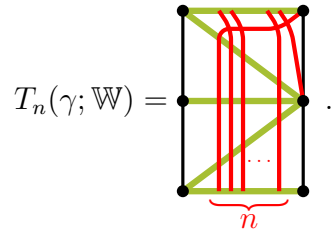
\tikz[baseline=-.6ex, scale=.08]{
        \coordinate (L1) at (-10,-15);
        \coordinate (L2) at (-10,0);
        \coordinate (L3) at (-10,15);
        \coordinate (R1) at (10,-15);
        \coordinate (R2) at (10,0);
        \coordinate (R3) at (10,15);
        \coordinate (C1) at ($(L1)!.5!(R1)$);
        \coordinate (C2) at ($(L2)!.5!(R2)$);
        \coordinate (C3) at ($(L3)!.5!(R3)$);
        \coordinate (C12) at ($(C1)!.5!(C2)$);
        \coordinate (C23) at ($(C2)!.5!(C3)$);
        \draw[very thick] (L1) -- (L3) --(R3) -- (R1) -- cycle;
        \draw[wline] (L1) -- (R1);
        \draw[wline] (L2) -- (R2);
        \draw[wline] (L3) -- (R3);
        \draw[wline] (L1) -- (R2);
        \draw[wline] (L3) -- (R2);
        \draw[webline, rounded corners] ($(L1)!.3!(R1)$) -- ($($(L3)!.3!(R3)$)+(0,-3)$) -- ($($(L3)!.8!(R3)$)+(0,-3)$) -- (R3);
        \draw[webline, rounded corners] ($(L1)!.4!(R1)$) -- ($($(L3)!.4!(R3)$)+(0,-3)$) -- ($(L3)!.3!(R3)$);
        \draw[webline, rounded corners] ($(L1)!.5!(R1)$) -- ($($(L3)!.5!(R3)$)+(0,-3)$) -- ($(L3)!.4!(R3)$);
        \draw[webline, rounded corners] ($(L1)!.8!(R1)$) -- ($($(L3)!.8!(R3)$)+(0,-3)$) -- ($(L3)!.7!(R3)$);
        \draw[webline, rounded corners] (R2) -- ($($(L3)!.9!(R3)$)+(0,-3)$) -- ($(L3)!.8!(R3)$);
        \fill[black] (L1) circle[radius=1];
        \fill[black] (L2) circle[radius=1];
        \fill[black] (L3) circle[radius=1];
        \fill[black] (R1) circle[radius=1];
        \fill[black] (R2) circle[radius=1];
        \fill[black] (R3) circle[radius=1];
        \node at ($($(C1)!.3!(R1)$)+(0,5)$) [red, scale=.7]{$\dots$};
        \node at ($(C1)+(1,-2)$) [red, xscale=2.5, rotate=90] {$\left\{\right.$};
        \node at ($(C1)+(1,-5)$) [red] {$n$};
    }\ .
    \end{align*}
    Then it is easy to see that $T_n(\gamma;\bW)=\gamma T_{n-1}(\gamma;\bW)-a_1a_2a_3a_4T_{n-2}(\gamma;\bW)$ by skein relations. Indeed,
    \begin{align*}
        T_n(\gamma;\bW)=
        \tikz[baseline=-.6ex, scale=.08]{
            \coordinate (L1) at (-10,-15);
            \coordinate (L2) at (-10,0);
            \coordinate (L3) at (-10,15);
            \coordinate (R1) at (10,-15);
            \coordinate (R2) at (10,0);
            \coordinate (R3) at (10,15);
            \coordinate (C1) at ($(L1)!.5!(R1)$);
            \coordinate (C2) at ($(L2)!.5!(R2)$);
            \coordinate (C3) at ($(L3)!.5!(R3)$);
            \coordinate (C12) at ($(C1)!.5!(C2)$);
            \coordinate (C23) at ($(C2)!.5!(C3)$);
            \draw[very thick] (L1) -- (L3) --(R3) -- (R1) -- cycle;
            \draw[wline] (L1) -- (R1);
            \draw[wline] (L2) -- (R2);
            \draw[wline] (L3) -- (R3);
            \draw[wline] (L1) -- (R2);
            \draw[wline] (L3) -- (R2);
            \draw[webline, rounded corners] ($(L1)!.3!(R1)$) -- ($(L3)!.3!(R3)$);
            \draw[webline, rounded corners] ($(L1)!.4!(R1)$) -- ($($(L3)!.4!(R3)$)+(0,-3)$)     -- ($($(L3)!.8!(R3)$)+(0,-3)$) -- (R3);
            \draw[webline, rounded corners] ($(L1)!.5!(R1)$) -- ($($(L3)!.5!(R3)$)+(0,-3)$) -   - ($(L3)!.4!(R3)$);
            \draw[webline, rounded corners] ($(L1)!.8!(R1)$) -- ($($(L3)!.8!(R3)$)+(0,-3)$) -   - ($(L3)!.7!(R3)$);
            \draw[webline, rounded corners] (R2) -- ($($(L3)!.9!(R3)$)+(0,-3)$) -- ($(L3)!.8!   (R3)$);
            \fill[black] (L1) circle[radius=1];
            \fill[black] (L2) circle[radius=1];
            \fill[black] (L3) circle[radius=1];
            \fill[black] (R1) circle[radius=1];
            \fill[black] (R2) circle[radius=1];
            \fill[black] (R3) circle[radius=1];
            \node at ($($(C1)!.3!(R1)$)+(0,5)$) [red, scale=.7]{$\dots$};
            \node at ($(C1)+(1,-2)$) [red, xscale=2.5, rotate=90] {$\left\{\right.$};
            \node at ($(C1)+(1,-5)$) [red] {$n$};
        }
        +
        \tikz[baseline=-.6ex, scale=.08]{
            \coordinate (L1) at (-10,-15);
            \coordinate (L2) at (-10,0);
            \coordinate (L3) at (-10,15);
            \coordinate (R1) at (10,-15);
            \coordinate (R2) at (10,0);
            \coordinate (R3) at (10,15);
            \coordinate (C1) at ($(L1)!.5!(R1)$);
            \coordinate (C2) at ($(L2)!.5!(R2)$);
            \coordinate (C3) at ($(L3)!.5!(R3)$);
            \coordinate (C12) at ($(C1)!.5!(C2)$);
            \coordinate (C23) at ($(C2)!.5!(C3)$);
            \draw[very thick] (L1) -- (L3) --(R3) -- (R1) -- cycle;
            \draw[wline] (L1) -- (R1);
            \draw[wline] (L2) -- (R2);
            \draw[wline] (L3) -- (R3);
            \draw[wline] (L1) -- (R2);
            \draw[wline] (L3) -- (R2);
            \draw[webline, rounded corners] ($(L1)!.3!(R1)$) -- ($($(L3)!.3!(R3)$)+(0,-3)$)     -- ($($(L3)!.4!(R3)$)+(0,-3)$) -- ($(L1)!.4!(R1)$);
            \draw[webline, rounded corners] (R3) -- ($($(L3)!.4!(R3)$)+(0,-3)$) -- ($(L3)!.3!   (R3)$);
            \draw[webline, rounded corners] ($(L1)!.5!(R1)$) -- ($($(L3)!.5!(R3)$)+(0,-3)$) -   - ($(L3)!.4!(R3)$);
            \draw[webline, rounded corners] ($(L1)!.8!(R1)$) -- ($($(L3)!.8!(R3)$)+(0,-3)$) -   - ($(L3)!.7!(R3)$);
            \draw[webline, rounded corners] (R2) -- ($($(L3)!.9!(R3)$)+(0,-3)$) -- ($(L3)!.8!   (R3)$);
            \fill[black] (L1) circle[radius=1];
            \fill[black] (L2) circle[radius=1];
            \fill[black] (L3) circle[radius=1];
            \fill[black] (R1) circle[radius=1];
            \fill[black] (R2) circle[radius=1];
            \fill[black] (R3) circle[radius=1];
            \node at ($($(C1)!.3!(R1)$)+(0,5)$) [red, scale=.7]{$\dots$};
            \node at ($(C1)+(1,-2)$) [red, xscale=2.5, rotate=90] {$\left\{\right.$};
            \node at ($(C1)+(1,-5)$) [red] {$n$};
        }
        =\gamma\ 
        \tikz[baseline=-.6ex, scale=.08]{
            \coordinate (L1) at (-10,-15);
            \coordinate (L2) at (-10,0);
            \coordinate (L3) at (-10,15);
            \coordinate (R1) at (10,-15);
            \coordinate (R2) at (10,0);
            \coordinate (R3) at (10,15);
            \coordinate (C1) at ($(L1)!.5!(R1)$);
            \coordinate (C2) at ($(L2)!.5!(R2)$);
            \coordinate (C3) at ($(L3)!.5!(R3)$);
            \coordinate (C12) at ($(C1)!.5!(C2)$);
            \coordinate (C23) at ($(C2)!.5!(C3)$);
        \draw[very thick] (L1) -- (L3) --(R3) - - (R1) -- cycle;
            \draw[wline] (L1) -- (R1);
            \draw[wline] (L2) -- (R2);
            \draw[wline] (L3) -- (R3);
            \draw[wline] (L1) -- (R2);
            \draw[wline] (L3) -- (R2);
            \draw[webline, rounded corners] ($(L1)!.4!(R1)$) -- ($($(L3)!.4!(R3)$)+(0,-3)$)     -- ($($(L3)!.8!(R3)$)+(0,-3)$) -- (R3);
            \draw[webline, rounded corners] ($(L1)!.5!(R1)$) -- ($($(L3)!.5!(R3)$)+(0,-3)$) -   - ($(L3)!.4!(R3)$);
            \draw[webline, rounded corners] ($(L1)!.8!(R1)$) -- ($($(L3)!.8!(R3)$)+(0,-3)$) -   - ($(L3)!.7!(R3)$);
            \draw[webline, rounded corners] (R2) -- ($($(L3)!.9!(R3)$)+(0,-3)$) -- ($(L3)!.8!   (R3)$);
            \fill[black] (L1) circle[radius=1];
            \fill[black] (L2) circle[radius=1];
            \fill[black] (L3) circle[radius=1];
            \fill[black] (R1) circle[radius=1];
            \fill[black] (R2) circle[radius=1];
            \fill[black] (R3) circle[radius=1];
            \node at ($($(C1)!.3!(R1)$)+(0,5)$) [red, scale=.7]{$\dots$};
            \node at ($(C1)+(2,-2)$) [red, xscale=2.1, rotate=90] {$\left\{\right.$};
            \node at ($(C1)+(2,-5)$) [red] {$n-1$};
        }
        +a_1a_2a_3a_4\ 
        \tikz[baseline=-.6ex, scale=.08]{
            \coordinate (L1) at (-10,-15);
            \coordinate (L2) at (-10,0);
            \coordinate (L3) at (-10,15);
            \coordinate (R1) at (10,-15);
            \coordinate (R2) at (10,0);
            \coordinate (R3) at (10,15);
            \coordinate (C1) at ($(L1)!.5!(R1)$);
            \coordinate (C2) at ($(L2)!.5!(R2)$);
            \coordinate (C3) at ($(L3)!.5!(R3)$);
            \coordinate (C12) at ($(C1)!.5!(C2)$);
            \coordinate (C23) at ($(C2)!.5!(C3)$);
            \draw[very thick] (L1) -- (L3) --(R3) -- (R1) -- cycle;
            \draw[wline] (L1) -- (R1);
            \draw[wline] (L2) -- (R2);
            \draw[wline] (L3) -- (R3);
            \draw[wline] (L1) -- (R2);
            \draw[wline] (L3) -- (R2);
            \draw[webline, rounded corners] (R3) -- ($($(L3)!.3!(R3)$)+(0,-4)$) --  ($($(L3)!.3!(R3)$)+(0,-1)$) -- ($($(L3)!.5!(R3)$)+(0,-1)$) -- ($(L1)!.5!(R1)$);
            \draw[webline, rounded corners] ($(L1)!.8!(R1)$) -- ($($(L3)!.8!(R3)$)+(0,-3)$) -   - ($(L3)!.7!(R3)$);
            \draw[webline, rounded corners] (R2) -- ($($(L3)!.9!(R3)$)+(0,-3)$) -- ($(L3)!.8!   (R3)$);
            \fill[black] (L1) circle[radius=1];
            \fill[black] (L2) circle[radius=1];
            \fill[black] (L3) circle[radius=1];
            \fill[black] (R1) circle[radius=1];
            \fill[black] (R2) circle[radius=1];
            \fill[black] (R3) circle[radius=1];
            \node at ($($(C1)!.3!(R1)$)+(0,5)$) [red, scale=.7]{$\dots$};
            \node at ($(C1)+(3,-2)$) [red, xscale=1.8, rotate=90] {$\left\{\right.$};
            \node at ($(C1)+(3,-5)$) [red] {$n-2$};
        }\ ,
    \end{align*}
    and the kink can be removed with the multiplicative factor $-1$.
\end{ex}

\subsection{Basis of the skein algebra of a walled surface}
In the case of the skein algebra $\SK{\Sigma}=\SK{\Sigma, \emptyset}$ with no walls, it is known that the set $\mathsf{SMulti}(\Sigma)$ of homotopy classes of simple multicurves gives a basis \cite[Lemma~4.1]{Muller16}.
Here, a simple multicurve is a collection of mutually disjoint ideal arcs and simple loops without a component bounding a disk or a monogon. 
For $\Sigma$ with a non-trivial wall system, we cannot consider $\mathsf{SMulti}(\Sigma)$ as a basis of $\SK{\Sigma,\bW}$, since a homotopy may change the $\bW$-isotopy class $[\gamma]_{\bW}$ of a simple multicurve $\gamma$ on the underlying surface $\Sigma$.
We consider a $\bW$-transverse and \emph{$\bW$-minimal} lift $\gamma_{\bW}$ of $\gamma\in\mathsf{SMulti}(\Sigma)$ in $(\Sigma,\bW)$.
Moreover, we describe local moves (the bottom line of \cref{fig:elementary-moves}) relating any two lifts $\tilde{\gamma}_{1,\bW}$ and $\tilde{\gamma}_{2,\bW}$ of $\gamma$.
It is easy to see that $\tilde{\gamma}_{1,\bW}$ and $\tilde{\gamma}_{1,\bW}$ represent the same element in $\SK{\Sigma,\bW}$.
The set of $\bW$-transverse and $\bW$-minimal lifts modulo the local moves gives a $\bZ_{A,\bW}$-basis of the skein algebra $\SK{\Sigma,\bW}$.
We summarize the assertion below. The proof will be provided in \cref{sec:minimal}.

\begin{figure}[ht]
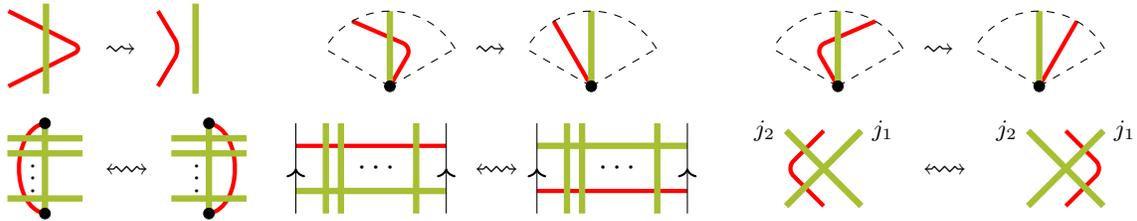

    \begin{align*}
        \ \tikz[baseline=-.6ex, scale=.1, yshift=-5cm]{
            \draw[webline, rounded corners] (-5,0) -- (5,5) -- (-5,10);
            \draw[wline] (0,-1) -- (0,11);
        }\ 
        &\rightsquigarrow
        \ \tikz[baseline=-.6ex, scale=.1, yshift=-5cm]{
                \draw[webline, rounded corners] (-5,0) -- (-2,5) -- (-5,10);
                \draw[wline] (0,-1) -- (0,11);
        }\ &
        \ \tikz[baseline=-.6ex, scale=.1, yshift=-5cm]{
            \coordinate (A) at (-10,0);
            \coordinate (B) at (10,0);
            \coordinate (P) at (0,0);
            \draw[webline, rounded corners] (P) -- (60:6) -- (120:10);
            \draw[wline] (P) -- (90:10);
            \draw[dashed] (30:10) arc (30:150:10cm);
            \draw[dashed] (P) -- (30:10);
            \draw[dashed] (P) -- (150:10);
            \draw[fill=black] (P) circle [radius=20pt];
        }\ 
        &\rightsquigarrow
        \ \tikz[baseline=-.6ex, scale=.1, yshift=-5cm]{
            \coordinate (A) at (-10,0);
            \coordinate (B) at (10,0);
            \coordinate (P) at (0,0);
            \draw[webline, rounded corners] (P) -- (120:10);
            \draw[wline] (P) -- (90:10);
            \draw[dashed] (30:10) arc (30:150:10cm);
            \draw[dashed] (P) -- (30:10);
            \draw[dashed] (P) -- (150:10);
            \draw[fill=black] (P) circle [radius=20pt];
        }\ &
        \ \tikz[baseline=-.6ex, scale=.1, yshift=-5cm]{
            \coordinate (A) at (-10,0);
            \coordinate (B) at (10,0);
            \coordinate (P) at (0,0);
            \draw[webline, rounded corners] (P) -- (120:6) -- (60:10);
            \draw[wline] (P) -- (90:10);
            \draw[dashed] (30:10) arc (30:150:10cm);
            \draw[dashed] (P) -- (30:10);
            \draw[dashed] (P) -- (150:10);
            \draw[fill=black] (P) circle [radius=20pt];
        }\ 
        &\rightsquigarrow
        \ \tikz[baseline=-.6ex, scale=.1, yshift=-5cm]{
            \coordinate (A) at (-10,0);
            \coordinate (B) at (10,0);
            \coordinate (P) at (0,0);
            \draw[webline, rounded corners] (P) -- (60:10);
            \draw[wline] (P) -- (90:10);
            \draw[dashed] (30:10) arc (30:150:10cm);
            \draw[dashed] (P) -- (30:10);
            \draw[dashed] (P) -- (150:10);
            \draw[fill=black] (P) circle [radius=20pt];
        }\ \\
        \ \tikz[baseline=-.6ex, scale=.1, yshift=-6cm]{
            \draw[webline] (0,0) to[bend left=80] (90:12);
            \draw[wline] (0,0) -- (90:12);
            \begin{scope}
                \draw[wline] (-5,10) -- (5,10);
                \draw[wline] (-5,8) -- (5,8);
                \draw[wline] (-5,2) -- (5,2);
                \node at (-1.5,5)[rotate=90]{\small $\cdots$};
            \end{scope}
            \draw[fill=black] (0,0) circle [radius=20pt];
            \draw[fill=black] (90:12) circle [radius=20pt];
        }\ 
        &\leftrightsquigarrow
        \ \tikz[baseline=-.6ex, scale=.1, yshift=-6cm]{
            \draw[webline] (0,0) to[bend right=80] (90:12);
            \draw[wline] (0,0) -- (90:12);
            \begin{scope}
                \draw[wline] (-5,10) -- (5,10);
                \draw[wline] (-5,8) -- (5,8);
                \draw[wline] (-5,2) -- (5,2);
                \node at (-1.5,5)[rotate=90]{\small $\cdots$};
            \end{scope}
            \draw[fill=black] (0,0) circle [radius=20pt];
            \draw[fill=black] (90:12) circle [radius=20pt];
        }\ &
        \tikz[baseline=-.6ex,scale=.1]{
            \draw[webline] (-10,3) -- (10,3);
            \draw[wline] (-10,-3) -- (10,-3);
            \begin{scope}
                \draw[wline] (-6,6) -- (-6,-6);
                \draw[wline] (-4,6) -- (-4,-6);
                \draw[wline] (6,6) -- (6,-6);
                \node at (1,0) {$\cdots$};
            \end{scope}
            \draw[->-={.5}{black}] (-10,-6) -- (-10,6);
            \draw[->-={.5}{black}] (10,-6) -- (10,6);
        }\ 
        &\leftrightsquigarrow
        \tikz[baseline=-.6ex,scale=.1]{
            \draw[wline] (-10,3) -- (10,3);
            \draw[webline] (-10,-3) -- (10,-3);
            \begin{scope}
                \draw[wline] (-6,6) -- (-6,-6);
                \draw[wline] (-4,6) -- (-4,-6);
                \draw[wline] (6,6) -- (6,-6);
                \node at (1,0) {$\cdots$};
            \end{scope}
            \draw[->-={.5}{black}] (-10,-6) -- (-10,6);
            \draw[->-={.5}{black}] (10,-6) -- (10,6);
        }\ &
        \ \tikz[baseline=-.6ex, scale=.1, yshift=-5cm]{
            \draw[webline, rounded corners] (0,0) -- (-5,5) -- (0,10);
            \draw[wline] (5,0) -- (-5,10);
            \draw[wline] (-5,0) -- (5,10);
            \node at (5,10) [right]{\scriptsize $j_1$};
            \node at (-5,10) [left]{\scriptsize $j_2$};
        }\ 
        &\leftrightsquigarrow
        \ \tikz[baseline=-.6ex, scale=.1, yshift=-5cm]{
            \draw[webline, rounded corners] (0,0) -- (5,5) -- (0,10);
            \draw[wline] (5,0) -- (-5,10);
            \draw[wline] (-5,0) -- (5,10);
            \node at (5,10) [right]{\scriptsize $j_1$};
            \node at (-5,10) [left]{\scriptsize $j_2$};
        }\ 
    \end{align*}
    \caption{The top line: \emph{bigon reduction moves}. The bottom line: a \emph{flip move} on arc and loop components, and a \emph{double-point jumping}.}
    \label{fig:elementary-moves}
\end{figure}

\begin{dfn}\label{def:W-minimal}
    Let $\bW=(\sfC,J,\ell)$ be a wall system on $\Sigma$.
    \begin{itemize}
        \item A $\bW$-transverse diagram $D$ in a walled surface $(\Sigma, \bW)$ is said to be \emph{$\bW$-minimal} if it realizes the minimum of the number of intersection points of $D\cup \bigcup\mathsf{C}$ within its equivalence class.
        \item a $\bW$-minimal diagram is a \emph{$\bW$-minimal simple multicurve} if $D_{\bW}$ is a simple multicurve in the underlying surface $\Sigma$.
        \item a $\bW$-minimal simple multicurve $D_{\bW}$ related to $D\in\mathsf{SMulti}(\Sigma)$ if $D_{\bW}$ is homotopic to $D$ in the underlying surface $\Sigma$.
    \end{itemize}
\end{dfn}

\begin{dfn}\label{def:taut}
A wall system $\bW$ on $\Sigma$ is said to be \emph{taut} if each curve in $\mathsf{C}$ is simple, does not bound any disk or monogon, and any pair of curves does not bound a bigon.
\end{dfn}

We can prove the following assertion.
\begin{thm}[{\cref{cor:minimal-multicurve}}]\label{thm:minimal-multicurve}
    Let $\bW$ be a taut wall system on $\Sigma$.
    For any simple multicurve $D$ in $\Sigma$, the $\bW$-minimal simple multicurve $D_{\bW}$ related to $D$ is unique up to homotopy fixing $\bW$ setwisely, flip moves and double-point jumping in \cref{fig:elementary-moves}.
\end{thm}

\begin{rem}\label{rem:invariance-flip}
    The diagrams before and after the flip moves in \cref{fig:elementary-moves} represent the same element in $\SK{\Sigma, \bW}$.
    One can verify this fact by using (the inverse of) a bigon reduction move and double-point jumpings. 
    The invariance under the double-point jumping is clear from the skein relation~\labelcref{rel:wall-R3}.
\end{rem}

\begin{dfn}\label{def:basis-web}
    Fix a taut wall system $\bW$ on $\Sigma$. A \emph{basis web in $\SK{\Sigma, \bW}$} is an element of $\SK{\Sigma, \bW}$ represented by a $\bW$-minimal simple multicurve on $(\Sigma, \bW)$.
    We denote the set of basis webs in $\SK{\Sigma, \bW}$ by $\Bweb{\Sigma, \bW}$.
\end{dfn}

\begin{thm}\label{thm:basis-web}
     For a taut wall system $\bW$ on $\Sigma$, the set $\Bweb{\Sigma,\bW}$ gives a basis of $\SK{\Sigma,\bW}$ as a $\bZ_{A,\bW}$-module.
\end{thm}
\begin{proof}
    One can show it by combining results in \cref{thm:minimal-multicurve}, \cref{rem:invariance-flip}, and the confluence theory in \cite{SikoraWestbury07} with reduction sequences consisting of reduction of crossings \labelcref{rel:Kauffman}, reduction of elliptic faces \labelcref{rel:trivial-loop,rel:monogon}, and reduction of bigons \labelcref{rel:wall-pass-int,rel:wall-pass-ext}.
\end{proof}

In the forthcoming paper \cite{IKY2}, we will prove the following statement by establishing a $\bZ_{A,\bW}$-algebra embedding of $\SK{\Sigma, \bW}$ into the \emph{stated skein algebra} of $(\Sigma,\bW)$, and the injectivity of the splitting homomorphisms of the latter algebra.

\begin{thm}[{\cite{IKY2}}]\label{thm:ore}
$\SK{\Sigma,\bW}$ is an Ore domain for a taut wall system $\bW$ on $\Sigma$. 
\end{thm}
For an ideal triangulation $\tri$ and $\alpha \in \tri$, let $[\alpha_\bW] \in \SK{\Sigma, \bW}$ be the element represented by a $\bW$-minimal diagram $\alpha_\bW$ related to $\alpha$. Observe that the elements
$[\alpha_\bW] \in \SK{\Sigma, \bW}$ are $A$-commuting with each other. Let $\bT_\bW^\tri$ be the quantum torus over $\bZ_{A,\bW}$ generated by $[\alpha_\bW]$ for $\alpha \in \tri$. 
\begin{thm}\label{thm:quantum-torus}
     For a taut wall system $\bW$ on $\Sigma$, $\SK{\Sigma, \bW}$ is embedded into the skew-field $\Frac \bT_\bW^\tri$ of fractions for any ideal triangulation $\tri$ of $\Sigma$. In particular, we have $\Frac\SK{\Sigma, \bW} \cong \Frac \bT_\bW^\tri$.
\end{thm}
\begin{proof}
    Given an ideal triangulation, the localization $\SK{\Sigma, \bW}[[\alpha_\bW]^{-1} \mid \alpha \in \tri]$ is isomorphic to the quantum torus $\bT_\bW^\tri$. 
    Indeed, any $\bW$-minimal curve $D$ can be expanded into a polynomial in $[\alpha_\bW]$'s by multiplying 
    $\prod_{\alpha\in\tri}[\alpha_\bW]^{n_\alpha}$,  
    where $n_\alpha$ is the geometric intersection number between $D$ and $\alpha$.
\end{proof}

\begin{dfn}[boundary-localization]\label{def:boundary_localization}
Let $\mathrm{mon}(\partial)$ denote the multiplicatively closed set generated by the arcs parallel to boundary intervals and $A^{\pm 1/2}$. Then the Ore localization of $\SK{\Sigma, \bW}$ along $\mathrm{mon}(\partial)$ is called the \emph{boundary-localized skein algebra} of $(\Sigma, \bW)$, and denoted by $\SK{\Sigma, \bW}[\partial^{-1}]$.
\end{dfn}

\subsection{Skein relation in the walled square}
\label{subsec:skein_formula}

Let $\bW = (\mathsf{C}, J, \ell)$ be a taut wall system on a surface $\Sigma$.
We give a formula of a skein relation for two distinct diagonals $\kappa$ and $\kappa'$ of any ideal quadrilateral $Q$ in $\Sigma$.
By applying wall-passing relations if necessary, we may assume that these are $\bW$-minimal.
Let us color the four complementary regions of the diagram $\kappa \cup \kappa'$ in $Q$ by $A$ and $B$ as shown in the left of \cref{fig:AB-regions1}, and call them $A$-regions and $B$-regions. 

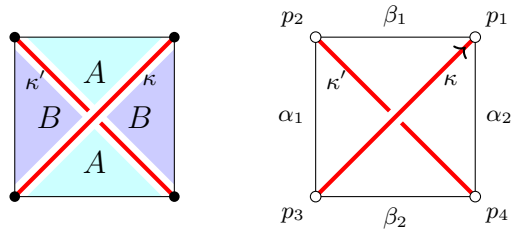
\begin{figure}[ht]
    \centering
    \begin{tikzpicture}[scale=.1]
    \begin{scope}[xshift=-20cm]
        \coordinate (p1) at (45:15);
        \coordinate (p2) at (135:15);
        \coordinate (p3) at (225:15);
        \coordinate (p4) at (315:15);
        \fill[cyan!20] (0,0)--(45:15)--(135:15)--cycle; 
        \fill[cyan!20,rotate=180] (0,0)--(45:15)--(135:15)--cycle; 
        \fill[blue!20,rotate=90] (0,0)--(45:15)--(135:15)--cycle; 
        \fill[blue!20,rotate=-90] (0,0)--(45:15)--(135:15)--cycle; 
        \draw[overarc] (-45:15) -- (135:15);
        \draw[overarc] (45:15) -- (-135:15);
        \draw (45:15)--(135:15)--(225:15)--(315:15)--cycle;
        \node at (90:6) {$A$};
        \node at (-90:6) {$A$};
        \node at (0:6) {$B$};
        \node at (180:6) {$B$};
        \node at (45:7) [right]{\scriptsize $\kappa$};
        \node at (135:7) [left]{\scriptsize $\kappa'$};
        \fill (p1) circle [radius=20pt];
        \fill (p2) circle [radius=20pt];
        \fill (p3) circle [radius=20pt];
        \fill (p4) circle [radius=20pt];
    \end{scope}
    \begin{scope}[xshift=20cm]
        \coordinate (p1) at (45:15);
        \coordinate (p2) at (135:15);
        \coordinate (p3) at (225:15);
        \coordinate (p4) at (315:15);
        \draw[overarc] (-45:15) -- (135:15);
        \draw[overarc, -<-={.1}{black}] (45:15) -- (-135:15);
        \draw (45:15)--(135:15)--(225:15)--(315:15)--cycle;
        \node at (45:7) [right]{\scriptsize $\kappa$};
        \node at (135:7) [left]{\scriptsize $\kappa'$};
        \draw[fill=white] (p1) circle [radius=20pt];
        \draw[fill=white] (p2) circle [radius=20pt];
        \draw[fill=white] (p3) circle [radius=20pt];
        \draw[fill=white] (p4) circle [radius=20pt];
        \node at (p1) [above right]{\scriptsize $p_1$};
        \node at (p2) [above left]{\scriptsize $p_2$};
        \node at (p3) [below left]{\scriptsize $p_3$};
        \node at (p4) [below right]{\scriptsize $p_4$};
        \node at ($(p1)!.5!(p2)$) [above]{\scriptsize $\beta_1$};
        \node at ($(p3)!.5!(p4)$) [below]{\scriptsize $\beta_2$};
        \node at ($(p2)!.5!(p3)$) [left]{\scriptsize $\alpha_1$};
        \node at ($(p4)!.5!(p1)$) [right]{\scriptsize $\alpha_2$};
    \end{scope}
    \end{tikzpicture}        
    \caption{Left: the $A$- and $B$-regions with respect to arcs $\kappa$ and $\kappa'$. Right: labels of vertices and interior of edges of $Q$ with respect to the orientation of $\kappa$.}
    \label{fig:AB-regions1}
\end{figure}

\begin{figure}[ht]
    \centering
    \begin{tikzpicture}[scale=.1]
    \begin{scope}[xshift=-20cm]
        \coordinate (p1) at (45:15);
        \coordinate (p2) at (135:15);
        \coordinate (p3) at (225:15);
        \coordinate (p4) at (315:15);
        \fill[cyan!20] (0,0)--(45:15)--(135:15)--cycle; 
        \fill[cyan!20, rotate=180] (0,0)--(45:15)--(135:15)--cycle; 
        \fill[blue!20, rotate=90] (0,0)--(45:15)--(135:15)--cycle; 
        \fill[blue!20, rotate=-90] (0,0)--(45:15)--(135:15)--cycle; 
        \draw (45:15)--(135:15)--(225:15)--(315:15)--cycle;
        \draw[wline] (p1)to[bend left=10](p3);
        \draw[wline] ($(p1)!.4!(p2)$)--(p3);
        \draw[wline] ($(p1)!.6!(p2)$)--(p3);
        \draw[wline] ($(p1)!.8!(p2)$)--(p3);
        \draw[wline] ($(p3)!.5!(p4)$)--(p1);
        \draw[wline] ($(p3)!.8!(p4)$)--(p1);
        \fill (p1) circle [radius=20pt];
        \fill (p2) circle [radius=20pt];
        \fill (p3) circle [radius=20pt];
        \fill (p4) circle [radius=20pt];
        \node at (p1) [above right]{\scriptsize $p_1$};
        \node at (p2) [above left]{\scriptsize $p_2$};
        \node at (p3) [below left]{\scriptsize $p_3$};
        \node at (p4) [below right]{\scriptsize $p_4$};
        \node at ($(p1)!.5!(p2)$) [above]{\scriptsize $\beta_1$};
        \node at ($(p3)!.5!(p4)$) [below]{\scriptsize $\beta_2$};
        \node at ($(p2)!.5!(p3)$) [left]{\scriptsize $\alpha_1$};
        \node at ($(p4)!.5!(p1)$) [right]{\scriptsize $\alpha_2$};
    \end{scope}
    \begin{scope}[xshift=20cm]
        \coordinate (p1) at (45:15);
        \coordinate (p2) at (135:15);
        \coordinate (p3) at (225:15);
        \coordinate (p4) at (315:15);
        \fill[cyan!20] (0,0)--(45:15)--(135:15)--cycle; 
        \fill[cyan!20, rotate=180] (0,0)--(45:15)--(135:15)--cycle; 
        \fill[blue!20, rotate=90] (0,0)--(45:15)--(135:15)--cycle; 
        \fill[blue!20,rotate=-90] (0,0)--(45:15)--(135:15)--cycle; 
        \draw (45:15)--(135:15)--(225:15)--(315:15)--cycle;
        \draw[wline] ($(p1)!.2!(p2)$)to[out=south,in=north]($(p4)!.6!(p3)$);
        \draw[wline] ($(p1)!.4!(p2)$)to[out=south,in=north]($(p4)!.1!(p3)$);
        \draw[wline] ($(p1)!.6!(p2)$)to[out=south,in=north]($(p4)!.3!(p3)$);
        \draw[wline] ($(p1)!.8!(p2)$)to[out=south,in=north]($(p4)!.8!(p3)$);
        \fill (p1) circle [radius=20pt];
        \fill (p2) circle [radius=20pt];
        \fill (p3) circle [radius=20pt];
        \fill (p4) circle [radius=20pt];
        \node at (p1) [above right]{\scriptsize $p_1$};
        \node at (p2) [above left]{\scriptsize $p_2$};
        \node at (p3) [below left]{\scriptsize $p_3$};
        \node at (p4) [below right]{\scriptsize $p_4$};
        \node at ($(p1)!.5!(p2)$) [above]{\scriptsize $\beta_1$};
        \node at ($(p3)!.5!(p4)$) [below]{\scriptsize $\beta_2$};
        \node at ($(p2)!.5!(p3)$) [left]{\scriptsize $\alpha_1$};
        \node at ($(p4)!.5!(p1)$) [right]{\scriptsize $\alpha_2$};
    \end{scope}
    \end{tikzpicture}        
    \caption{Left: walls in $\bW\cap Q$ corresponding to $\sfC_{A,\kappa}$. Right: walls in $\bW\cap Q$ corresponding to $\sfC_{A,A}$.}
    \label{fig:AB-regions2}
\end{figure}

Let us define some multisets $\sfC_{A,\kappa}$, $\sfC_{A,\kappa'}$, $\sfC_{B,\kappa}$, $\sfC_{B,\kappa'}$, $\sfC_{A, A}$ and $\sfC_{B,B}$ whose elements are in $\sfC$. 
Fix an auxiliary orientation of $\kappa$, and label the vertices of $Q$ as $p_1,p_2,p_3,p_4$ in this counter-clockwise order from the head of $\kappa$.
Denote 
the edge $p_1p_2$ (resp.~$p_3p_4$) of $Q$ 
by $\beta_1$ (resp. $\beta_2$).
Similarly, 
$\alpha_1$ and $\alpha_2$
be 
the edges $p_2p_3$ and $p_4p_1$, respectively. 

Let $\sfC_{A,\kappa}$ (resp. $\sfC_{A,\kappa'}$) be the multiset consisting of the curves $\gamma \in \sfC$ having either of the following properties:
\begin{itemize}
    \item parallel to $\kappa$ (resp. $\kappa'$),
    \item some of the components of $\gamma \cap Q$ connect the interior of $\beta_1$ and $p_3$ (resp. $p_4$), or
    \item some of the components of $\gamma \cap Q$ connect the interior of $\beta_2$ and $p_1$ (resp. $p_2$).
\end{itemize}
The multiplicity is given by the total number of such components. 
The multisets $\sfC_{B,\kappa}$ and $\sfC_{B,\kappa'}$ are defined in a similar way, by replacing $A$ with $B$.
The multiset $\sfC_{A,A}$ (resp.~$\sfC_{B,B}$) consists of the curves $\gamma \in \sfC$ such that some of the components of $\gamma \cap Q$ connect the interiors of $\beta_1$ and $\beta_2$ (resp.~$\alpha_1$ and $\alpha_2$), with the multiplicity given by the number of such components. 
We remark that $\sfC_{A,A}$ and $\sfC_{B,B}$ does not contain walls connecting endpoints of $\kappa$ or $\kappa'$. 
Then, one can easily verify the following formula:
\begin{prop}\label{eq:polygon}
With the notation as above, we have
\begin{align*}
    \kappa\kappa'
    &=A
    \Big(\!\prod_{\gamma\in\sfC_{A,A}}\!\!\!a_{\ell(\gamma)}\Big)
    \Big(\!\prod_{\delta\in\sfC_{A,\kappa}}\!\!\! z_{\ell(\delta),+}\Big)
    \Big(\!\prod_{\eta\in\sfC_{A,\kappa'}}\!\!\! z_{\ell(\eta),-}\Big)
    \alpha_1\alpha_2\nonumber\\
    &\hspace{3cm} + A^{-1}
    \Big(\!\prod_{\gamma\in\sfC_{B,B}}\!\!\! a_{\ell(\gamma)}\Big)
    \Big(\!\prod_{\delta\in\sfC_{B,\kappa'}}\!\!\! z_{\ell(\delta),+}\Big)
    \Big(\!\prod_{\eta\in\sfC_{B,\kappa}}\!\!\! z_{\ell(\eta),-}\Big)
    \beta_1\beta_2.
\end{align*}
\end{prop}

\begin{ex}[Base affine space $\mathrm{SL}_4/N$
]\label{ex:affine_SL4}
    Let us compute the skein algebra corresponding to the cluster algebra associated with the affine base space $\mathrm{SL}_4/N$.
    We use labels of components of the wall and arcs similar to \cite[Example 16.6]{FT18} as follows:
    \begin{center}
        \begin{tikzpicture}[scale=.1, rotate=-60]
            \foreach \i [evaluate=\i as \x using (\i+1)*60]in {0,1,...,5}
            {
                \coordinate (A\i) at (\x:15);
            }
            \foreach \i [evaluate=\i as \j using {mod(\i+2,6)}]in {1,2,4,5}
            {
                \draw[wline] (A\i) -- (A\j);
            }
            \draw[wline] (A0) -- (A3);
            \draw[wline] (A2) -- (A5);
            \foreach \i [evaluate=\i as \j using {mod(\i+1,6)}]in {0,1,...,5}
            {
                \bdryline{(A\i)}{(A\j)}{2cm}
            }
            \foreach \k in {0,1,...,5}
            {
                \draw[fill] (A\k) circle (20pt);
            }
            \node at (120:10) [below=.1cm]{\scriptsize $\texttt{234}$};
            \node at (120:10) [left]{\scriptsize $\texttt{1}$};
            \node at (-120:4) [below]{\scriptsize $\texttt{34}$};
            \node at (-120:9) [below=.2cm]{\scriptsize $\texttt{4}$};
            \node at (0:5) [left]{\scriptsize $\texttt{12}$};
            \node at (0:10) [below left]{\scriptsize $\texttt{123}$};
        \end{tikzpicture}
        , \hspace{1em}
        \begin{tikzpicture}[scale=.1]
            \foreach \i [evaluate=\i as \x using (\i+1)*60]in {0,1,...,5}
            {
                \coordinate (A\i) at (\x:15);
            }
            \foreach \i [evaluate=\i as \j using {mod(\i+2,6)}]in {0,1,...,5}
            {
                \draw[webline] (A\i) -- (A\j);
            }
            \draw[webline] (A0) -- (A3);
            \draw[webline] (A1) -- (A4);
            \draw[webline] (A2) -- (A5);
            \foreach \i [evaluate=\i as \j using {mod(\i+1,6)}]in {0,1,...,5}
            {
                \bdryline{(A\i)}{(A\j)}{2cm}
            }
            \foreach \k in {0,1,...,5}
            {
                \draw[fill] (A\k) circle (20pt);
            }
            \node at (180:4) {\scriptsize $\Omega$};
            \node at (180:11) [above=.1cm]{\scriptsize $\texttt{2}$};
            \node at (180:11) [below=.1cm]{\scriptsize $\texttt{3}$};
            \node at (60:4) {\scriptsize $\texttt{24}$};
            \node at (-60:4) {\scriptsize $\texttt{13}$};
            \node at (-120:11) [above left]{\scriptsize $\texttt{14}$};
            \node at (-120:13) [right=.1cm]{\scriptsize $\texttt{124}$};
            \node at (0:9) [above=.5cm]{\scriptsize $\texttt{23}$};
            \node at (180:2) [above=.9cm]{\scriptsize $\texttt{134}$};
        \end{tikzpicture}
        , \hspace{1em}
        \begin{tikzpicture}[scale=.1]
            \foreach \i [evaluate=\i as \x using (\i+1)*60]in {0,1,...,5}
            {
                \coordinate (A\i) at (\x:15);
            }
            \foreach \i [evaluate=\i as \j using {mod(\i+1,6)}]in {0,1,...,5}
            {
                \draw[webline] (A\i) to[bend left] (A\j);
            }
            \foreach \i [evaluate=\i as \j using {mod(\i+1,6)}]in {0,1,...,5}
            {
                \bdryline{(A\i)}{(A\j)}{2cm}
            }
            \foreach \k in {0,1,...,5}
            {
                \draw[fill] (A\k) circle (20pt);
            }
            \node at (30:8) {\scriptsize $\delta_1$};
            \node at (90:8) {\scriptsize $\delta_2$};
            \node at (150:8) {\scriptsize $\delta_3$};
            \node at (210:8) {\scriptsize $\delta_4$};
            \node at (270:8) {\scriptsize $\delta_5$};
            \node at (330:8) {\scriptsize $\delta_6$};
        \end{tikzpicture}.
    \end{center}
    We write $\gamma_n$ for the wall with the label $n$ and $\alpha_\bullet$ for non-boundary ideal arcs.
    For example, let us consider the product $\alpha_{23}\alpha_{\Omega}$.
    Then, 
    $\sfC_{A,\alpha_{23}}=\{\gamma_{234},\gamma_{12}\}$,
    $\sfC_{A,\alpha_{\Omega}}=\{\gamma_{34}\}$,
    $\sfC_{B,\alpha_{23}}=\{\gamma_{234}\}$,
    $\sfC_{A,\alpha_{\Omega}}=\{\gamma_{123},\gamma_{34}\}$, $\sfC_{A,A}=\sfC_{B,B}=\emptyset$ (the diagram below is useful for finding the elements in $\sfC_{A,\alpha_{\bullet}}$ and $\sfC_{B,\alpha_{\bullet}}$).
    By \cref{eq:polygon}, we obtain 
    \begin{align}\label{eq:SL4_exchange}
    \alpha_{23}\alpha_{\Omega}=A z_{234,+}z_{12,+}z_{34,-}\alpha_{3}\delta_{1} + A^{-1} z_{123,+}z_{34,+}z_{234,-}\alpha_2\delta_{6}.
    \end{align}
    
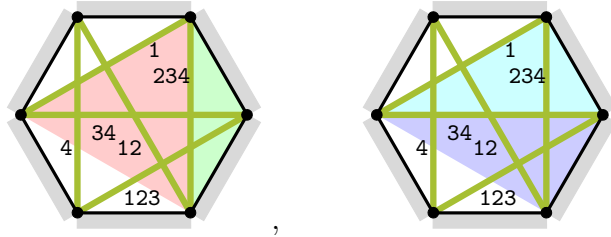
\begin{center}
        \begin{tikzpicture}[scale=.1, rotate=-60]
            \foreach \i [evaluate=\i as \x using (\i+1)*60]in {0,1,...,5}
            {
                \coordinate (A\i) at (\x:15);
            }
        \fill[green!20] (A5) -- (A0) -- (A1) -- cycle;
        \fill[red!20] (A1) -- (A3) -- (A5) -- cycle;
            \foreach \i [evaluate=\i as \j using {mod(\i+2,6)}]in {1,2,4,5}
            {
                \draw[wline] (A\i) -- (A\j);
            }
            \draw[wline] (A0) -- (A3);
            \draw[wline] (A2) -- (A5);
            \foreach \i [evaluate=\i as \j using {mod(\i+1,6)}]in {0,1,...,5}
            {
                \bdryline{(A\i)}{(A\j)}{2cm}
            }
            \foreach \k in {0,1,...,5}
            {
                \draw[fill] (A\k) circle (20pt);
            }
            \node at (120:10) [below=.1cm]{\scriptsize $\texttt{234}$};
            \node at (120:10) [left]{\scriptsize $\texttt{1}$};
            \node at (-120:4) [below]{\scriptsize $\texttt{34}$};
            \node at (-120:9) [below=.2cm]{\scriptsize $\texttt{4}$};
            \node at (0:5) [left]{\scriptsize $\texttt{12}$};
            \node at (0:10) [below left]{\scriptsize $\texttt{123}$};
        \end{tikzpicture}
        , \hspace{2em}
        \begin{tikzpicture}[scale=.1, rotate=-60]
            \foreach \i [evaluate=\i as \x using (\i+1)*60]in {0,1,...,5}
            {
                \coordinate (A\i) at (\x:15);
            }
        \fill[cyan!20] (A0) -- (A1) -- (A3) -- cycle;
        \fill[blue!20] (A0) -- (A3) -- (A5) -- cycle;
            \foreach \i [evaluate=\i as \j using {mod(\i+2,6)}]in {1,2,4,5}
            {
                \draw[wline] (A\i) -- (A\j);
            }
            \draw[wline] (A0) -- (A3);
            \draw[wline] (A2) -- (A5);
            \foreach \i [evaluate=\i as \j using {mod(\i+1,6)}]in {0,1,...,5}
            {
                \bdryline{(A\i)}{(A\j)}{2cm}
            }
            \foreach \k in {0,1,...,5}
            {
                \draw[fill] (A\k) circle (20pt);
            }
            \node at (120:10) [below=.1cm]{\scriptsize $\texttt{234}$};
            \node at (120:10) [left]{\scriptsize $\texttt{1}$};
            \node at (-120:4) [below]{\scriptsize $\texttt{34}$};
            \node at (-120:9) [below=.2cm]{\scriptsize $\texttt{4}$};
            \node at (0:5) [left]{\scriptsize $\texttt{12}$};
            \node at (0:10) [below left]{\scriptsize $\texttt{123}$};
        \end{tikzpicture}
    \end{center}
\end{ex}

\begin{ex}[genus $1$ example]
Here, we see an example where some elements of $\sfC_{A, \bullet}$ or $\sfC_{B, \bullet}$ have multiplicities.
Let us consider the walled surface $(\Sigma, \bW)$ shown in \cref{fig:walled-polygons}, where the underlying surface $\Sigma$ is a torus with one boundary component with two special points, and $\bW$ consists of a single wall $\xi = (\gamma, 0)$.

In this setting, we consider the multiplication of two ideal arcs $\kappa$ and $\kappa'$ shown in \cref{fig:torus_single_wall}.
Then, 
$\sfC_{A, \alpha} = \sfC_{A, \beta} = \sfC_{B, \beta} = \sfC_{A,A} = \emptyset$,
$\sfC_{B, \alpha} = \{\gamma\}$ and
$\sfC_{B,B} = \{ \gamma, \gamma \}$.
By \cref{eq:polygon}, we get
\begin{align*}
    \kappa \kappa' = A \alpha^2 + A^{-1} a_0^2 z_{0,-} \beta_1 \beta_2.
\end{align*}
\end{ex}
\begin{figure}[ht]
    \centering
    \begin{tikzpicture}[scale=1.1]
    \filldraw [gray!30]  (-1.7,2) rectangle (-1.5,-1);
    \draw [wline] (-1.5,-1) .. controls (-0.65,-0.7) and (0.25,-0.6) .. (0.75,-0.6);
    \draw [wline] (-1,2) .. controls (-0.5,1) and (0.5,0) ..node[midway,left]{$\gamma$}  (1,-0.25);
    \draw [wline] (-0.5,2) .. controls (0,1) and (0.5,0.5) ..(1.25,0.1);
    \draw [wline] (0,2) .. controls (0.35,1.4) and (0.75,0.95) .. (1.5,0.5);
    \draw [very thick] (-1.5,2) node [fill, circle, inner sep=1.5] (v5) {} -- (-1.5,0.5) node [fill, circle, inner sep=1.5] {} -- (-1.5,-1) node [fill, circle, inner sep=1.5] (v1) {};
    \node [fill, circle, inner sep=1.5] (v2) at (0.5,-1) {};
    \node [fill, circle, inner sep=1.5] (v3) at (1.5,0.5) {};
    \node [fill, circle, inner sep=1.5] (v4) at (0.5,2) {};
    \draw [->>-={.35}{}] (v1) -- (v2);
    \draw [->-={.4}{}] (v2) -- (v3);
    \draw [-<<-={.8}{}] (v3) -- (v4);
    \draw [-<-={.65}{}] (v4) -- (v5);
    \draw [red, ultra thick] (v4) edge (v2);
    \draw [red, ultra thick] (v5) edge (v3);
    \draw [red, ultra thick] (v5) .. controls (-1.1,0.5) and (-0.5,-0.2) .. (v2);
    \draw [red, ultra thick] (v2) .. controls (0.55,-0.55) and (1.1,0.25) .. (v3);
    \draw [red, ultra thick] (v4) .. controls (0.6,1.5) and (1.1,0.85) .. (v3);
    \node [red] at (-0.25,1.2) {$\kappa$};
    \node [red] at (0.3,-.1) {$\kappa'$};
    \node [red] at (-1,0) {$\beta_1$};
    \node [red] at (1.15,-0.4) {$\alpha$};
    \node [red] at (1.15,1.4) {$\beta_2$};
    \end{tikzpicture}
    \caption{The walled surface $(\Sigma,\bW)$ and some arcs on it. Here two pairs of sides are glued together to form a torus with one boundary component.}
    \label{fig:torus_single_wall}
\end{figure}
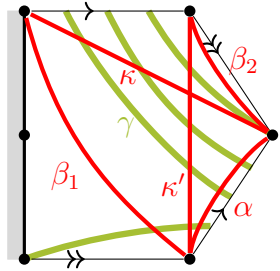
\section{Cluster algebras from marked surfaces}\label{sect:cluster}
In this section, we define the quantum cluster algebra with coefficients, and recall the quantum cluster algebra associated with a marked surface by Muller \cite{Muller16}.
Moreover, we recall the lamination and its shear coordinate, as a preparation for the next section.



\smallskip

We will work on the following coefficient ring. 
For a semifield $\bP=(\bP,\oplus,\cdot)$, let 
$\bZP$ denote the group ring of the group $(\bP,\cdot)$ over $\bZ_q$. It consists of the elements of the form $\sum_i c_i(q) u_i$ with $c_i(q) \in \bZ_q$, $u_i \in \mathbb{P}$. 

\subsection{Quantum cluster algebras with coefficients}
In this section, we define quantum cluster algebras with coefficients, which is a slight generalization of the quantum cluster algebra \cite{BZ} along with the formulation given in \cite{FZ-CA4}. See also \cite[Section 4]{CFMM}.

Fix a pair $(I,I_\uf)$, where $I=\{1,\dots,N\}$ is a finite set of indices and $I_\uf \subset I$ is a subset. The indices in $I_\uf$ are called \emph{unfrozen indices}, while those in the complement $I_\f:=I \setminus I_\uf$ are called \emph{frozen indices}. 
Let $\bP$ be a semifield, 
and $\cF$ a skew-field.

A \emph{(labeled) quantum seed} in $\cF$ is a tuple $(\ve,\Pi, \mathbf{A}, \mathbf{p})$, where
\begin{itemize}
    \item $\ve=(\ve_{ij})_{i \in I_\uf, j \in I}$ is an integral matrix such that the submatrix $(\ve_{ij})_{i,j \in I_\uf}$ is skew-symmetric;
    \item $\Pi=(\pi_{ij})_{i,j \in I}$ is a skew-symmetric integral matrix satisfying the \emph{compatibility relation} 
    \begin{align*}
        \sum_{k \in I}\ve_{ik}\pi_{kj} = \delta_{ij}d_j
    \end{align*}
    for all $i \in I_\uf$ and $j \in I$, where $d_i$ is a positive integer for $i \in I_\uf$. It defines a skew-symmetric form on a lattice $M=\bigoplus_{i \in I}\bZ f_i$ by $\Pi(f_i,f_j):=\pi_{ij}$;
    \item $\mathbf{A}: M \to \bT \subset \cF$ is a based quantum torus over $\bZP$ associated with $(M,\Pi)$ 
    whose skew-field of fractions coincides with $\cF$;
    \item $\mathbf{p} = (p_i^\pm)_{i \in I_\uf}$ is a $2|I_\uf|$-tuple of elements in $\bP$.
\end{itemize}
We call $\ve$ the \emph{exchange matrix}, $\Pi$ the \emph{compatibility matrix}, $\mathbf{A}$ the \emph{toric frame}, and $\mathbf{p}$ the \emph{coefficient tuple} of the quantum seed.
The elements $A_i:=\mathbf{A}(f_i) \in \cF$ for $i \in I$ are called the \emph{(quantum) cluster variables}, and those for $i \in I_\f$ are called the \emph{frozen variables}.
If the coefficient tuple $\mathbf{p}$ satisfies $p^+_i \oplus p^-_i = 1$ for all $i \in I_\uf$, then the seed $(\varepsilon, \Pi, \mathbf{A}, \mathbf{p})$ is said to be \emph{normalized}.
When 
$\bP$ is the trivial semifield, we say that the quantum seed $(\ve,\Pi, \mathbf{A}, \mathbf{p}) = (\ve, \Pi,\mathbf{A})$ is \emph{coefficient-free}.

By \cite[Lemma 4.4]{BZ}, a toric frame $\mathbf{A}$ is uniquely determined by the cluster variables. Indeed, 
we have
\begin{align}\label{eq:extension_toric_frame}
    \mathbf{A}\left(\sum_{i \in I} x_i f_i\right) = q^{\frac{1}{2}\sum_{l < k}x_k x_l \pi_{kl} } A_1^{x_1}\dots A_N^{x_N}
\end{align}
for all $(x_1,\dots,x_N) \in \bZ^N$.
Note that both sides are invariant under permutations of indices. 
Elements of the form $\mathbf{A}_\alpha$ for $\alpha \in M$ are called \emph{cluster monomials}. 

Given two quantum seeds $(\ve,\Pi,\mathbf{A},\mathbf{p})$ and $(\ve',\Pi',\mathbf{A}',\mathbf{p}')$ in $\cF$ are related by a quantum seed mutation at an unfrozen index $k \in I_\uf$ if they are related by the following rule. 
Let $E_{k,\epsilon}= (e_{ij})_{i,j \in I}$ and $F_{k,\epsilon}= (f_{ij})_{i,j \in I}$ be the matrices defined by
\begin{align*}
    e_{ij}:=\begin{cases}
    \delta_{ij} & \mbox{if $j \neq k$},\\
    -1 & \mbox{if $i=k=j$},\\
    [-\epsilon \ve_{ki}]_+ & \mbox{if $i\neq k=j$},
    \end{cases}
\end{align*}
and 
\begin{align*}
    f_{ij}:=\begin{cases}
    \delta_{ij} & \mbox{if $i \neq k$},\\
    -1 & \mbox{if $i=k=j$},\\
    [\epsilon \ve_{jk}]_+ & \mbox{if $i=k \neq j$},
    \end{cases}
\end{align*}
respectively for $\epsilon \in \{+,-\}$.
Then,
    \begin{align}
    \ve' &= F^{\mathsf{T}}_{k,\epsilon} \ve E^{\mathsf{T}}_{k,\epsilon}, \label{eq:matrix-mutation}\\
    \Pi' &= E^{\mathsf{T}}_{k,\epsilon} \Pi E_{k,\epsilon}, \label{eq:compatible-mutation}\\
    A'_i &= 
    \begin{cases}
    \mathbf{A}(f_i) & \mbox{if $i \neq k$}, \\
   p_k^+\mathbf{A}\big(-f_k+\sum_{j \in I} [\ve_{kj}]_+ f_j\big) + p_k^-\mathbf{A}\big(-f_k+\sum_{j \in I} [-\ve_{kj}]_+ f_j\big) & \mbox{if $i=k$}.
    \end{cases}\label{eq:q-mutation} \\
    p'^\pm_k &= p^\mp_k,\label{eq:coeff_mut_i=k}\\
    \frac{p'^+_i}{p'^-_i} &=\begin{cases}\label{eq:coeff_mut_i!=k}
    (p^+_k)^{\varepsilon_{ik}} \dfrac{p^+_i}{p^-_i} & \mbox{if } \epsilon_{ik} \geq 0,\\
    (p^-_k)^{\varepsilon_{ik}} \dfrac{p^+_i}{p^-_i} & \mbox{if } \epsilon_{ik} \leq 0
    \end{cases}
    \quad \text{for } i \neq k.
\end{align}
Here, $A'_i = \textbf{A}'(f'_i)$ for the lattice $M' = \bigoplus_{i \in I}\bZ f'_i$
It can be verified that \eqref{eq:matrix-mutation} and \eqref{eq:compatible-mutation} do not depend on the sign $\epsilon$ \cite[Proposition 3.4]{BZ}. The mutation rule \eqref{eq:matrix-mutation} is the same one as the well-known matrix mutation formula.
The relation \eqref{eq:q-mutation} is a variant of the \emph{quantum exchange relation} with coefficients. 

Note that the coefficient tuple (and hence the cluster variables) of the quantum seed $(\ve',\Pi',\mathbf{A}',\mathbf{p}')$ is not uniquely determined from $(\ve,\Pi,\mathbf{A},\mathbf{p})$ by the mutation relations \eqref{eq:coeff_mut_i=k} and \eqref{eq:coeff_mut_i!=k}. This ambiguity can be described as the orbit of a certain rescaling action by elements of $\bP$ \cite[Proposition 2.3]{Fra16}. Fraser thus called them the \emph{seed orbits}. 

If the quantum seed $(\ve,\Pi,\mathbf{A},\mathbf{p})$ is normalized, then the normalization condition for $\mathbf{p}'$ uniquely determines $(\ve',\Pi',\mathbf{A}',\mathbf{p}')$.
In this case, we write $(\ve',\Pi',\mathbf{A}',\mathbf{p}') = \mu_k (\ve,\Pi,\mathbf{A},\mathbf{p})$.

\begin{rem}
\begin{enumerate}
    \item For $\epsilon \in \{+,-\}$, we have an isomorphism $\mu_{k,\epsilon}^\ast:(M',\Pi') \xrightarrow{\sim} (M,\Pi)$, $f'_i\mapsto \sum_{j \in I} f_j (E_{k,\epsilon})_{ji}$. 
    \item The quantum exchange relation \eqref{eq:q-mutation} is the same one as \cite[(29)]{CFMM}. For $k=i$, it can be further rewritten as
\begin{align}\label{eq:exch_rel}
    A_k A'_k = p_k^+q^{\frac{1}{2} \sum_j [\ve_{kj}]_+\pi_{kj}}\mathbf{A}\left(\sum_{j \in I} [\ve_{kj}]_+ f_j\right) + p_k^-q^{\frac{1}{2} \sum_j [-\ve_{kj}]_+\pi_{kj}}\mathbf{A}\left(\sum_{j \in I} [-\ve_{kj}]_+ f_j\right)
\end{align}
\end{enumerate}
\end{rem}

\begin{dfn}[cf. {\cite[Definition 2.6]{FT18}}]\label{def:ex_pattern}
Let $\mathbf{E}$ be a connected $|I_\uf|$-regular graph. For a vertex $v$ of $\mathbf{E}$, let $\mathrm{st}(v)$ be the set of edges incident to $v$. 
\begin{enumerate}
    \item An \emph{attachment} of a quantum seed $(\ve,\Pi, \mathbf{A}, \mathbf{p})$ in $\cF$ at a vertex $v$ is a bijective labeling of their indices by the set $\mathrm{st}(v)$. With such an attachment, we typically use the notation $\mathbf{A}=\mathbf{A}^{\!(v)}=(A_\alpha^{(v)})_{\alpha \in \mathrm{st}(v)}$, and so on.
    \item An \emph{exchange pattern} $\sfs$ on $\mathbf{E}$ is a quantum seed attachment $\sfs^{(v)}=(\ve^{(v)},\Pi^{(v)}, \mathbf{A}^{\!(v)}, \mathbf{p}^{(v)})$ at each vertex $v$ so that for each edge $v \overbarnear{\alpha} v'$, the quantum seeds $\sfs^{(v)},\sfs^{(v')}$ are related by the mutation $\mu_\alpha$.  
\end{enumerate}
\end{dfn}

\begin{dfn}
The \emph{quantum cluster algebra} associated with an exchange pattern $\sfs$ 
of quantum seeds is the $\bZP$-subalgebra $\CA_{\sfs} \subset \cF$ generated by all the quantum cluster variables and the inverses of frozen variables.
\end{dfn}

\subsection{Quantum cluster algebras from marked surfaces}
Following \cite{Muller16}, we recall the construction of a coefficient-free quantum cluster algebra $\CA_\Sigma$ in the skew-field $\Frac\Skein{\Sigma}$ of fractions of the skein algebra $\Skein{\Sigma}$ for a marked surface $\Sigma$.

For an ideal triangle $T$ of an ideal triangulation $\tri$,
we define the matrix $\ve(T) = (\ve^{\tri}_{\alpha \beta}(T))_{\alpha \in \tri_\uf, \beta \in \tri}$ by
\begin{align*}
\ve^{\tri}_{\alpha \beta}(T):= \begin{cases}
        1 & \mbox{if $T$ has $\alpha$ and $\beta$ as its consecutive edges in the counter-clockwise order}, \\
        -1 & \mbox{if the same holds with the clockwise order}, \\
        0 & \mbox{otherwise}.
    \end{cases}
\end{align*}
Moreover, we set $\ve^{\tri}:=\sum_T \ve(T)$, where $T$ runs over all ideal triangles of $\tri$.

For each $\alpha \in \tri$, let $A_\alpha \in \Skein{\Sigma}$ be the corresponding element in the skein algebra.
We define the compatibility matrix $\Pi^{\tri}=(\pi_{\alpha\beta})_{\alpha,\beta \in \tri}$ as follows. For $\alpha \in \tri$, let $\alpha_1,\alpha_2$ denote its two ends (with an arbitrary numbering). For two edges $\alpha,\beta \in \tri$, define 
\begin{align*}
    \pi_{\alpha_i,\beta_j}:=\begin{cases}
    1 & \mbox{if $\alpha_i$ is clockwise to $\beta_j$ at a common marked point}, \\
    -1 & \mbox{if $\alpha_i$ is counter-clockwise to $\beta_j$  at a common marked point}, \\
    0 & \mbox{otherwise}
    \end{cases}
\end{align*}
and set $\pi_{\alpha\beta}:=\sum_{i,j=1}^2 \pi_{\alpha_i,\beta_j}$. Then it is easy to verify that
\begin{align*}
    A_\alpha A_\beta = q^{\pi_{\alpha\beta}} A_\beta A_\alpha
\end{align*}
holds for all $\alpha,\beta \in \tri$. 
Consider a lattice $M^\tri=\bigoplus_{\alpha \in \tri}\bZ f_\alpha$ with basis parametrized by arcs in $\tri$, on which $\Pi^\tri$ defines a skew-symmetric form. 
Define a toric frame $\mathbf{A}^{\!\tri}:M^\tri \to \bT^\tri \subset \Skein{\Sigma}$ by $\mathbf{A}^{\!\tri}(f_\alpha):=A_\alpha$ for $\alpha \in \tri$, where $\bT^\tri$ is a based quantum torus associated with $(M^\tri,\Pi^\tri)$ over $\bZ_q$. 

\begin{lem}[{\cite[Theorem 6.14 and Proposition 7.8]{Muller16}}]\label{lem:Muller}
The triple $(\ve^\tri,\Pi^\tri,\mathbf{A}^{\!\tri})$ forms a quantum seed. Namely:
\begin{enumerate}
    \item The skew-field $\Frac\bT^\tri$ of fractions coincides with $\Frac\Skein{\Sigma}$.
    \item The compatibility relation 
    \begin{align*}
    \sum_{\beta \in \tri} \ve_{\alpha\beta}^\tri \pi_{\beta\gamma} = 4\delta_{\alpha\gamma}
    \end{align*}
    holds for all $\alpha \in \tri_\uf$, $\gamma \in \tri$.
\end{enumerate}
 
\end{lem}

It is known that the mutations of these quantum seeds are compatible with the flips of ideal triangulations, that is,  $(\ve^{f_\kappa \tri},\Pi^{f_\kappa \tri},\mathbf{A}^{f_\kappa \tri}) = \mu_\kappa (\ve^\tri,\Pi^\tri,\mathbf{A}^{\!\tri})$ for any ideal triangulation $\tri$ and $\kappa \in \tri_\uf$ \cite[Theorem 7.9]{Muller16}. 
Recall the graph $\Tri_\Sigma$ of ideal triangulations of $\Sigma$. For a vertex $\tri$ of this graph, the set $\mathrm{st}(\tri)$ is identified with the set $\tri_\uf$ of interior edges.  
The quantum seed attachment $\sfs_\Sigma: \tri \mapsto \sfs^\tri := (\varepsilon^\tri, \Pi^\tri, \textbf{A}^{\!\tri})$ on the graph $\Tri_\Sigma$ defines an exchange pattern (\cref{def:ex_pattern}).
We denote the corresponding quantum cluster algebra by $\mathscr{A}^q_\Sigma:=\mathscr{A}^q_{\sfs_\Sigma}$.

Now, we have two algebras $\mathscr{A}^q_\Sigma$ and $\Skein{\Sigma}$ in $\Frac \Skein{\Sigma}$, and an obvious inclusion $\mathscr{A}^q_\Sigma \subset \Skein{\Sigma}$ from the construction.

\begin{thm}[{\cite{Muller16}}]
If $\Sigma$ has at least two marked points, then the inclusion $\mathscr{A}^q_\Sigma \hookrightarrow \Skein{\Sigma}$ induces an isomorphism
\begin{align*}
    \mathscr{A}^q_\Sigma \cong \Skein{\Sigma}[\partial^{-1}].
\end{align*}
Here, $\Skein{\Sigma}[\partial^{-1}]$ denotes the boundary-localized skein algebra (\cref{def:boundary_localization}).
\end{thm}

\begin{figure}[h]
\begin{align*}
\tikz[scale=1.15,xscale=-1,>=latex]{
\path(0,0) node [fill, circle, inner sep=1.5pt] (x1){};
\path(135:2) node [fill, circle, inner sep=1.5pt] (x2){};
\path(0,2*1.4142) node [fill, circle, inner sep=1.5pt] (x3){};
\path(45:2) node [fill, circle, inner sep=1.5pt] (x4){};
\draw[blue](x1) to (x2) to (x3) to (x4) to (x1) to  (x3);
\color{mygreen}{
    \draw(0,1.4142) circle(2pt) coordinate(v0);
    \draw(135:2)++(45:1) circle(2pt) coordinate(v1);
    \draw(45:2)++(135:1) circle(2pt) coordinate(v2);
    \draw(45:1) circle(2pt) coordinate(v3);
    \draw(135:1) circle(2pt) coordinate(v4);
    \qarrow{v0}{v4}
    \qarrow{v4}{v1}
    \qarrow{v1}{v0}
    \qarrow{v0}{v2}
    \qarrow{v2}{v3}
    \qarrow{v3}{v0}
    }
\draw[<->, thick, dashed] (-1.7142, 1.2142) -- (-4.2142, 1.2142);
\draw(-2.9642, 1) node{mutation};
\color{blue}
\draw[<->, thick, dashed] (-1.7142, 1.6142) -- (-4.2142, 1.6142);
\draw(-2.9642, 1.9) node{flip};
\begin{scope}[rotate=90, yshift=130, xshift=1.4142cm]
\color{black}
\path(0,0) node [fill, circle, inner sep=1.5pt] (x1){};
\path(135:2) node [fill, circle, inner sep=1.5pt] (x2){};
\path(0,2*1.4142) node [fill, circle, inner sep=1.5pt] (x3){};
\path(45:2) node [fill, circle, inner sep=1.5pt] (x4){};
\draw[blue](x1) to (x2) to (x3) to (x4) to (x1) to  (x3);
\color{mygreen}{
    \draw(0,1.4142) circle(2pt) coordinate(v0);
    \draw(135:2)++(45:1) circle(2pt) coordinate(v1);
    \draw(45:2)++(135:1) circle(2pt) coordinate(v2);
    \draw(45:1) circle(2pt) coordinate(v3);
    \draw(135:1) circle(2pt) coordinate(v4);
    \qarrow{v0}{v4}
    \qarrow{v4}{v1}
    \qarrow{v1}{v0}
    \qarrow{v0}{v2}
    \qarrow{v2}{v3}
    \qarrow{v3}{v0}
    }
\end{scope}
}
\end{align*}
\vspace{-1cm}
    \caption{The local picture of the quiver representing the matrix $\ve^{\tri}$.}
    \label{fig:tri_quivl}
\end{figure}

\subsection{Laminations on marked surfaces}\label{sec:lamination}
Fomin--Thurston give a cluster algebra of a marked surface with normalized coefficients using integral $\cX$-laminatins for the classical setting ($q=1$) in \cite{FT18}.
Since the shear coordinate of the $\cX$-laminations is useful to describe the cluster structure of the walled skein algebra, we recall them here.
In \cref{subsec:comparison}, we will compare walled skein algebras and Fomin--Thurston's cluster algebra of a marked surface with laminations.

\begin{dfn}\label{d:X-lamination}
A \emph{curve} $\gamma$ in $\Sigma$ is a simple curve in $\Sigma$ such that
\begin{itemize}
    \item $\partial \gamma = \emptyset$ or $ \in \partial \Sigma \setminus \spe$ and
    \item $\gamma$ is not a boundary of an unpunctured disk or a once punctured disk.
\end{itemize}
An \emph{integral $\cX$-lamination} (or simply \emph{lamination}) on $\Sigma$ consists of a finite collection of mutually disjoint curves in $\Sigma$ modulo isotopy relative to $M$.
\end{dfn}
Let $\mathcal{L}^x(\Sigma,\bZ)$ denote the set of integral $\X$-laminations.

Given an ideal triangulaton $\tri$ of $\Sigma$, we define a coordinate system on $\cL^x(\Sigma,\bZ)$ following \cite{FG07}. 
For $L \in \mathcal{L}^x(\Sigma,\bZ)$, represent the components of $L$ and the ideal arcs in $\tri$ so that the intersections of these curves are minimal.

For each interior arc $\alpha\in \tri_\uf$ and a curve $\gamma$, let $\widehat{\mathrm{Int}}_\tri(\alpha,\gamma)$ be the intersection number given by the sum of the following contributions: an intersection as in the left (resp. right) of \cref{f:intersection sign} contributes as $+1$ (resp. $-1$), and the others $0$. 
Then, we define
\begin{align*}
    x^\tri_\alpha(L) := \sum_{j} \widehat{\mathrm{Int}}_\tri(\alpha,\gamma_i).
\end{align*}
for a lamination $L = \bigsqcup_i \gamma_i \in \cL^x(\Sigma, \bZ)$.

\begin{figure}[h]
\[
\begin{tikzpicture}[scale=1.15, xscale=-1]
\path(0,0) node [fill, circle, inner sep=1.5pt] (x1){};
\path(135:2) node [fill, circle, inner sep=1.5pt] (x2){};
\path(0,2*1.4142) node [fill, circle, inner sep=1.5pt] (x3){};
\path(45:2) node [fill, circle, inner sep=1.5pt] (x4){};
\draw[blue] (x1) to (x2) to (x3) to (x4) to (x1) to node[midway,right]{$\alpha$} node[midway,below left,black]{$\ominus$} (x3);
\draw [red] (-1,0.5)  to[out=45,in=215] (0,1.2) to[out=45,in=215] (1,3);
\draw [red] (-1,0.5) node[below]{$\gamma$}; 

\begin{scope}[xshift=5cm]
\path(0,0) node [fill, circle, inner sep=1.5pt] (x1){};
\path(135:2) node [fill, circle, inner sep=1.5pt] (x2){};
\path(0,2*1.4142) node [fill, circle, inner sep=1.5pt] (x3){};
\path(45:2) node [fill, circle, inner sep=1.5pt] (x4){};
\draw[blue](x1) to (x2) to (x3) to (x4) to (x1) to node[midway,right]{$\alpha$} node[midway,above left,black]{$\oplus$} (x3);
\draw [red] (2,0.7) to[out=135,in=-45] (0,1.5) to[out=135,in=-45] (-1,3);
\draw [red] (2,0.7) node[below]{$\gamma$}; 
\end{scope}
\end{tikzpicture}
\]
\caption{Contribution of a intersection to $\widehat{\mathrm{Int}}_\tri(\alpha,\gamma)$.}
\label{f:intersection sign}
\end{figure}


\begin{prop}[{\cite[Section 3.1]{FG07}}]\label{p:x-lamination}
For any ideal triangulation $\tri$ of $\Sigma$, the map
\begin{align*}
    x_{\tri}: \mathcal{L}^x(\Sigma,\bZ) \xrightarrow{\sim} \bZ^{\tri_\uf}, \quad L \mapsto \{x^\tri_\alpha(L)\}_{\alpha \in \tri_\uf}
\end{align*}
gives a bijection.
Moreover if an ideal triangulation $\tri'$ is obtained by the flip along an arc $\kappa \in \tri_\uf$, the coordinate transformation $x_{\tri'} \circ x_{\tri}^{-1}$ is given by
\begin{align}\label{eq:trop_X_trans}
    x^{\tri'}_{\alpha}=
    \begin{cases}
    -x^{\tri}_\kappa & \mbox{if } \alpha = \kappa',\\
    x^{\tri}_\alpha - \ve^{\tri}_{\alpha\kappa}\min\left\{0,\ - \mathrm{sgn}(\ve^{\tri}_{\alpha\kappa})x^{\tri}_\kappa\right\} & \mbox{otherwise}.
    \end{cases}
\end{align}
Here, $\kappa' \in \tri'_\uf$ is the arc such that $\tri \setminus \{\kappa\} = \tri' \setminus \{\kappa'\}$.
\end{prop}

This map $x_{\tri}: \cL^x(\Sigma, \bZ) \to \bZ^{\tri_\uf}$ is called the \emph{shear coordinate}.
For later discussions, we define the map
\begin{align*}
    a^\tri_\alpha: \cL^x(\Sigma, \bZ) \to \frac 1 2\bZ
\end{align*}
by the half of the geometric intersection number with $\alpha \in \tri$.
Then, we have the following:
\begin{lem}
\begin{enumerate}
    \item For an ideal triangulation $\tri$ and $\kappa \in \tri$, we have
    \begin{align*}
        a^{\tri'}_\alpha = 
        \begin{cases}
        -a^{\tri}_\kappa + \min\big\{ \sum_{\beta} [\ve^\tri_{\kappa \beta}]_+ a^\tri_\beta,\, \sum_{\beta}[-\ve^\tri_{\kappa \beta}]_+ a^\tri_\beta \big\} & \mbox{if } \alpha = \kappa',\\
        a^\tri_\alpha & \mbox{otherwise}.
        \end{cases}
    \end{align*}
    Here, $\tri'$ is the ideal triangulation such that $\tri' \setminus \{\kappa'\} = \tri \setminus \{\kappa\}$.
    \item For any $\alpha \in \tri_\uf$,
    \begin{align*}
    x^\tri_\alpha = \sum_{\beta \in \tri} \ve^\tri_{\alpha\beta} a^\tri_\beta.
    \end{align*}
\end{enumerate}
\end{lem}

\section{Comparison of skein and cluster algebras}\label{sec:comparison}


We are going to treat only tropical semifield.
We write $\mathrm{Trop}(u_j \,|\, j \in J)$ for the tropical semifield in variables $\{u_j \mid j \in J\}$ for a finite index set $J$.
Namely, it is the multiplicative group of Laurent monomials in the formal variables $u_j$ for $j \in J$ equipped with the addition operation $\oplus$ given by
    \begin{align}\label{eq:pm_coeff}
        \prod_{j \in J}u_j^{a_j} \oplus \prod_{j \in J} u_j^{b_j} = \prod_{j \in J} u_j^{\min(a_i,b_i)}.
    \end{align}

\subsection{Cluster structure of $\Skein{\Sigma, \bW}$}

In this section, we construct the quantum cluster algebra with the suitable coefficients in a walled skein algebra, and prove these are coincide except for small cases.
The quantum parameters are identified as $q=A$. 
Here, the wall system $\bW=(\sfC,J,\ell)$ is always assumed to be taut. 

For a curve $\gamma \in \mathsf{C}$, we define the curves $\gamma_{\pm}$ as follows:
\begin{itemize}
    \item if $\gamma \in \mathsf{C}_{\mathrm{arc}}$, $\gamma_{+}$ (resp. $\gamma_{-}$) is obtained by slightly sliding both of its endpoints following (resp. against) the boundary orientation induced from $\Sigma$ (see \cref{fig:deform_wall_lam});

    \item if $\gamma \in \mathsf{C}_\mathrm{loop}$, let $\gamma_{+}=\gamma_{-}:=\gamma$.
\end{itemize}



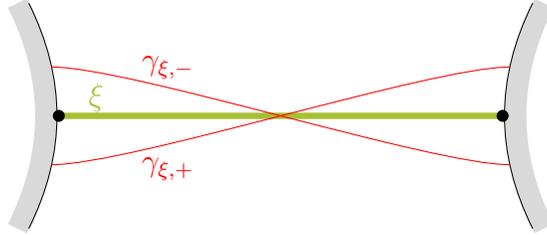
\begin{figure}[ht]
    \centering
    \begin{tikzpicture}
    \node[red] at (0,1.15) {$\gamma_{\xi,-}$};
    \node[red] at (0,-0.2) {$\gamma_{\xi,+}$};
    \node[olive] at (-0.95,0.75) {$\xi$};
    \draw [gray!30, line width=8pt](-2,2) .. controls (-1.5,1) and (-1.5,0) .. (-2,-1);
    \draw [gray!30, line width=8pt](5,2) .. controls (4.5,1) and (4.5,0) .. (5,-1);
    \draw (-1.85,2) .. controls (-1.35,1) and (-1.35,0) .. (-1.85,-1);
    \draw (4.85,2) .. controls (4.35,1) and (4.35,0) .. (4.85,-1);
    \draw [wline] (-1.45,0.5) node [fill, circle, black, inner sep=1.6] {} -- (4.45,0.5) node [fill, circle, black, inner sep=1.6] {};
    \draw [red](-1.55,1.15) .. controls (-0.5,1.15) and (3.5,-0.15) .. (4.55,-0.15);
    \draw [red](-1.55,-0.15) .. controls (-0.5,-0.15) and (3.5,1.15) .. (4.55,1.15);
    \end{tikzpicture}
    \caption{Deformation of walls to laminations.}
    \label{fig:deform_wall_lam}
\end{figure}
We use the semifield $\bP_\bW:=\mathrm{Trop}(z_{j,+}, z_{j,-} \mid j \in J)$, which satisfies 
$\bZ_q \bP_\bW = \bZ_{q,\bW}$.
For an ideal triangulation $\tri$ and $\alpha \in \tri_\uf$, we define 
\begin{align}\label{eq:coefficients_walls}
    p^{\tri,\pm}_{\bW; \alpha}:= \prod_{j \in J} \prod_{\gamma \in \sfC_j} z_{j, +}^{[\pm x^\tri_\alpha(\gamma_{+})]_+} z_{j, -}^{[\pm x^\tri_\alpha(\gamma_{-})]_+}
\end{align}
and let
$\bp^\tri_\bW := (p^{\tri,\pm}_{\bW; \alpha})_{\alpha \in \tri_\uf}$. If $W_j$ has no crossings, then the collection $L_{j,\pm}:=\{ \gamma_{\pm} \mid \xi \in \mathsf{C}_j\}$ defines a lamination. If this is the case for all $j \in J$, then \eqref{eq:coefficients_walls} can be rewritten as
\begin{align}\label{eq:coefficients_walls_normalized}
    p^{\tri,\pm}_{\bW; \alpha}:= \prod_{j \in J}  z_{j, +}^{[\pm x^\tri_\alpha(L_{j,+})]_+} z_{j, -}^{[\pm x^\tri_\alpha(L_{j,-})]_+}.
\end{align}
Recall that any simple multicurve $C$ admits a $\bW$-minimal diagram $C_\bW$ (\cref{lem:W-minimal-lift}). For an ideal arc $\alpha \in \tri$, let
\begin{align}\label{eq:CV_walled}
    A_{\bW;\alpha}:=[\alpha_\bW]_\bW \in \Skein{\Sigma, \bW}.
\end{align}
Here $[-]_\bW$ stands for the isotopy class relative to $\bigcup \sfC$. 
This assignment is extended to a toric frame $\mathbf{A}^{\!\tri}_\bW: M^\tri \to \bT_\bW^\tri \subset \mathop{\mathrm{Frac}}\Skein{\Sigma, \bW}$, $f_\alpha \mapsto A_{\bW;\alpha}$, where $\bT_\bW^\tri$ is the quantum torus over $\bZ_{q,\bW}$ generated by $A_{\bW;\alpha}$ for $\alpha \in \tri$.
Then the tuple $(\epsilon^{\tri},\Pi^{\tri},\mathbf{A}^{\!\tri}_{\bW}, \mathbf{p}^{\tri}_{\bW})$ is a quantum seed in the ambient field $\mathop{\mathrm{Frac}}\Skein{\Sigma, \bW}$.
Indeed, we have $\mathop{\mathrm{Frac}}\Skein{\Sigma, \bW} \cong \Frac \bT_\bW^\tri$ by \cref{thm:quantum-torus}.

\begin{lem}\label{lem:mut_in_skein}
The attachement $\sfs_{\Sigma, \bW}: \tri \mapsto (\varepsilon^\tri, \Pi^\tri, \mathbf{A}^\tri_\bW, \bp^\tri_\bW)$ of quantum seeds in $\Frac \Skein{\Sigma, \bW}$ is an exchange pattern on $\Tri_\Sigma$.
In particular, it is normalized if $\sfC_j$ has no crossings for any $j \in J$.
\end{lem}

\begin{proof}
We must check the quantum exchange relation \eqref{eq:exch_rel} and the exchange relation for coefficients \eqref{eq:coeff_mut_i=k}, \eqref{eq:coeff_mut_i!=k} in this realization.

Let us focus on the quadrilateral in $\tri$ with sides are $\alpha_1, \alpha_2, \beta_1, \beta_2 \in \tri$ and diagonal $\kappa \in \tri$.
Let $\kappa' \in \tri'$ satisfying $\tri \setminus \{\kappa\} = \tri' \setminus \{\kappa'\}$ as in \cref{fig:AB-regions1}. 
Note that the coefficients can be rewritten as
\begin{align*}
    p^{\tri,\pm}_{\bW; \alpha} = \prod_{\gamma \in \mathsf{C}} z_{\ell(\gamma), +}^{[\pm x^\tri_\alpha(\gamma_{+})]_+} z_{\ell(\gamma), -}^{[\pm x^\tri_\alpha(\gamma_{-})]_+} 
\end{align*}
for $\alpha \in \tri_\uf$. 
Then, with the notation in \cref{subsec:skein_formula}, we get
\begin{align*}
p_{\bW;\kappa}^{\tri,+}
&= \bigg( \prod_{\gamma \in \sfC_{B, B} \cup \sfC_{B, \kappa'}} z_{\ell(\gamma), +} \bigg) \bigg( \prod_{\zeta \in \sfC_{B, B} \cup \sfC_{B, \kappa}} z_{\ell(\zeta), -} \bigg)\\
&=\bigg(\!\prod_{\eta\in\sfC_{B,B}}\!\!\! a_{\ell(\eta)}\bigg)
\bigg(\!\prod_{\gamma\in\sfC_{B,\kappa'}}\!\!\! z_{\ell(\gamma),+}\bigg)
\bigg(\!\prod_{\zeta\in\sfC_{B,\kappa}}\!\!\! z_{\ell(\zeta),-}\bigg),\\
p_{\bW;\kappa}^{\tri,-}
&= \bigg( \prod_{\gamma \in \sfC_{A, A} \cup \sfC_{A, \kappa}} z_{\ell(\gamma), +} \bigg) \bigg( \prod_{\zeta \in \sfC_{A, A} \cup \sfC_{A, \kappa'}} z_{\ell(\zeta), -} \bigg)\\
&= \bigg(\!\prod_{\eta\in\sfC_{A,A}}\!\!\!a_{\ell(\eta)}\bigg) 
\bigg(\!\prod_{\gamma\in\sfC_{A,\kappa}}\!\!\! z_{\ell(\gamma),+}\bigg)
\bigg(\!\prod_{\zeta\in\sfC_{A,\kappa'}}\!\!\! z_{\ell(\zeta),-}\bigg).
\end{align*}
By comparing with \cref{eq:polygon}, we see that the quantum exchange relation
\begin{align*}
    A_{\bW;\kappa} A'_{\bW;\kappa} = q^{-1}(p^\tri_{\bW;\kappa})^+A_{\bW;\beta_1}A_{\bW;\beta_2} + q(p^\tri_{\bW;\kappa})^- A_{\bW;\alpha_1}A_{\bW;\alpha_2}
\end{align*}
holds in $\Skein{\Sigma, \bW}$. Thus the first assertion is proved.

In the case where $W_j$ has no crossings for all $j \in J$, we can use the formula \eqref{eq:coefficients_walls_normalized}. Observe that for a real number $a \in \bR$, either $[a]_+=0$ or $[-a]_+=0$ holds. Hence 
\begin{align*}
    z_{j,\epsilon}^{[x_\alpha^\tri(L_{j,\epsilon})]_+} \oplus z_{j,\epsilon}^{[-x_\alpha^\tri(L_{j,\epsilon})]_+} = 1
\end{align*}
holds for all $\alpha \in \tri_\uf$, $j \in J$ and $\epsilon \in \{+,-\}$. Thus the normalization condition $p_{\bW;\alpha}^{\tri,+} \oplus p_{\bW;\alpha}^{\tri,-}=1$ holds for all $\alpha \in \tri_\uf$. 
\end{proof}

\begin{ex}[Non-normalized seeds from a wall system with crossings]\label{ex:non-normalized seed}
Let us consider a wall system $\bW=(\sfC,J,\ell)$, where $\sfC=\{\gamma_1,\gamma_2\}$ consists of two ideal arcs as shown in \cref{fig:non-normalized}, $J:=\{0\}$ and $\ell : \sfC \to \{0\}$ be the constant map. Observe that $\sfC=\sfC_0$ has a crossing. 
Then the coefficients with respect to the ideal triangulation $\tri=\{\alpha_1,\alpha_2,\beta_1,\beta_2,\kappa\}$ are computed as
\begin{align*}
    p_{\bW;\kappa}^{\tri,+}=z_{0,+}z_{0,-}=p_{\bW;\kappa}^{\tri,-},
\end{align*}
which is non-normalized.
We have the exchange relation
\begin{align}\label{eq:non-normalized relation}
    A_{\kappa;\bW}A_{\kappa';\bW}=z_{0,+}z_{0,-}A_{\alpha_1;\bW}A_{\alpha_2;\bW} + z_{0,+}z_{0,-}A_{\beta_1;\bW}A_{\beta_2;\bW}.
\end{align}
\end{ex}

\begin{figure}[ht]
    \centering
    \begin{tikzpicture}
    \draw [gray!30, line width=8pt] (-2.55,2.9) .. controls (-2.2,3) and (-2,3.2) .. (-1.9,3.55);
    \draw [gray!30, line width=8pt] (0.9,3.55) .. controls (1,3.25) and (1.25,3) .. (1.55,2.9);
    \draw [gray!30, line width=8pt] (-2.55,0.1) .. controls (-2.25,0) and (-2,-0.25) .. (-1.9,-0.55);
    \draw [gray!30, line width=8pt] (0.9,-0.55) .. controls (1.05,-0.25) and (1.3,0) .. (1.55,0.1);
    \draw (-2.5,2.75) .. controls (-2.1,2.9) and (-1.9,3.1) .. (-1.75,3.5);
    \draw (0.75,3.5) .. controls (0.9,3.15) and (1.15,2.9) .. (1.5,2.75);
    \draw (1.5,0.25) .. controls (1.15,0.1) and (0.95,-0.15) .. (0.75,-0.5);
    \draw (-1.75,-0.5) .. controls (-1.9,-0.15) and (-2.15,0.1) .. (-2.5,0.25);
    \node [fill, circle, inner sep=1.3] (v1) at (-2.05,3.05) {};
    \node [fill, circle, inner sep=1.3] (v2) at (-2.05,-0.05) {};
    \node [fill, circle, inner sep=1.3] (v3) at (1.05,-0.05) {};
    \node [fill, circle, inner sep=1.3] (v4) at (1.05,3.05) {};
    \draw [blue] (v1) edge (v2);
    \draw [blue] (v2) edge (v3);
    \draw [blue] (v3) edge (v4);
    \draw [blue] (v1) edge (v4);
    \draw [blue] (v2) edge (v4);
    \draw [blue] (v1) edge (v3);
    \node [blue] at (-0.5,3.35) {$\beta_1$};
    \node [blue] at (-2.4,1.5) {$\alpha_1$};
    \node [blue] at (-0.5,-0.4) {$\beta_2$};
    \node [blue] at (1.4,1.5) {$\alpha_2$};
    \node [blue] at (0,2.4) {$\kappa$};
    \node [blue] at (-1,2.45) {$\kappa'$};
    \draw[wline] (v2) to[bend right=15] node[midway,right]{$\gamma_1$} (v4);
    \draw[wline] (v1) to[bend right=15] node[midway,left]{$\gamma_2$} (v3);
    \end{tikzpicture}
    \caption{A wall system with a crossing where $\ell(\gamma_1)=\ell(\gamma_2)$.}
    \label{fig:non-normalized}
\end{figure}
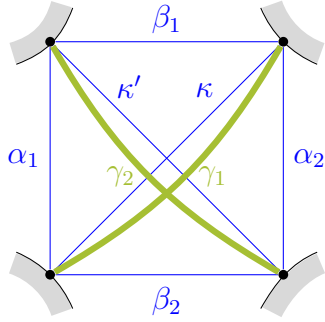

\begin{rem}\label{rem:taut_cluster}
As a technical matter, the taut condition for the wall system $\bW$ is needed to ensure the existence of minimal positions (\cref{lem:W-minimal-lift}), and hence to introduce the cluster variables as in \eqref{eq:CV_walled}. 
\end{rem}

Let $\CA^q_{\Sigma, \bW} := \CA^q_{\sfs_{\Sigma, \bW}}$.
We are going to compare the two algebras $\Skein{\Sigma, \bW}$ and $\CA^q_{\Sigma, \bW}$ over $\bZP_\bW=\bZ_{q,\bW}$ in the skew-field $\mathop{\mathrm{Frac}} \Skein{\Sigma, \bW}$.

\begin{thm}\label{thm:wskein_cluster}
For any walled surface $(\Sigma, \bW)$, we have $\CA^q_{\Sigma, \bW} \subset \Skein{\Sigma, \bW}$ 
in $\mathop{\mathrm{Frac}}\Skein{\Sigma, \bW}$.
Moreover, the equality
\begin{align*}
    \CA^q_{\Sigma, \bW} =\Skein{\Sigma, \bW}[\partial^{-1}]
\end{align*}
holds if $\Sigma$ has at least two marked points, where $\Skein{\Sigma, \bW}[\partial^{-1}]$ is the boundary-localized skein algebra (\cref{def:boundary_localization}).
\end{thm}

\begin{proof}
By \cref{lem:mut_in_skein}, we have $\CA^q_{\Sigma, \bW} \subset \Skein{\Sigma, \bW} \subset \Skein{\Sigma, \bW}[\partial^{-1}]$, since all the cluster variables are realized as elements of the skein algebra.

By \cref{thm:basis-web}, $\Skein{\Sigma, \bW}[\partial^{-1}]$ is generated by simple loops and ideal arcs, where the latter are cluster variables.
Therefore in order to prove the converse inclusion, it suffices to prove that any simple arc can be expanded to a polynomial of ideal arcs. 
Here we use the ``sticking trick'': 
the skein relations induce the relation
\begin{align}
        \mathord{
        \ \tikz[baseline=-.6ex, scale=.08, yshift=-4cm]{
            \coordinate (P) at (0,0);
            \coordinate (L) at (-5,0);
            \coordinate (R) at (5,0);
            \coordinate (C) at (0,5);
            \coordinate (T) at (0,10);
            \draw[webline] (120:10) to[out=south, in=west] (C) to[out=east, in=south] (60:10);
            \draw[dashed] (10,0) arc (0:180:10cm) -- cycle;
            \bdryline{(-10,0)}{(10,0)}{1.5cm}
            \fill (L) circle [radius=1cm];
            \fill (R) circle [radius=1cm];
            \node[scale=0.9] at (0,-4) {$E$};
        \ }
    }
    &=qA_E^{-1}
    \mathord{
        \ \tikz[baseline=-.6ex, scale=.08, yshift=-4cm]{
            \coordinate (P) at (0,0);
            \coordinate (L) at (-5,0);
            \coordinate (R) at (5,0);
            \coordinate (C) at (0,5);
            \coordinate (T) at (0,10);
            \begin{scope}
                \clip (10,0) arc (0:180:10cm) -- cycle;
                \draw[webline] (R) to[out=north, in=south] ($(L)+(0,10)$);
                \draw[webline, overarc] (L) to[out=north, in=south] ($(R)+(0,10)$);
            \end{scope}
            \draw[dashed] (10,0) arc (0:180:10cm);
            \bdryline{(-10,0)}{(10,0)}{1.5cm}
            \fill (L) circle [radius=1cm];
            \fill (R) circle [radius=1cm];
        \ }
    }
    -q^{2}A_E^{-1}
    \mathord{
        \ \tikz[baseline=-.6ex, scale=.08, yshift=-4cm]{
            \coordinate (P) at (0,0);
            \coordinate (L) at (-5,0);
            \coordinate (R) at (5,0);
            \coordinate (C) at (0,5);
            \coordinate (T) at (0,10);
            \begin{scope}
                \clip (10,0) arc (0:180:10cm) -- cycle;
                \draw[webline] (L) -- ($(L)+(0,10)$);
                \draw[webline] (R) -- ($(R)+(0,10)$);
            \end{scope}
            \draw[dashed] (10,0) arc (0:180:10cm);
            \bdryline{(-10,0)}{(10,0)}{1.5cm}
            \fill (L) circle [radius=1cm];
            \fill (R) circle [radius=1cm];
        \ }
    }\label{rel:stick}
\end{align}
for any boundary interval $E$, together possibly with additional coefficients $z_{j,\pm}$. 
If $\Sigma$ has at least two marked points, then there are two distinct boundary intervals $E_1$ and $E_2$.
Then we can expand a simple loop to a polynomial of simple arcs by applying the sticking trick \eqref{rel:stick} twice with respect to $E_1$ and $E_2$, see \cref{fig:loop-exp}.
Hence, any simple loop is an element of the quantum cluster algebra. 
We remark that the coefficients of the polynomial may have variables $z_{j,\pm}$, 
which causes no problem.
\end{proof}

\begin{figure}[ht]
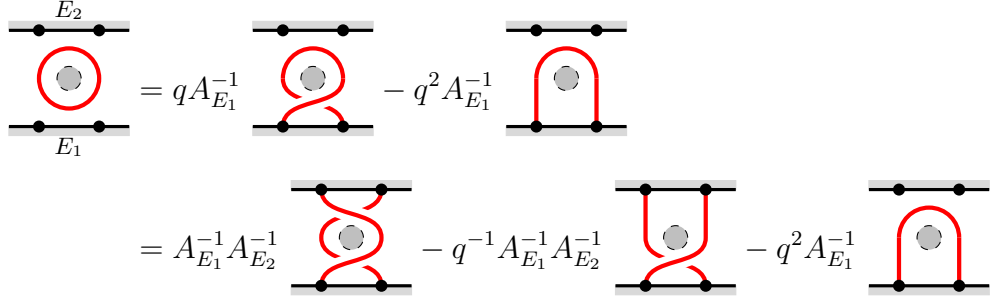

    \centering
    \begin{align*}
    \mathord{
        \ \tikz[baseline=-.6ex, scale=.08, yshift=-6cm]{
            \coordinate (P) at (0,0);
            \coordinate (L) at (-5,0);
            \coordinate (R) at (5,0);
            \coordinate (C) at (0,5);
            \coordinate (T) at (0,10);
            \draw[webline] (0,8) circle [radius=5cm];
            \draw[fill=lightgray, dashed] (0,8) circle [radius=2cm];
            \bdryline{(-10,0)}{(10,0)}{1.5cm}
            \bdryline{(-10,16)}{(10,16)}{-1.5cm}
            \fill (-5,0) circle [radius=1cm];
            \fill (5,0) circle [radius=1cm];
            \fill (-5,16) circle [radius=1cm];
            \fill (5,16) circle [radius=1cm];
            \node at (0,0) [below]{\scriptsize $E_1$};
            \node at (0,16) [above]{\scriptsize $E_2$};
        \ }
    }
    &=qA_{E_1}^{-1}
    \mathord{
        \ \tikz[baseline=-.6ex, scale=.08, yshift=-6cm]{
            \coordinate (P) at (0,0);
            \coordinate (L) at (-5,0);
            \coordinate (R) at (5,0);
            \coordinate (C) at (0,5);
            \coordinate (T) at (0,10);
            \draw[webline, overarc] (R) to[out=north, in=south] (-5,8);
            \draw[webline, overarc] (L) to[out=north, in=south] (5,8);
            \draw[webline] (-5,8) to[out=north, in=west] (0,13) to[out=east, in=north] (5,8);
            \draw[fill=lightgray, dashed] (0,8) circle [radius=2cm];
            \bdryline{(-10,0)}{(10,0)}{1.5cm}
            \bdryline{(-10,16)}{(10,16)}{-1.5cm}
            \fill (-5,0) circle [radius=1cm];
            \fill (5,0) circle [radius=1cm];
            \fill (-5,16) circle [radius=1cm];
            \fill (5,16) circle [radius=1cm];
        \ }
    }
    -q^{2}A_{E_1}^{-1}
    \mathord{
        \ \tikz[baseline=-.6ex, scale=.08, yshift=-6cm]{
            \coordinate (P) at (0,0);
            \coordinate (L) at (-5,0);
            \coordinate (R) at (5,0);
            \coordinate (C) at (0,5);
            \coordinate (T) at (0,10);
            \draw[webline, overarc] (L) to[out=north, in=south] (-5,8);
            \draw[webline, overarc] (R) to[out=north, in=south] (5,8);
            \draw[webline] (-5,8) to[out=north, in=west] (0,13) to[out=east, in=north] (5,8);
            \draw[fill=lightgray, dashed] (0,8) circle [radius=2cm];
            \bdryline{(-10,0)}{(10,0)}{1.5cm}
            \bdryline{(-10,16)}{(10,16)}{-1.5cm}
            \fill (-5,0) circle [radius=1cm];
            \fill (5,0) circle [radius=1cm];
            \fill (-5,16) circle [radius=1cm];
            \fill (5,16) circle [radius=1cm];
        \ }
    }\\
    &=A_{E_1}^{-1}A_{E_2}^{-1}
    \mathord{
        \ \tikz[baseline=-.6ex, scale=.08, yshift=-6cm]{
            \coordinate (P) at (0,0);
            \coordinate (L) at (-5,0);
            \coordinate (R) at (5,0);
            \coordinate (C) at (0,5);
            \coordinate (T) at (0,10);
            \draw[webline, overarc] (R) to[out=north, in=south] (-5,8) to[out=north, in=south] ($(R)+(0,16)$);
            \draw[webline, overarc] (L) to[out=north, in=south] (5,8) to[out=north, in=south] ($(L)+(0,16)$);
            \draw[fill=lightgray, dashed] (0,8) circle [radius=2cm];
            \bdryline{(-10,0)}{(10,0)}{1.5cm}
            \bdryline{(-10,16)}{(10,16)}{-1.5cm}
            \fill (-5,0) circle [radius=1cm];
            \fill (5,0) circle [radius=1cm];
            \fill (-5,16) circle [radius=1cm];
            \fill (5,16) circle [radius=1cm];
        \ }
    }
    -q^{-1}A_{E_1}^{-1}A_{E_2}^{-1}
    \mathord{
        \ \tikz[baseline=-.6ex, scale=.08, yshift=-6cm]{
            \coordinate (P) at (0,0);
            \coordinate (L) at (-5,0);
            \coordinate (R) at (5,0);
            \coordinate (C) at (0,5);
            \coordinate (T) at (0,10);
            \draw[webline, overarc] (R) to[out=north, in=south] (-5,8) to[out=north, in=south] ($(L)+(0,16)$);
            \draw[webline, overarc] (L) to[out=north, in=south] (5,8) to[out=north, in=south] ($(R)+(0,16)$);
            \draw[fill=lightgray, dashed] (0,8) circle [radius=2cm];
            \bdryline{(-10,0)}{(10,0)}{1.5cm}
            \bdryline{(-10,16)}{(10,16)}{-1.5cm}
            \fill (-5,0) circle [radius=1cm];
            \fill (5,0) circle [radius=1cm];
            \fill (-5,16) circle [radius=1cm];
            \fill (5,16) circle [radius=1cm];
        \ }
    }
    -q^{2}A_{E_1}^{-1}
    \mathord{
        \ \tikz[baseline=-.6ex, scale=.08, yshift=-6cm]{
            \coordinate (P) at (0,0);
            \coordinate (L) at (-5,0);
            \coordinate (R) at (5,0);
            \coordinate (C) at (0,5);
            \coordinate (T) at (0,10);
            \draw[webline, overarc] (L) to[out=north, in=south] (-5,8);
            \draw[webline, overarc] (R) to[out=north, in=south] (5,8);
            \draw[webline] (-5,8) to[out=north, in=west] (0,13) to[out=east, in=north] (5,8);
            \draw[fill=lightgray, dashed] (0,8) circle [radius=2cm];
            \bdryline{(-10,0)}{(10,0)}{1.5cm}
            \bdryline{(-10,16)}{(10,16)}{-1.5cm}
            \fill (-5,0) circle [radius=1cm];
            \fill (5,0) circle [radius=1cm];
            \fill (-5,16) circle [radius=1cm];
            \fill (5,16) circle [radius=1cm];
        \ }
    }
    \end{align*}
    \caption{The expansion of a simple loop into a polynomial of ideal arcs, together possibly with additional coefficients $z_{j,\pm}$.}
    \label{fig:loop-exp}
\end{figure}



\subsection{Comparison with cluster algebras for marked surfaces with laminations}\label{subsec:comparison}

Fomin--Thurston gave a topological realization of arbitrary cluster algebras associated with marked surfaces of geometric type via \emph{multi-laminations}, which are tuples of laminations \cite{FT18}.
In this section, we compare our (quantum) cluster algebras realized inside the walled skein algebras to their (quantum) cluster algebras.

Let us briefly recall the Fomin--Thurston's construction, together with its straightforward generalization to the quantum case. 
A \emph{multi-lamination} is a tuple of integral $\cX$-laminations $\bL = (L_j \mid j \in J)$ for some finite index set $J$.
A marked surface with laminations is a pair $(\Sigma, \bL)$ of a marked surface and a multi-lamination on it.
For an multi-lamination $\bL$ with an index set $J$, we define $\bP_\bL := \mathrm{Trop}(u_j \,|\, j \in J)$.



Given a multi-lamination $\bL = (L_j \mid j \in J)$, 
we are going to construct a quantum cluster algebra $\mathscr{A}^q_{\Sigma, \bL}$ with coefficients in the skew-field $\Frac(\bZ_q\bP_\bL \otimes_{\bZ_q} \mathscr{S}^q_\Sigma)$, according to \cite{FT18}.
Here, notice the difference from the construction in the previous section, where the ambient skew-field was $\mathop{\mathrm{Frac}}\mathscr{S}_{\Sigma, \bW}^q$.
We use the variables $A_{\bL;\alpha}$ rescaled by the elements $u_j$ according to the shear coordinates of $\bL$,
following \cite[Definition 15.3]{FT18}.
Namely, we define the toric frame $\mathbf{A}^{\!\tri}_\bL: M^\tri \to \bT^\tri_\bL \subset \Frac(\bZ_q\bP_\bL \otimes_{\bZ_q} \mathscr{S}^q_\Sigma)$ by
\begin{align}\label{eq:rescaled_variable}
    \mathbf{A}^{\!\tri}_\bL(f_\alpha) = A_{\bL; \alpha} := \prod_{j \in J}u_j^{-a^\tri_\alpha(L_i)} A_\alpha \in \bZ_q\bP_\bL \otimes_{\bZ_q} \mathscr{S}^q_\Sigma.
\end{align}
Here $\bT^\tri_\bL$ is a based quantum torus associated with $(M^\tri, \Pi^\tri)$ over $\bZ_q \bP_\bL$.
Since $\bT^\tri_\bL = \bZ_q \bP_\bL \otimes_{\bZ_q} \bT^\tri$, we get $\mathop{\mathrm{Frac}} \bT^\tri_\bL=\Frac(\bZ_q\bP_\bL \otimes_{\bZ_q} \mathscr{S}^q_\Sigma)$ from \cref{lem:Muller} (1).

We define the coefficient tuple $\mathbf{p}^\tri_\bL = (p^{\tri, \pm}_{\bL;\alpha})_{\alpha \in \tri_\uf}$ by
\begin{align}\label{eq:coefficients_lamination}
    p^{\tri,\pm}_{\bL;\alpha} :=
    \prod_{j \in J} u_j^{[\pm x^\tri_\alpha(L_j)]_+} \in \bP_\bL.
\end{align}

\begin{lem}\label{lem:ex_rel_coeff}
The attachment $\sfs_{\Sigma, \bL}: \tri \mapsto \sfs^\tri_{\Sigma, \bL} = (\varepsilon^\tri, \Pi^\tri, \mathbf{A}^{\!\tri}_\bL, \mathbf{p}^\tri_\bL)$ of quantum seeds in $\Frac(\bZ_q\bP_\bL \otimes_{\bZ_q} \mathscr{S}^q_\Sigma)$ is an exchange pattern on $\Tri_\Sigma$.
\end{lem}

Here, we omit the proof since it is essentially given in \cite[Proof of Theorem 15.6]{FT18}. 
Let $\mathscr{A}^q_{\Sigma, \bL}$ denote the quantum cluster algebra associated with the exchange pattern $\sfs_{\Sigma, \bL}$.

\begin{ex}[Principal coefficients]\label{ex:principal}

Given an ideal triangulation $\tri$ of a marked surface $\Sigma$, we define the multi-lamination $\bL^+_\tri = (L_{\alpha,+} \mid \alpha \in \tri)$, where the lamination $L_{\alpha,+}$ consists of a single curve obtained from $\alpha$ by slightly sliding its endpoints on special points following the boundary orientation induced from $\Sigma$. 
The coefficients realized by $\bL^+_\tri$ is called the \emph{principal coefficients} with respect to $\tri$. 

We can also consider the lamination  $L_{\alpha,-}$ consisting of a single curve obtained from $\alpha$ by slightly sliding its endpoints on special points against the boundary orientation induced from $\Sigma$. Let us call the coefficients realized by $\bL_\tri^\pm:=(L_{\alpha,+},L_{\alpha,-} \mid \alpha \in \tri)$ the \emph{double principal coefficients}. 
\end{ex}

Recall that when $\sfC_j$ has no crossings for any $j \in J$, the corrections $L_{j, \pm} = \{ \gamma_{\pm} \mid \gamma \in \sfC_j \}$ define laminations.
Let us consider the multi-lamination $\bL(\bW) := (L_{j, \epsilon})_{j \in J,~\epsilon \in \{+,-\}}$. For example, we have $\bL(\bW_\tri)=\bL_\tri^\pm$ for the principal wall system (\cref{ex:wall_principal}). 
Recall that the coefficient tuples $\bp_\bW^\tri$ are normalized by \cref{lem:mut_in_skein}; moreover, we have $\bp^\tri_{\bW}=\bp^\tri_{\bL(\bW)}$ under the semifield isomorphism 
\begin{align}
    \bP_\bW=\mathrm{Trop}(z_{j,\epsilon}\mid j \in J,\ \epsilon \in \{+,-\}) \cong \mathrm{Trop}(u_{j,\epsilon}\mid (j,\epsilon) \in J_\pm)=\bP_{\bL_\pm}, \quad z_{j,\epsilon} \mapsto u_{j,\epsilon}.\label{eq:P_W=P_L_pm}
\end{align}
Compare \eqref{eq:coefficients_lamination} and \eqref{eq:coefficients_walls_normalized}.
Therefore, we get the following:

\begin{thm}\label{cor:wskein_cluster_lam}
If $\sfC_j$ has no crossings for any $j \in J$, then
we have a canonical $\bZP_{\bW}$-algebra isomorphism
\begin{align*}
    \CA^q_{\Sigma, \bW} \xrightarrow{\sim} \CA^q_{\Sigma, \bL(\bW)}
\end{align*}
such that $A_{\bW;\alpha} \mapsto A_{\bL(\bW);\alpha}$ for any ideal arc $\alpha$, and $p_{\bW;\alpha}^{\tri,\pm} \mapsto p_{\bL(\bW);\alpha}^{\tri,\pm}$ for any ideal triangulation $\tri$ and $\alpha \in \tri_\uf$. 
\end{thm}
In view of \cref{thm:wskein_cluster}, we get the following ``skein realization'' statement: $\CA^q_{\Sigma, \bL(\bW)} \subset \Skein{\Sigma, \bW}[\partial^{-1}]$ holds
in general, and the equality
\begin{align*}
    \CA^q_{\Sigma, \bL(\bW)} =\Skein{\Sigma, \bW}[\partial^{-1}]
\end{align*}
holds if $\Sigma$ has at least two marked points.

\begin{proof}
By their constructions, the ambient skew-fields of $\CA^q_{\Sigma, \bW}$ and 
$\CA^q_{\Sigma, \bL(\bW)}$ are $\Frac \Skein{\Sigma, \bW}$ and $\Frac (\bZP_{\bL(\bW)}\otimes_{\bZ_q} \mathscr{S}^q_{\Sigma})$, respectively. Fixing an ideal triangulation $\tri$ of $\Sigma$, we have isomorphisms 
\begin{itemize}
    \item $\Frac \Skein{\Sigma, \bW} \cong \Frac\bT_\bW^\tri$ by \cref{thm:quantum-torus};
    \item $\Frac (\bZP_{\bL(\bW)}\otimes_{\bZ_q} \mathscr{S}^q_{\Sigma}) \cong \Frac \bT_{\bL(\bW)}^\tri$ by the construction above.
\end{itemize}
Let 
\begin{align}\label{eq:tori_isom}
    \Frac \bT_\bW^\tri \xrightarrow{\sim} \Frac \bT_{\bL(\bW)}^\tri, \quad A_{\bW;\alpha} \mapsto A_{\bL(\bW);\alpha}
\end{align}
be the isomorphism between them. Indeed, they are both based quantum tori associated with $(M^\tri,\Pi^\tri)$. By \cref{lem:ex_rel_coeff,lem:mut_in_skein}, the cluster variables in both sides satisfy the same quantum exchange relation for a flip $\tri\setminus\{\kappa\}=\tri'\setminus\{\kappa'\}$. Hence the isomorphisms \eqref{eq:tori_isom} commute with quantum exchange relations and restricts to a $\bZ_q\bP_\bW$-algebra isomorphism $\CA^q_{\Sigma, \bW} \xrightarrow{\sim} \CA^q_{\Sigma, \bL(\bW)}$.
\end{proof}

\begin{rem}\label{rmk:compair_coeff}
By a similar argument to the proof of \cref{cor:wskein_cluster_lam}, we get an algebra isomorphism $\Skein{\Sigma, \bW} \cong \bZP_\bW\otimes_{\bZ_q} \mathscr{S}^q_{\Sigma}$ if $|\bM| \geq 2$. Indeed, they are generated by respective cluster variables in this case, hence the correspondence \eqref{eq:tori_isom} restricts to the desired isomorphism. 
\end{rem}

\subsection*{Skein realization of $\CA^q_{\Sigma, \bL}$.}
Observe that the multi-lamination $\bL(\bW)$ used in above
needs to contain the pair $L_{j,\pm}$, in particular does not realize all possible normalized coefficients. 
Let us begin with any multi-lamination $\bL = ({L}_j \mid j \in J)$.
We define the wall system $\bW(\bL) = (\sfC_\bL, J, \ell)$ as follows.
For each 
$\gamma \in L_j$, let $\xi_{\gamma}$ be a curve such that $(\xi_\gamma)_+ = \gamma$.
Then we set $\sfC_\bL := \{ \xi_\gamma \mid \gamma \in L_j,\ j \in J\}$ and $\ell(\xi_\gamma) := j$ if $\gamma \in L_j$. For example, $\bW(\bL_\tri^+)=\bW_\tri$ for the multi-lamination corresponding to the principal coefficients (\cref{ex:wall_principal,ex:principal}).

Let us consider the quotient
\begin{align}\label{eq:quotient_z-}
    \Skein{\Sigma, \bW(\bL)}\big|_{\boldsymbol{z}_{-}=1}:=\Skein{\Sigma, \bW(\bL)}/ (\mathcal{I}_-\cdot \Skein{\Sigma, \bW(\bL)}),
\end{align}
where $\mathcal{I}_- \subset \bZ_{q,\bW(\bL)}$ denote the ideal generated by $(z_{j,-}-1)$ for $j \in J$. We note that its base ring is $\bZ_{q,\bW(\bL)}/\mathcal{I}_- \cong \bZP_{\bL}$. 

\begin{thm}\label{thm:realization_any_lamination}
For any multi-lamination $\bL$ on $\Sigma$, we have a canonical inclusion $\CA^q_{\Sigma, \bL} \subset \Skein{\Sigma, \bW(\bL)}\big|_{\boldsymbol{z}_{-}=1}$ of $\bZP_{\bL}$-algebras. Moreover, the equality 
\begin{align*}
    \CA^q_{\Sigma, \bL} =\Skein{\Sigma, \bW(\bL)}\big|_{\boldsymbol{z}_{-}=1}[\partial^{-1}]
\end{align*}
holds if $\Sigma$ has at least two marked points. 
\end{thm}

\begin{proof}
Let $\bL_{\pm} := \bL(\bW(\bL)) = (L_j, L_j^- \mid j \in J)$, where $L_j^-:=\{(\xi_\gamma)_- \mid \gamma \in L_j\}$.
Since $(\sfC_\bL)_j$ has no crossings for any $j \in J$, we have the $\bZ_q \bP_{\bL_\pm}$-algebra isomorphism $\CA^q_{\Sigma, \bL_\pm} \cong \CA^q_{\Sigma,\bW(\bL)}$ by \cref{cor:wskein_cluster_lam}.
Under this identification, we have the inclusion
\begin{align}\label{eq:inclusion_pm}
    \CA^q_{\Sigma, \bL_\pm} \subset \Skein{\Sigma, \bW(\bL)}[\partial^{-1}]
\end{align}
by \cref{thm:wskein_cluster}.
This inclusion becomes an equality if $\Sigma$ has at least two marked points.
Here, they are considered to be $\bZP_{\bL_\pm}$-algebras via the identification $\bP_{\bW(\bL)} \cong \bP_{\bL_\pm}$ (see \eqref{eq:P_W=P_L_pm}).
Let us consider the semifield homomorphism 
\begin{align*}
    \psi: \bP_{\bL_\pm} \to \bP_{\bL}, \quad u_{j,+} \mapsto u_j, \quad u_{j,-} \mapsto 1.
\end{align*}
Then it induces a \emph{coefficient specialization} 
\begin{align*}
    \Psi: \CA^q_{\Sigma, \bL_\pm} \to \CA^q_{\Sigma, \bL}
\end{align*}
in the sense of \cite[Definition 12.1]{FZ-CA4}. Indeed, the criterion \cite[(12.1)]{FZ-CA4} for $\psi$ to be a coefficient specialization holds true, since $\psi$ is a semifield homomorphism in our case.  
Moreover, observe that the specialization $\psi$ exactly corresponds to $z_{j,-}=1$ in the skein algebra side via the inclusion \eqref{eq:inclusion_pm}. Therefore we have the commutative diagram
\begin{equation*}
    \begin{tikzcd}
    \CA^q_{\Sigma, \bL_\pm} \ar[r] \ar[d,"\Psi"'] & \Skein{\Sigma, \bW(\bL)}[\partial^{-1}] \ar[d]\\
    \CA^q_{\Sigma, \bL} \ar[r]& \Skein{\Sigma, \bW(\bL)}\big|_{\boldsymbol{z}_{-}=1}[\partial^{-1}],
    \end{tikzcd}
\end{equation*}
from which we get the desired assertion.  
\end{proof}

\begin{ex}[Base affine space $SL_4/N$]
Let $\Sigma$ be the hexagon. 
Then we have $\mathcal{O}(SL_4/N) \cong \mathscr{A}_{\Sigma,\bL}$, where the multi-lamination $\bL$ is given in \cite[Figure 45]{FT18}. The corresponding wall system $\bW(\bL)$ is exactly the one in \cref{ex:affine_SL4}. Typical cluster variables are represented by flag minors $\Delta_I((z_{ij})N):=\det(z_{ij} \mid i \in I,~j \leq |I|)$ for $I \subsetneq \{1,2,3,4\}$, $I \neq \emptyset$ but we also have a cluster variable $\Omega:=-\Delta_1\Delta_{234}+\Delta_2\Delta_{134}$ that is not a single flag minor. As given in \cite[Example 16.6]{FT18}, it satisfies the exchange relation
\begin{align*}
    \Delta_{23}\Omega = \Delta_{12}\Delta_{234}\Delta_3 + \Delta_{123}\Delta_{34}\Delta_2.
\end{align*}
In our skein realization, it is obtained from \eqref{eq:SL4_exchange} by specializing $A=1$, $z_{\bullet,-}=1$ and $\delta_\bullet=1$. 
\end{ex}

\subsection{Quasi-homomorphism of skein algebras}
In this section, we try to understand some quasi-homomorphisms of \cite{Fra16} in terms of skein algebras of walled surfaces.

For a quantum cluster algebra $\mathcal{A}$ with a coefficient semifield $\bP$, we say that two elements $A,A' \in \mathcal{A}$ are \emph{proportional} and write $A \asymp A'$ if there exists $\ell \in \bZ$ and $M \in \bP$ such that $A = q^{\ell/2} M A'$. 

\begin{dfn}
Let $\mathscr{A}_\sfs$ and $\mathscr{A}_{\overline{\sfs}}$ be quantum cluster algebras associated with exchange patterns $\sfs$ and $\overline{\sfs}$ on $n$-regular graphs $\mathbf{E}$ and $\overline{\mathbf{E}}$ with coefficient semifields $\bP$ and $\overline{\bP}$, respectively.
A $\bZP$-algebra homomorphism $\Psi: \mathscr{A}_\sfs \to \mathscr{A}_{\overline{\sfs}}$ satisfying $\Psi(\bP) \subset \overline{\bP}$ is called a \emph{quasi-homomorphism} from $\sfs$ to $\overline{\sfs}$ if there is a graph isomorphism $v \mapsto \overline{v}$ between the graphs $\mathbf{E}$ and $\overline{\mathbf{E}}$ inducing $\mathrm{st}(v) \xrightarrow{\sim} \mathrm{st}(\overline{v})$, $\alpha \mapsto \overline{\alpha}$ such that
\begin{align*}
    \Psi(A_\alpha^{(v)}) \asymp A^{(\overline{v})}_{\overline{\alpha}} \quad \mbox{and} \quad \Psi(\hat{y}^{(v)}_\alpha) = \hat{y}^{(\overline{v})}_{\overline{\alpha}}
\end{align*}
for all $\alpha \in \mathrm{st}(v)$. 
Here
\begin{align*}
    \hat{y}^{(v)}_\alpha := \frac{p^{(v),+}_\alpha}{p^{(v),-}_\alpha} \bigg[\prod_{\beta \in \mathrm{st}(v)} (A^{(v)}_\beta)^{\epsilon^{(v)}_{\alpha\beta}} \bigg].
\end{align*}
\end{dfn}

\begin{rem}
In \cite{CHL24}, Chang--Huang--Li have introduced the notion of \emph{quantum quasi-homomorphisms} between quantum cluster algebras. From our point, they treat the case $\bP=\{1\}$ but with frozen variables. While they allow rescalings by frozen variables, we only allow those by coefficients in the definition of proportionality. 
\end{rem}

\paragraph{\textbf{(1) Resolution of crossings.}}
Fix a wall system $\bW=(\sfC,J,\ell)$ on a surface $\Sigma$ having crossings.
Let $\bW'$ be a wall system obtained from $\bW$ by resolving some of the crossings of walls with a common label in any possible directions as shown below, keeping their labels: 
 \begin{align}
            \ \tikz[baseline=-.6ex, scale=.1]{
                \draw[dashed] (0,0) circle(5cm);
                \draw[wline] (-45:5) -- (135:5) node[olive,left,scale=0.8]{$j$};
                \draw[wline] (-135:5) -- (45:5)node[olive,right,scale=0.8]{$j$};
            \ }
            \mapsto 
            \ \tikz[baseline=-.6ex, scale=.1]{
                \draw[dashed] (0,0) circle(5cm);
                \draw[wline] (45:5) node[olive,right,scale=0.8]{$j$} to[out=south west, in=north west] (-45:5);
                \draw[wline] (-135:5) to[out=north east, in=south east] (135:5) node[olive,left,scale=0.8]{$j$};
            \ }
            \mbox{ or }\ \tikz[baseline=-.6ex, scale=.1]{
                \draw[dashed] (0,0) circle(5cm);
                \draw[wline] (-45:5) node[olive,right,scale=0.8]{$j$} to[out=north west, in=north east] (-135:5);
                \draw[wline] (135:5) node[olive,left,scale=0.8]{$j$} to[out=south east, in=south west] (45:5);
            \ }
    \end{align}
Since the smoothing is a local move, we have the linear map
\begin{align}\label{eq:rescaling_skein}
    \widetilde{\Psi}: \bZ_{q, \bW} \mathsf{Tang}(\Sigma, \bW) \to \bZ_{q, \bW} \mathsf{Tang}(\Sigma, \bW'),\quad
    [K]_\bW \mapsto [K]_{\bW'}.
\end{align}
Here, $[-]_\bW$ and $[-]_{\bW'}$ stand for the isotopy classes relative to $\bW$ and $\bW'$, respectively. 
Note that $\bP_\bW = \bP_{\bW'}$ in this case.


\begin{thm}\label{thm:qhom_sm}
For two wall systems $\bW$ and $\bW'$ related as above, the map $\widetilde{\Psi}$ induces a $\bZ_{q, \bW}$-algebra homomorphism
\begin{align*}
    \Psi: \Skein{\Sigma, \bW} \to \Skein{\Sigma, \bW'}.
\end{align*}
Moreover, if $\Sigma$ has at least two marked points and $\bW'$ is still taut, then the map $\Psi$ is a quasi-homomorphism 
from $\sfs_{\Sigma, \bW}$ to $\sfs_{\Sigma, \bW'}$ 
\end{thm}

\begin{proof}
First, we prove the well-definedness of the map $\Psi$.
We need to check that $\Psi$ preserves the defining relations \labelcref{rel:Kauffman,rel:trivial-loop,rel:elevation,rel:monogon,rel:wall-pass-int,rel:wall-pass-ext,rel:wall-R3} for the walled skein algebras. The relations except for \eqref{rel:wall-R3} is obviously preserved, since they do not involve intersections between walls. The following is the only non-trivial pattern of resolution concerning \eqref{rel:wall-R3}:
\begin{align*}
    \hspace{2cm}&\hspace{-2cm} 
    \widetilde{\Psi} \bigg(
    \ \tikz[baseline=-.6ex, scale=.1, yshift=-5cm]{
        \draw[webline, rounded corners] (0,0) -- (-5,5) -- (0,10);
        \draw[wline] (5,0) -- (-5,10);
        \draw[wline] (-5,0) -- (5,10);
        \node at (5,10) [right]{\scriptsize $j$};
        \node at (-5,10) [left]{\scriptsize $j$};
    }\ 
    -
    \ \tikz[baseline=-.6ex, scale=.1, yshift=-5cm]{
        \draw[webline, rounded corners] (0,0) -- (5,5) -- (0,10);
        \draw[wline] (5,0) -- (-5,10);
        \draw[wline] (-5,0) -- (5,10);
        \node at (5,10) [right]{\scriptsize $j$};
        \node at (-5,10) [left]{\scriptsize $j$};
    }\ 
    \bigg)\\
    &=
        \ \tikz[baseline=-.6ex, scale=.1, yshift=-5cm]{
        \draw[webline, rounded corners] (0,0) -- (-5,5) -- (0,10);
        \draw[wline, rounded corners] (-5,0) -- (-1,5) -- (-5,10);
        \draw[wline, rounded corners] (5,0) -- (1,5) -- (5,10);
        \node at (5,10) [right]{\scriptsize $j$};
        \node at (-5,10) [left]{\scriptsize $j$};
    }\ 
    -
    \ \tikz[baseline=-.6ex, scale=.1, yshift=-5cm]{
        \draw[webline, rounded corners] (0,0) -- (5,5) -- (0,10);
        \draw[wline, rounded corners] (-5,0) -- (-1,5) -- (-5,10);
        \draw[wline, rounded corners] (5,0) -- (1,5) -- (5,10);
        \node at (5,10) [right]{\scriptsize $j$};
        \node at (-5,10) [left]{\scriptsize $j$};
    }\\
    &=
    \bigg(
    \ \tikz[baseline=-.6ex, scale=.1, yshift=-5cm]{
        \draw[webline, rounded corners] (0,0) -- (-5,5) -- (0,10);
        \draw[wline, rounded corners] (-5,0) -- (-1,5) -- (-5,10);
        \draw[wline, rounded corners] (5,0) -- (1,5) -- (5,10);
        \node at (5,10) [right]{\scriptsize $j$};
        \node at (-5,10) [left]{\scriptsize $j$};
    }\ 
    -
    a_j \tikz[baseline=-.6ex, scale=.1, yshift=-5cm]{
        \draw[webline, rounded corners] (0,0) -- (0,10);
        \draw[wline, rounded corners] (-5,0) -- (-1,5) -- (-5,10);
        \draw[wline, rounded corners] (5,0) -- (1,5) -- (5,10);
        \node at (5,10) [right]{\scriptsize $j$};
        \node at (-5,10) [left]{\scriptsize $j$};
    }\ 
    \bigg)
    +
    \bigg(
    a_j \tikz[baseline=-.6ex, scale=.1, yshift=-5cm]{
        \draw[webline, rounded corners] (0,0) -- (0,10);
        \draw[wline, rounded corners] (-5,0) -- (-1,5) -- (-5,10);
        \draw[wline, rounded corners] (5,0) -- (1,5) -- (5,10);
        \node at (5,10) [right]{\scriptsize $j$};
        \node at (-5,10) [left]{\scriptsize $j$};
    }\ 
    -
    \ \tikz[baseline=-.6ex, scale=.1, yshift=-5cm]{
        \draw[webline, rounded corners] (0,0) -- (5,5) -- (0,10);
        \draw[wline, rounded corners] (-5,0) -- (-1,5) -- (-5,10);
        \draw[wline, rounded corners] (5,0) -- (1,5) -- (5,10);
        \node at (5,10) [right]{\scriptsize $j$};
        \node at (-5,10) [left]{\scriptsize $j$};
    }\ 
    \bigg),
\end{align*}
where the last two parenthesized terms belong to the defining ideal. Thus $\Psi$ is a well-defined as a $\bZ_{q,\bW}$-linear map. 

It is obvious from the definitions that the map $\Psi$ preserves the multiplication defined by stacking the framed tangles.

Assume that $\bW'$ is still taut. Then the exchange patterns $\sfs_{\Sigma, \bW}$ to $\sfs_{\Sigma, \bW'}$ are well-defined (\cref{rem:taut_cluster}). The proportionality condition $\Psi([\alpha_\bW]_\bW) = [\alpha_\bW]_{\bW'} \asymp [\alpha_{\bW'}]_{\bW'}$ for any arc $\alpha$ on $\Sigma$ holds since the difference between them is only the position with respective to $\bW'$ and hence produces only a product of coefficients.
Thus, $\Psi$ is a quasi-homomorphism from $\sfs_{\Sigma, \bW}$ to $\sfs_{\Sigma, \bW'}$ by \cite[Proposition 5.2]{Fra16}.
\end{proof}

\begin{ex}
Recall \cref{ex:non-normalized seed}, which corresponds to a non-normalized exchange pattern $\sfs_{\Sigma, \bW}$. Let us consider the wall system $\bW'$ obtained from $\bW$ by resolving the crossing of $\gamma_1$ and $\gamma_2$ as shown in the right of \cref{fig:resolution}. Then the resulting wall system $\bW'$ has no crossings, and hence the exchange pattern $\sfs_{\Sigma,\bW'}$ is normalized by \cref{lem:mut_in_skein}. 
We see that
\begin{align*}
    A_{\bW;\kappa} = [\kappa]_\bW \xmapsto{\Psi} [\kappa]_{\bW'}=  z_{0,+} [\overline{\kappa}]_{\bW'} = z_{0,+} A_{\bW';\kappa}.
\end{align*}
Similarly $\Psi(A_{\bW;\kappa'}) = z_{0,-} A_{\bW';\kappa'}$, and
and $\Psi(A_{\gamma;\bW})=A_{\gamma;\bW'}$ for $\gamma \in \{\alpha_1,\alpha_2,\beta_1,\beta_2\}$. The exchange relation corresponding to \eqref{eq:non-normalized relation} is
\begin{align*}
    A_{\kappa;\bW'}A_{\kappa';\bW'}=A_{\alpha_1;\bW'}A_{\alpha_2;\bW'} + A_{\beta_1;\bW'}A_{\beta_2;\bW'},
\end{align*}
which is normalized. 
\end{ex}

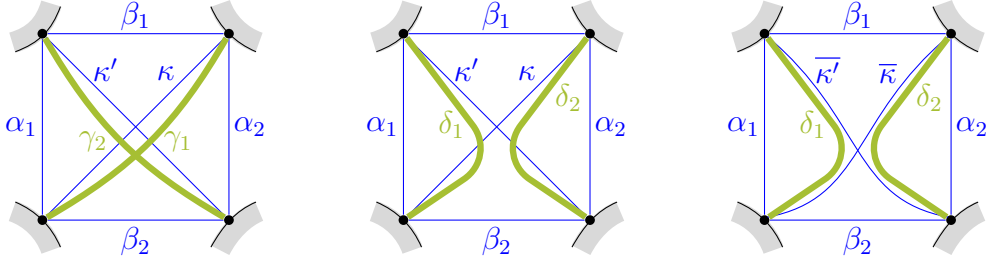
\begin{figure}[ht]
    \centering
    \begin{tikzpicture}[scale=0.8]
    \draw [gray!30, line width=8pt] (-2.55,2.9) .. controls (-2.2,3) and (-2,3.2) .. (-1.9,3.55);
    \draw [gray!30, line width=8pt] (0.9,3.55) .. controls (1,3.25) and (1.25,3) .. (1.55,2.9);
    \draw [gray!30, line width=8pt] (-2.55,0.1) .. controls (-2.25,0) and (-2,-0.25) .. (-1.9,-0.55);
    \draw [gray!30, line width=8pt] (0.9,-0.55) .. controls (1.05,-0.25) and (1.3,0) .. (1.55,0.1);
    \draw (-2.5,2.75) .. controls (-2.1,2.9) and (-1.9,3.1) .. (-1.75,3.5);
    \draw (0.75,3.5) .. controls (0.9,3.15) and (1.15,2.9) .. (1.5,2.75);
    \draw (1.5,0.25) .. controls (1.15,0.1) and (0.95,-0.15) .. (0.75,-0.5);
    \draw (-1.75,-0.5) .. controls (-1.9,-0.15) and (-2.15,0.1) .. (-2.5,0.25);
    \node [fill, circle, inner sep=1.3] (v1) at (-2.05,3.05) {};
    \node [fill, circle, inner sep=1.3] (v2) at (-2.05,-0.05) {};
    \node [fill, circle, inner sep=1.3] (v3) at (1.05,-0.05) {};
    \node [fill, circle, inner sep=1.3] (v4) at (1.05,3.05) {};
    \draw [blue] (v1) edge (v2);
    \draw [blue] (v2) edge (v3);
    \draw [blue] (v3) edge (v4);
    \draw [blue] (v1) edge (v4);
    \draw [blue] (v2) edge (v4);
    \draw [blue] (v1) edge (v3);
    \node [blue] at (-0.5,3.35) {$\beta_1$};
    \node [blue] at (-2.4,1.5) {$\alpha_1$};
    \node [blue] at (-0.5,-0.4) {$\beta_2$};
    \node [blue] at (1.4,1.5) {$\alpha_2$};
    \node [blue] at (0,2.4) {$\kappa$};
    \node [blue] at (-1,2.45) {$\kappa'$};
    \draw[wline] (v2) to[bend right=15] node[midway,right]{$\gamma_1$} (v4);
    \draw[wline] (v1) to[bend right=15] node[midway,left]{$\gamma_2$} (v3);

    \begin{scope}[xshift=6cm]
    \draw [gray!30, line width=8pt] (-2.55,2.9) .. controls (-2.2,3) and (-2,3.2) .. (-1.9,3.55);
    \draw [gray!30, line width=8pt] (0.9,3.55) .. controls (1,3.25) and (1.25,3) .. (1.55,2.9);
    \draw [gray!30, line width=8pt] (-2.55,0.1) .. controls (-2.25,0) and (-2,-0.25) .. (-1.9,-0.55);
    \draw [gray!30, line width=8pt] (0.9,-0.55) .. controls (1.05,-0.25) and (1.3,0) .. (1.55,0.1);
    \draw (-2.5,2.75) .. controls (-2.1,2.9) and (-1.9,3.1) .. (-1.75,3.5);
    \draw (0.75,3.5) .. controls (0.9,3.15) and (1.15,2.9) .. (1.5,2.75);
    \draw (1.5,0.25) .. controls (1.15,0.1) and (0.95,-0.15) .. (0.75,-0.5);
    \draw (-1.75,-0.5) .. controls (-1.9,-0.15) and (-2.15,0.1) .. (-2.5,0.25);
    \node [fill, circle, inner sep=1.3] (v1) at (-2.05,3.05) {};
    \node [fill, circle, inner sep=1.3] (v2) at (-2.05,-0.05) {};
    \node [fill, circle, inner sep=1.3] (v3) at (1.05,-0.05) {};
    \node [fill, circle, inner sep=1.3] (v4) at (1.05,3.05) {};
    \draw [blue] (v1) edge (v2);
    \draw [blue] (v2) edge (v3);
    \draw [blue] (v3) edge (v4);
    \draw [blue] (v1) edge (v4);
    \draw [blue] (v2) edge (v4);
    \draw [blue] (v1) edge (v3);
    \node [blue] at (-0.5,3.35) {$\beta_1$};
    \node [blue] at (-2.4,1.5) {$\alpha_1$};
    \node [blue] at (-0.5,-0.4) {$\beta_2$};
    \node [blue] at (1.4,1.5) {$\alpha_2$};
    \node [blue] at (0,2.4) {$\kappa$};
    \node [blue] at (-1,2.45) {$\kappa'$};
    \draw[wline,rounded corners=15pt] (v2) -- (-0.5,1) -- node[midway,below]{$\delta_1$} (v1);
    \draw[wline,rounded corners=15pt] (v3) -- (-0.5,1) --node[midway,right]{$\delta_2$} (v4);
    \end{scope}

    \begin{scope}[xshift=12cm]
    \draw [gray!30, line width=8pt] (-2.55,2.9) .. controls (-2.2,3) and (-2,3.2) .. (-1.9,3.55);
    \draw [gray!30, line width=8pt] (0.9,3.55) .. controls (1,3.25) and (1.25,3) .. (1.55,2.9);
    \draw [gray!30, line width=8pt] (-2.55,0.1) .. controls (-2.25,0) and (-2,-0.25) .. (-1.9,-0.55);
    \draw [gray!30, line width=8pt] (0.9,-0.55) .. controls (1.05,-0.25) and (1.3,0) .. (1.55,0.1);
    \draw (-2.5,2.75) .. controls (-2.1,2.9) and (-1.9,3.1) .. (-1.75,3.5);
    \draw (0.75,3.5) .. controls (0.9,3.15) and (1.15,2.9) .. (1.5,2.75);
    \draw (1.5,0.25) .. controls (1.15,0.1) and (0.95,-0.15) .. (0.75,-0.5);
    \draw (-1.75,-0.5) .. controls (-1.9,-0.15) and (-2.15,0.1) .. (-2.5,0.25);
    \node [fill, circle, inner sep=1.3] (v1) at (-2.05,3.05) {};
    \node [fill, circle, inner sep=1.3] (v2) at (-2.05,-0.05) {};
    \node [fill, circle, inner sep=1.3] (v3) at (1.05,-0.05) {};
    \node [fill, circle, inner sep=1.3] (v4) at (1.05,3.05) {};
    \draw [blue] (v1) edge (v2);
    \draw [blue] (v2) edge (v3);
    \draw [blue] (v3) edge (v4);
    \draw [blue] (v1) edge (v4);
    \draw [blue] (-2,0) .. controls (-0.4,0.3) and (-0.6,1.65) .. (1,3);
    \draw [blue] (-2,3) .. controls (-0.4,1.65) and (-0.6,0.3) .. (1,0);
    \node [blue] at (-0.5,3.35) {$\beta_1$};
    \node [blue] at (-2.4,1.5) {$\alpha_1$};
    \node [blue] at (-0.5,-0.4) {$\beta_2$};
    \node [blue] at (1.4,1.5) {$\alpha_2$};
    \node [blue] at (0,2.4) {$\overline{\kappa}$};
    \node [blue] at (-1,2.45) {$\overline{\kappa'}$};
    \draw[wline,rounded corners=15pt] (v2) -- (-0.5,1) -- node[midway,below]{$\delta_1$} (v1);
    \draw[wline,rounded corners=15pt] (v3) -- (-0.5,1) --node[midway,right]{$\delta_2$} (v4);
    \end{scope}
    \end{tikzpicture}
    \caption{A resolution of wall system, where $\ell(\gamma_1)=\ell(\gamma_2)=\ell(\delta_1)=\ell(\delta_2)$.}
    \label{fig:resolution}
\end{figure}

\paragraph{\textbf{(2) Specialization from the principal wall.}}
Let us consider the \emph{principal wall system} 
$\bW_\tri := (\mathsf{C}_\tri, \tri_\uf, \ell_\tri)$, where $\mathsf{C}_\tri$ consists of one arc parallel to each edge of $\tri_\uf$, and the labeling $\ell_\tri: \mathsf{C}_\tri \to \tri_\uf$ given by the tautological bijection. 
We note that each $\mathsf{C}_\alpha$ consists of a single arc parallel to the edge $\alpha \in \tri_\uf$.
Let $z_{\alpha,\pm}^\tri$ denote the generators of $\bZ_{A, \bW_{\tri}}$. 
The wall system $\bW_\tri$ corresponds to the double principal coefficients (\cref{ex:principal}). 

\begin{thm}\label{thm:specialization_prin}
Let $\bW = (\mathsf{C}, J, \ell)$ be a taut wall system on $\Sigma$,
and let $\psi: \bZ_{q, \bW_{\tri_\uf}} \to \bZ_{q, \bW}$ be the ring homomorphism defined by
\begin{align}\label{eq:prin_psi}
    \psi(z_{\alpha,\pm}^\tri) := \prod_{\xi \in \mathsf{C}} z_{\ell(\xi),+}^{[\mp x_\alpha^\tri(\gamma_{\xi,+})]_+} z_{\ell(\xi),-}^{[\mp x_\alpha^\tri(\gamma_{\xi,-})]_+}
\end{align}
for $\alpha \in \tri_\uf$. 
Then, we have a $\bZ_{q,\bW_\tri}$-algebra homomorphism $\Psi=\Psi_\tri: \Skein{\Sigma, \bW_\tri} \to \Skein{\Sigma, \bW}$ such that it is a quasi-homomorphism from $\sfs_{\Sigma, \bW_\tri}$ to $\sfs_{\Sigma, \bW}$.
\end{thm}

\begin{proof}
Define $\Psi([\alpha_{\bW_\tri}]_{\bW_\tri}) := [\alpha_\bW]_\bW$ for $\alpha \in \tri$
and extend it to a unique $\bZ_{q,\bW_\tri}$-algebra homomorphism 
\begin{align*}
    \Psi: \Frac \bT_{\bW_\tri}^\tri \to \Frac \bT_\bW^\tri.
\end{align*}
Recall
$\Frac \bT_{\bW_\tri}^\tri \cong \mathop{\mathrm{Frac}}\Skein{\Sigma, \bW_\tri}$ and $\Frac \bT_\bW^\tri \cong \mathop{\mathrm{Frac}}\Skein{\Sigma, \bW}$ from \cref{thm:quantum-torus}.

We then claim that $\Psi$ satisfies
$\Psi(\hat{y}^\tri_{\bW_\tri; \alpha}) = \hat{y}^\tri_{\bW;\alpha}$ for $\alpha \in \tri_\uf$.
We can check it directly:
\begin{align*}
    \hat{y}^\tri_{\bW_\tri; \alpha} &=
    \frac{p^{\tri,+}_{\bW_\tri; \alpha}}{p^{\tri,-}_{\bW_\tri; \alpha}} \bigg[\prod_{\beta \in \tri} [\alpha_{\bW_\tri}]^{\varepsilon^\tri_{\alpha \beta}}_{\bW_\tri}\bigg]\\
    &= \frac{z^\tri_{\alpha, -}}{z^\tri_{\alpha, +}} \bigg[\prod_{\beta \in \tri} [\alpha_{\bW_\tri}]^{\varepsilon^\tri_{\alpha \beta}}_{\bW_\tri}\bigg]\\
    &\xmapsto{\Psi} \prod_{\xi \in \mathsf{C}} \frac{z_{\ell(\xi),+}^{[x_\alpha^\tri(\gamma_{\xi,+})]_+} z_{\ell(\xi),-}^{[x_\alpha^\tri(\gamma_{\xi,-})]_+}}{z_{\ell(\xi),+}^{[-x_\alpha^\tri(\gamma_{\xi,+})]_+} z_{\ell(\xi),-}^{[-x_\alpha^\tri(\gamma_{\xi,-})]_+}} \bigg[\prod_{\beta \in \tri} [\alpha_{\bW}]^{\varepsilon^\tri_{\alpha \beta}}_{\bW}\bigg]
    = \hat{y}^\tri_{\bW;\alpha}.
\end{align*}
Here, the second equation follows from $[x_\alpha^\tri(\gamma_{\kappa, -\epsilon})]_\sigma = \delta_{\alpha, \kappa} \delta_{\epsilon, \sigma}$ for $\kappa \in \tri_\uf$ and $\epsilon, \sigma \in \{+,-\}$. (Note that the definition of $\gamma_{\xi, \pm}$, see \cref{fig:deform_wall_lam}.)
Since $\Psi([\alpha_{\bW_\tri}]_{\bW_\tri}) \asymp [\alpha_\bW]_\bW$ for $\alpha \in \tri$, $\Psi$ is a quasi-homomorphism from $\sfs_{\Sigma, \bW_\tri}$ to $\sfs_{\Sigma, \bW}$ by \cite[Proposition 3.2]{Fra16}.

Therefore, $\Psi([\gamma_{\bW_\tri}]_{\bW_\tri}) \asymp [\gamma_\bW]_\bW$ for any ideal arc $\gamma$
and hence $\Psi$ restricts to $\Psi: \Skein{\Sigma, \bW_\tri} \to \Skein{\Sigma, \bW}$.
\end{proof}

For any wall system $\bW$, let $\bZ^+_{q,\bW} \subset \bZ_{q,\bW}$ be the sub-monoid consisting of elements of the form $\sum_{\lambda,\mu,\nu} c_{\lambda,\mu,\nu} q^{\lambda/2} \prod_{j \in J} z_{j,+}^{\mu_j} z_{j,-}^{\nu_j}$ with $c_{\lambda,\mu,\nu} \geq 0$.
A $\bZ_{q,\bW}$-basis $\mathbf{B}=\{B_\lambda \mid \lambda \in \Lambda\}$ of $\Skein{\Sigma, \bW}$ is said to be \emph{positive} if its structure constants belong to $\bZ^+_{q,\bW}$. 
\begin{cor}\label{cor:positive_bases}
If $\mathbf{B}=\{B_\lambda \mid \lambda \in \Lambda\}$ is a positive $\bZ_{q,\bW_\tri}$-basis of $\Skein{\Sigma, \bW_\tri}$, then $(\Psi_\tri)_\ast\mathbf{B}:=\{\Psi_\tri(B_\lambda) \mid \lambda \in \Lambda\}$ is a positive $\bZ_{q,\bW}$-basis of $\Skein{\Sigma, \bW}$. 
\end{cor}

\section{Minimal position of curves in a walled surface}\label{sec:minimal}
In this section, we provide a proof of \cref{thm:minimal-multicurve}.
The proof will be given by combinatorial argument about a minimal position of curves relative to a wall system $\bW$.
Observe that the wall-passing relations in \cref{def:wall-pass} realize any homotopy of an arc or loop on $\Sigma$ with no walls up to coefficients.
Thus, any tangle diagram $D$ in $(\Sigma, \bW)$ can be written as a sum of simple multicurves on the underlying surface $\Sigma$.
We show that there is a surjective map from the set of \emph{$\bW$-minimal} simple multicurves to $\mathsf{SMulti}(\Sigma)$ such that flip moves and double-point jumpings (\cref{fig:elementary-moves}) acts transitively on each fiber.

\begin{rem}
    The following discussion on the minimal position works only for a marked surface with a taut wall system.
    In what follows, we will always assume that a walled surface $(\Sigma, \bW)$ has a taut wall system otherwise specified. 
\end{rem}

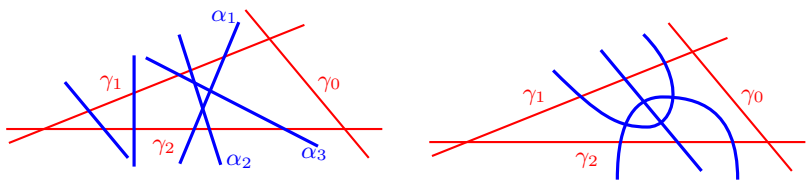
\begin{figure}
    \centering
    \begin{tikzpicture}[scale=.2]
    \coordinate (A) at (10,0);
    \coordinate (B) at (5,6);
    \coordinate (C) at (-10,0);
    \draw[red, thick, shorten >=-.5cm, shorten <=-.5cm] (A) -- (B);
    \draw[red, thick, shorten >=-.5cm, shorten <=-.5cm] (B) -- (C);
    \draw[red, thick, shorten >=-.5cm, shorten <=-.5cm] (C) -- (A);
    \draw[blue, very thick, shorten >=-.5cm, shorten <=-.5cm] ($(C)!.55!(A)$) -- ($(C)!.65!(B)$);
    \draw[blue, very thick, shorten >=-.5cm, shorten <=-.5cm] ($(C)!.3!(A)$) -- ($(C)!.4!(B)$);
    \draw[blue, very thick, shorten >=-.5cm, shorten <=-.5cm] ($(C)!.8!(A)$) -- ($(C)!.6!(B)$);
    \draw[blue, very thick, shorten >=-.5cm, shorten <=-.5cm] ($(C)!.5!(A)$) -- ($(C)!.8!(B)$);
    \draw[blue, very thick, shorten >=-.5cm, shorten <=-.5cm] ($(C)!.2!(A)$) -- ($(C)!.2!(B)$);
    \node at ($(A)!.5!(B)$) [red, right]{\scriptsize $\gamma_0$};
    \node at ($(B)!.7!(C)$) [red, above]{\scriptsize $\gamma_1$};
    \node at ($(C)!.4!(A)$) [red, below]{\scriptsize $\gamma_2$};
    \node at ($(C)!.8!(B)$) [above=.3cm, blue]{\scriptsize $\alpha_1$};
    \node at ($(C)!.65!(A)$) [below=.2cm, blue]{\scriptsize $\alpha_2$};
    \node at ($(C)!.9!(A)$) [below=.1cm, blue]{\scriptsize $\alpha_3$};
\end{tikzpicture}
\hspace{3em}
\begin{tikzpicture}[scale=.2]
    \coordinate (A) at (10,0);
    \coordinate (B) at (5,6);
    \coordinate (C) at (-10,0);
    \draw[red, thick, shorten >=-.5cm, shorten <=-.5cm] (A) -- (B);
    \draw[red, thick, shorten >=-.5cm, shorten <=-.5cm] (B) -- (C);
    \draw[red, thick, shorten >=-.5cm, shorten <=-.5cm] (C) -- (A);
    \draw[blue, very thick, shorten >=-.5cm, shorten <=-.5cm] ($(C)!.5!(A)$) to[out=north, in=west] (3,3) to[out=east, in=north] ($(C)!.9!(A)$);
    \draw[blue, very thick, shorten >=-.5cm, shorten <=-.5cm] ($(C)!.5!(B)$) to[out=south east, in=west] (2,1) to[out=east, in=south east] ($(C)!.9!(B)$);
    \draw[blue, very thick, shorten >=-.5cm, shorten <=-.5cm] ($(C)!.7!(A)$) -- ($(C)!.7!(B)$);
    \node at ($(A)!.5!(B)$) [red, right]{\scriptsize $\gamma_0$};
    \node at ($(B)!.7!(C)$) [red, above]{\scriptsize $\gamma_1$};
    \node at ($(C)!.4!(A)$) [red, below]{\scriptsize $\gamma_2$};
\end{tikzpicture}
    \caption{Left: any pair of blue arcs has at most one intersection point in the triangle consisting of $\gamma_0$, $\gamma_1$, and $\gamma_2$. $\gamma_0$ does not intersect any blue arcs. $\alpha_1$, $\alpha_2$, and $\gamma_2$ make an innermost triangle in \cref{lem:innermost-triagle}. $\alpha_2$, $\alpha_3$, and $\gamma_1$ also make another innermost triangle. Right: there is no innermost triangle adjacent to $\gamma_1$ or $\gamma_2$.}
    \label{fig:HSlem}
\end{figure}

We prepare a fundamental lemma on curves in a surface given by Hass--Scott \cite{HS}.
\begin{lem}[{\cite[Lemma~1.2]{HS}}]\label{lem:innermost-triagle}
    Let us consider a triangle $T$ surrounded by three arcs $\gamma_0,\gamma_1,\gamma_2$ and several embedded arcs across $T$.
    If any two of the embedded arcs have at most one intersection point in $T$ and any of them does not touch $\gamma_0$ in $T$, then $\gamma_i$ has an innermost triangle adjacent to it for each $i=1,2$. See \cref{fig:HSlem}.
\end{lem}

Based on this lemma, Hass--Scott obtained a lemma about an innermost triangle in a bigon.
We rewrite it in the form which fits in with our situation.
In a walled surface $(\Sigma, \bW)$ with $\bW = (\sfC, J, \ell)$, a bigon $B$ bounded by a wall $\gamma \in \sfC$ and a curve $\alpha$ is said to be \emph{innermost} if other walls in $\sfC$ bound no bigons inside $B$. 

\begin{figure}
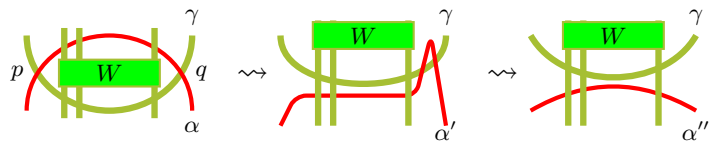

    \[
        \tikz[baseline=-.6ex,scale=.1]{
            \draw[wline] (-11,5) to[out=south,in=west] (0,-5) to[out=east, in=south] (11,5);
            \draw[webline] (-11,-5) to[out=north,in=west] (0,5) to[out=east, in=north] (11,-5);
            \begin{scope}
                \draw[wline] (-6,6) -- (-6,-6);
                \draw[wline] (-4,6) -- (-4,-6);
                \draw[wline] (6,6) -- (6,-6);
                \node[draw=olive, thick, fill=green, minimum height=.2cm, text width=1.2cm, text centered, inner sep=2pt] at (0,0) {\scriptsize $W$};
            \end{scope}
            \node at (11,5) [above]{\scriptsize $\gamma$};
            \node at (11,-5) [below]{\scriptsize $\alpha$};
            \node at (-10,0) [left]{\scriptsize $p$};
            \node at (10,0) [right]{\scriptsize $q$};
        }\ 
        \rightsquigarrow
        \tikz[baseline=-.6ex,scale=.1]{
            \draw[wline] (-11,5) to[bend right=90] (11,5);
            \draw[webline,rounded corners] (-11,-7) -- (-9,-3) -- (7,-3) -- (9,5) -- (11,-7);
            \begin{scope}
                \draw[wline] (-6,7) -- (-6,-7);
                \draw[wline] (-4,7) -- (-4,-7);
                \draw[wline] (6,7) -- (6,-7);
                \node[draw=olive, thick, fill=green, minimum height=.2cm, text width=1.2cm, text centered, inner sep=2pt] at (0,5) {\scriptsize $W$};
            \end{scope}
            \node at (11,5) [above]{\scriptsize $\gamma$};
            \node at (11,-5) [below]{\scriptsize $\alpha'$};
        }\ 
        \rightsquigarrow
        \tikz[baseline=-.6ex,scale=.1]{
            \draw[wline] (-11,5) to[bend right=60] (11,5);
            \draw[webline] (-11,-5) to[bend left] (11,-5);
            \begin{scope}
                \draw[wline] (-6,7) -- (-6,-7);
                \draw[wline] (-4,7) -- (-4,-7);
                \draw[wline] (6,7) -- (6,-7);
                \node[draw=olive, thick, fill=green, minimum height=.2cm, text width=1.2cm, text centered, inner sep=2pt] at (0,5) {\scriptsize $W$};
            \end{scope}
            \node at (11,5) [above]{\scriptsize $\gamma$};
            \node at (11,-5) [below]{\scriptsize $\alpha''$};
        }\ 
    \]
    \caption{The deformation of $\alpha$ to $\alpha''$ reduces an innermost bigon $B$ with vertices $p,q$. Here $W$ is a part of the taut wall system $\bW$ such that any wall passes from one side of $B$ to the other side. It can be realized by a sequence of double-point jumpings and one bigon reduction move.}
    \label{fig:detour-move}
\end{figure}

\begin{lem}[{\cite[Lemma~1.4]{HS}}]\label{lem:innermost-bigon}
    Let $\Sigma$ be a marked surface with a taut wall system $\bW=(\sfC,J,\ell)$. 
    Let $B$ be an innermost bigon bounded by a wall $\gamma \in \sfC$ and a curve $\alpha$.
    Then, there exists a sequence of double-point jumpings from $\alpha$ to $\alpha'$ such that no walls of $\bW$ pass through the resulting bigon $B'$ bounded by $\gamma$ and $\alpha'$.
    Moreover, one can reduce the innermost bigon $B'$ using a bigon reduction move from $\alpha'$ to $\alpha''$, see \cref{fig:detour-move}.
\end{lem}
\begin{proof}
    By the taut condition of $\bW$ and the innermost condition of $B$, any walls of $\bW\cap B$ pass from $\gamma$ to $\alpha$ in $B$. Moreover, any two walls in $\bW\cap B$ have at most one intersection point in $B$ due to the tautness of $\bW$.
    We prove that $\alpha$ can be deformed into $\alpha'$ fixing $q$ by induction on the number $n$ of all the double points in $B$. (We also count intersections on the boundary of $B$.)

    The assertion holds when $n=0$ by definition.
    For $n>0$, we start at $p$ and go along $\gamma$ until arriving at the first intersection point of $\gamma$ and another wall $\gamma_1$. 
    Then $\gamma$, $\gamma_1$, and $\alpha$ form a triangle $T$ such that $\gamma$ has no double point with other walls across $T$.
    Hence, $\alpha$ has an adjacent innermost triangle in $T$ by \cref{lem:innermost-triagle}.
    We remark that the innermost triangle might be $T$ itself. Then, one can sweep out a double point from $B$ by a double-point jumping at the innermost triangle.
    Then, we obtain a new bigon with $n-1$ double points.
\end{proof}

\begin{figure}
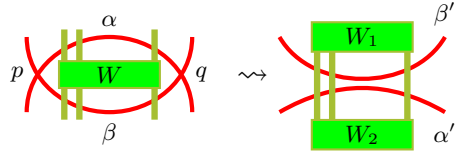

    \centering
        \[
        \tikz[baseline=-.6ex,scale=.1]{
            \coordinate (U) at (0,5);
            \coordinate (D) at (0,-5);
            \draw[webline] (-11,5) to[out=south,in=west] (0,-5) to[out=east, in=south] (11,5);
            \draw[webline] (-11,-5) to[out=north,in=west] (0,5) to[out=east, in=north] (11,-5);
            \begin{scope}
                \draw[wline] (-6,6) -- (-6,-6);
                \draw[wline] (-4,6) -- (-4,-6);
                \draw[wline] (6,6) -- (6,-6);
                \node[draw=olive, thick, fill=green, minimum height=.2cm, text width=1.2cm, text centered, inner sep=2pt] at (0,0) {\scriptsize $W$};
            \end{scope}
            \node at (0,5) [above]{\scriptsize $\alpha$};
            \node at (0,-5) [below]{\scriptsize $\beta$};
            \node at (-10,0) [left]{\scriptsize $p$};
            \node at (10,0) [right]{\scriptsize $q$};
        }\ 
        \rightsquigarrow
        \tikz[baseline=-.6ex,scale=.1]{
            \draw[webline] (-11,5) to[bend right=60] (11,5);
            \draw[webline] (-11,-5) to[bend left] (11,-5);
            \begin{scope}
                \draw[wline] (-6,7) -- (-6,-7);
                \draw[wline] (-4,7) -- (-4,-7);
                \draw[wline] (6,7) -- (6,-7);
                \node[draw=olive, thick, fill=green, minimum height=.2cm, text width=1.2cm, text centered, inner sep=2pt] at (0,5) {\scriptsize $W_1$};
                \node[draw=olive, thick, fill=green, minimum height=.2cm, text width=1.2cm, text centered, inner sep=2pt] at (0,-8) {\scriptsize $W_2$};
            \end{scope}
            \node at (11,5) [above]{\scriptsize $\beta'$};
            \node at (11,-5) [below]{\scriptsize $\alpha'$};
        }\ 
    \]
    \caption{$W$ is a part of the wall $\bW$ which is divided into $W_1$ and $W_2$. A sequence of double-point jumping and bigon reduction moves reduce the bigon bounded by $\alpha$ and $\beta$.}
    \label{fig:red-bigon}
\end{figure}

\begin{prop}\label{prop:reduction-bigon}
    Let $\bW$ be a taut wall system on a marked surface $\Sigma$, and $\alpha$ and $\beta$ any $\bW$-transverse simple arc diagrams.
    If subarcs of $\alpha$ and $\beta$ bound a bigon $B$ in $\Sigma$, then the bigon $B$ is reduced by a sequence of double-point jumpings and bigon reduction moves in a neighborhood of $B$, see \cref{fig:red-bigon}.
    Moreover, if $\alpha$ and $\beta$ are $\bW$-minimal, then the deformation is associated without bigon reduction moves.
\end{prop}
\begin{proof}
    If a wall in $\bW$ and a boundary of $B$ (which is a subarc of $\alpha$ or $\beta$) bound an (innermost) bigon, then we can remove it by \cref{lem:innermost-bigon} (see also \cref{fig:detour-move}).
    Thus, we only have to consider the case that $B$ is innermost.
    Then, one can apply the same argument in the proof of \cref{lem:innermost-bigon} and the resulting deformation is what we want.
    The bigon reduction moves reduce the number of intersection points with $\bW$. 
    Hence, the above deformation consists only of double-point jumpings if $\alpha$ and $\beta$ are $\bW$-minimal.
\end{proof}

\begin{lem}\label{lem:innermost-annulus}
    Let $\bW = (\sfC, J, \ell)$ be a taut wall system on a marked surface $\Sigma$.
    For any wall $\gamma \in \sfC_{\mathrm{loop}}$ and a simple loop diagram $\alpha$ bounding together an annulus $A$, there exists a sequence of bigon reduction moves and double-point jumpings of $\alpha$ which deforms $\alpha$ into $\alpha'$ as the middle picture in \cref{fig:loop-flip}.
\end{lem}

\begin{proof}
    Firstly, one can remove the innermost bigons adjacent to $\alpha$ by \cref{lem:innermost-bigon}.
    Hence we can assume that every wall in $A$ passes through from $\gamma$ to $\alpha$.
    We prove the claim by induction on the number of double points in the interior of $A$.
    One can obtain $\alpha'$ by applying a double-point jumping when $\bW$ has only one double point in the interior of $A$.
    When $\bW$ has $n>1$ double points in the interior of $A$, we choose an intersection point $p$ of a wall $\eta_0 \in \sfC$ with $\alpha$.
    Start at $p$ and go along $\eta_0$ until arriving at the intersection point with another wall $\eta_1$. 
    We remark that such a pair $\eta_0$ and $\eta_1$ exist because $n$ is positive.
    Then $\eta_0$, $\eta_1$, and $\alpha$ form a triangle $T$. 
    No other walls in $T$ intersect with $\eta_0$.
    The existence of $T$ is guaranteed by removing all bigons adjacent to $\alpha$ in the first step.
    There exists an innermost triangle adjacent to $\alpha$ in $T$ by \cref{lem:innermost-triagle}, and one can apply double-point jumping at this innermost triangle.
    This deformation reduces the number of double points of $\sfC$ in the interior of $A$ into $n-1$. 
\end{proof}

\begin{figure}
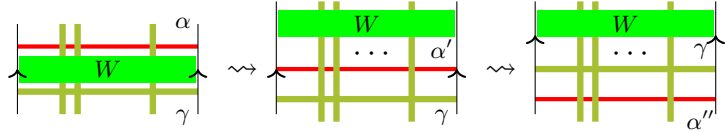

    \begin{align*}
        \tikz[baseline=-.6ex,scale=.1]{
            \draw[webline] (-12,3) -- (12,3);
            \draw[wline] (-12,-3) -- (12,-3);
            \begin{scope}
                \draw[wline] (-6,6) -- (-6,-6);
                \draw[wline] (-4,6) -- (-4,-6);
                \draw[wline] (6,6) -- (6,-6);
                \node[fill=green, minimum height=.2cm, text width=2.2cm, text centered, inner sep=2pt] at (0,0) {\scriptsize $W$};
            \end{scope}
            \draw[->-={.5}{black}] (-12,-6) -- (-12,6);
            \draw[->-={.5}{black}] (12,-6) -- (12,6);
            \node at (10,-4) [below]{\scriptsize $\gamma$};
            \node at (10,4) [above]{\scriptsize $\alpha$};
        }\ 
        \rightsquigarrow
        \tikz[baseline=-.6ex,scale=.1]{
            \draw[webline] (-12,0) -- (12,0);
            \draw[wline] (-12,-4) -- (12,-4);
            \draw[->-={.4}{black}] (-12,-6) -- (-12,9);
            \draw[->-={.4}{black}] (12,-6) -- (12,9);
            \begin{scope}
                \draw[wline] (-6,9) -- (-6,-7);
                \draw[wline] (-4,9) -- (-4,-7);
                \draw[wline] (6,9) -- (6,-7);
                \node at (1,2) {$\cdots$};
                \node[thick, fill=green, minimum height=.2cm, text width=2.2cm, text centered, inner sep=2pt] at (0,6) {\scriptsize $W$};
            \end{scope}
            \node at (10,-4) [below]{\scriptsize $\gamma$};
            \node at (10,0) [above]{\scriptsize $\alpha'$};
        }\ 
        \rightsquigarrow
        \tikz[baseline=-.6ex,scale=.1]{
            \draw[wline] (-12,0) -- (12,0);
            \draw[webline] (-12,-4) -- (12,-4);
            \draw[->-={.7}{black}] (-12,-6) -- (-12,9);
            \draw[->-={.7}{black}] (12,-6) -- (12,9);
            \begin{scope}
                \draw[wline] (-6,9) -- (-6,-7);
                \draw[wline] (-4,9) -- (-4,-7);
                \draw[wline] (6,9) -- (6,-7);
                \node at (1,2) {$\cdots$};
                \node[thick, fill=green, minimum height=.2cm, text width=2.2cm, text centered, inner sep=2pt] at (0,6) {\scriptsize $W$};
            \end{scope}
            \node at (10,0) [above]{\scriptsize $\gamma$};
            \node at (10,-4) [below]{\scriptsize $\alpha''$};
        }\ 
    \end{align*}
    \caption{These pictures represent the annulus $A$ by identifying the left and right sides. The first deformation is in \cref{lem:innermost-annulus}, and the last deformation does not change the representing web in $\SK{\Sigma, \bW}$.}
    \label{fig:loop-flip}
\end{figure}

\begin{figure}
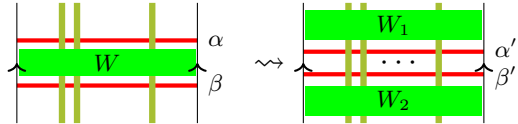

    \centering
    \[
        \tikz[baseline=-.6ex,scale=.1]{
            \draw[webline] (-12,3) -- (12,3);
            \draw[webline] (-12,-3) -- (12,-3);
            \begin{scope}
                \draw[wline] (-6,8) -- (-6,-8);
                \draw[wline] (-4,8) -- (-4,-8);
                \draw[wline] (6,8) -- (6,-8);
                \node[fill=green, minimum height=.2cm, text width=2.2cm, text centered, inner sep=2pt] at (0,0) {\scriptsize $W$};
            \end{scope}
            \draw[->-={.5}{black}] (-12,-8) -- (-12,8);
            \draw[->-={.5}{black}] (12,-8) -- (12,8);
            \node at (12,-3) [right]{\scriptsize $\beta$};
            \node at (12,3) [right]{\scriptsize $\alpha$};
        }\ 
        \rightsquigarrow
        \tikz[baseline=-.6ex,scale=.1]{
            \draw[webline] (-12,1.5) -- (12,1.5);
            \draw[webline] (-12,-1.5) -- (12,-1.5);
            \begin{scope}
                \draw[wline] (-6,8) -- (-6,-8);
                \draw[wline] (-4,8) -- (-4,-8);
                \draw[wline] (6,8) -- (6,-8);
                \node at (1,0) {$\cdots$};
                \node[fill=green, minimum height=.2cm, text width=2.2cm, text centered, inner sep=2pt] at (0,5) {\scriptsize $W_1$};
                \node[fill=green, minimum height=.2cm, text width=2.2cm, text centered, inner sep=2pt] at (0,-5) {\scriptsize $W_2$};
            \end{scope}
            \draw[->-={.5}{black}] (-12,-8) -- (-12,8);
            \draw[->-={.5}{black}] (12,-8) -- (12,8);
            \node at (12,-2) [right]{\scriptsize $\beta'$};
            \node at (12,2) [right]{\scriptsize $\alpha'$};
        }\ 
    \]
    \caption{$W$ is a part of the wall $\bW$ in an annulus $A$ which is divided into $W_1$ and $W_2$. A sequence of double-point jumping, bigon reduction moves, and flip moves sweep $W$ out of the annulus bounded by $\alpha$ and $\beta$.}
    \label{fig:reduct-A}
\end{figure}

\begin{prop}\label{prop:reduction-annulus}
    Let $\bW$ be a taut wall system on $\Sigma$. If $\bW$-transverse simple loop diagrams $\alpha$ and $\beta$ bound an annulus, then a sequence of double-point jumpings, bigon reduction moves and flip moves deform $\alpha$ and $\beta$ into $\alpha'$ and $\beta'$ as depicted in \cref{fig:reduct-A}, respectively.
    Moreover, if $\alpha$ and $\beta$ are $\bW$-minimal, then the deformation is composed only of double-point jumpings and flip moves.
\end{prop}
\begin{proof}
    Using \cref{lem:innermost-annulus} and the deformation in \cref{fig:loop-flip}, one can prove it similarly to the proof of \cref{prop:reduction-bigon}.
\end{proof}

\begin{cor}\label{cor:minimal-curve}
    For a taut wall system $\bW$ on $\Sigma$,
    any two $\bW$-minimal ideal arcs or loops representing the same homotopy class are related by a sequence of double-point jumpings and flip moves.
\end{cor}
\begin{proof}
    If two homotopic ideal arcs $\alpha$ and $\beta$ are $\bW$-minimal, then there exist subarcs of $\alpha$ and $\beta$ that bound a bigon with vertices $p,q \in \alpha \cap \beta$. 
    These subarcs are related by double-point jumpings by \cref{prop:reduction-bigon}.
    We also use flip moves when $p$ and $q$ are marked points.
    Similarly, one can confirm the assertion for loops by \cref{prop:reduction-annulus}.
\end{proof}

\begin{lem}\label{lem:W-minimal-lift}
    Let $\bW$ be a taut wall system on $\Sigma$.
    For any simple multicurve $D$ on $\Sigma$, there exists a $\bW$-minimal diagram $D_{\bW}$ consisting of mutually disjoint simple arcs and loops without curves bounding a disk or a monogon such that $D_{\bW}$ is homotopic to $D$ in $\Sigma$.  
\end{lem}

\begin{proof}
    For any simple multicurve $D$, we can deform it to a $\bW$-transverse diagram $D'$ on $(\Sigma, \bW)$ by a perturbation if necessary.
    We repeatedly reduce the innermost bigon bounded by $\bW$ and $D'$ by using \cref{lem:innermost-bigon,fig:detour-move}, and obtain a $\bW$-minimal simple multicurve $D_{\bW}$.
    We remark that the bigon reduction makes a simple multicurve into a simple multicurve with a smaller intersection with $\bW$.
\end{proof}

\begin{cor}\label{cor:minimal-multicurve}
    Let $\bW$ be a taut wall system on $\Sigma$.
    For any simple multicurve $D$ in $\Sigma$, the $\bW$-minimal simple multicurve $D_{\bW}$ related to $D$ is unique up to homotopy fixing $\bW$ setwisely, flip moves and double-point jumpings.
\end{cor}
\begin{proof}
    Let $D_{\bW}$ and $D'_{\bW}$ be two $\bW$-minimal simple multicurves that are homotopic (in $\Sigma$).
    \cref{prop:reduction-bigon} and \cref{prop:reduction-annulus} imply that $D'_{\bW}$ is deformed into a $\bW$-minimal simple multicurve homotopic to $D_{\bW}$ relative to $\bW$ by a sequence of bigon reduction moves, double-point jumpings, and flip moves.
\end{proof}

\appendix
\section{Generalizations of wall systems}\label{sec:generalization}
These generalizations show the potential to generalize coefficients of cluster algebras.

\subsection{Punctured surfaces}\label{subsec:punctures}
Let $\Sigma$ be a surface with marked points $\bM=\bM_{\partial}\cup\bM_{\circ}$, where $\bM_{\partial}:=\bM \cap \partial\Sigma$ and $\bM_{\circ}:=\bM\cap \Sigma\setminus\partial\Sigma$.
The points in $\bM_\circ$ are called punctures. 
In this setting, a wall system $\bW=(\sfC,J,\ell)$ on $\Sigma$ is defined so that the curves $\sfC$ are either loops or incident to $\bM$.
We consider the set $\mathsf{Tang}(\Sigma,\bW)$ of tangles in $(\Sigma,\bW)$ in the same way as in \cref{def:tangle}.
We are going to define skein algebras by adding skein relations at punctures.
Let $\cR$ be a quotient of $\bZ_{A,\bW}[v_p^{\pm 1}\mid p \in \bM_{\circ}]$ by the \emph{puncture condition at $p\in\bM_{\circ}$}:
\[
    Z_{p}:=\prod_{\gamma\in \sfC_p}z_{\ell(\gamma),+}=\prod_{\gamma\in \sfC_p} z_{\ell(\gamma),-}
\]
where $\sfC_p$ is the multiset of walls incident to $p$ whose multiplicity is given by the number of half-edges incident to $p$.\footnote{In our work in progress, we observe that 
the puncture condition can be removed if we consider an appropriate skein algebras of punctured surfaces with tagged arcs, modifying the Roger--Yang algebra. The classical limit of this algebra is more close to the cluster algebra, including those with coefficients.}

We define skein relations for each $p\in\bM_{\circ}$ as follows. 

\begin{dfn}[Wall-passing relation at punctures]
\begin{align*}
        \ \tikz[baseline=-.6ex, scale=.1, yshift=-2cm]{
        \coordinate (A) at (-10,0);
        \coordinate (B) at (10,0);
        \coordinate (P) at (0,0);
        \draw[webline, rounded corners] (P) -- (60:6) -- (120:10);
        \draw[wline] (P) -- (90:10);
        \draw[dashed] (30:10) arc (30:150:10cm);
        \draw[dashed] (P) -- (30:10);
        \draw[dashed] (P) -- (150:10);
        \draw[fill=white] (P) circle [radius=20pt];
        \node at (90:10) [below right]{\scriptsize $j$};
    }\ 
    =z_{j,+}
    \ \tikz[baseline=-.6ex, scale=.1, yshift=-2cm]{
        \coordinate (A) at (-10,0);
        \coordinate (B) at (10,0);
        \coordinate (P) at (0,0);
        \draw[webline, rounded corners] (P) -- (120:10);
        \draw[wline] (P) -- (90:10);
        \draw[dashed] (30:10) arc (30:150:10cm);
        \draw[dashed] (P) -- (30:10);
        \draw[dashed] (P) -- (150:10);
        \draw[fill=white] (P) circle [radius=20pt];
        \node at (90:10) [below right]{\scriptsize $j$};
    }\ ,
    \ \tikz[baseline=-.6ex, scale=.1, yshift=-2cm]{
        \coordinate (A) at (-10,0);
        \coordinate (B) at (10,0);
        \coordinate (P) at (0,0);
        \draw[webline, rounded corners] (P) -- (120:6) -- (60:10);
        \draw[wline] (P) -- (90:10);
        \draw[dashed] (30:10) arc (30:150:10cm);
        \draw[dashed] (P) -- (30:10);
        \draw[dashed] (P) -- (150:10);
        \draw[fill=white] (P) circle [radius=20pt];
        \node at (90:10) [below left]{\scriptsize $j$};
    }\ 
    =z_{j,-}
    \ \tikz[baseline=-.6ex, scale=.1, yshift=-2cm]{
        \coordinate (A) at (-10,0);
        \coordinate (B) at (10,0);
        \coordinate (P) at (0,0);
        \draw[webline, rounded corners] (P) -- (60:10);
        \draw[wline] (P) -- (90:10);
        \draw[dashed] (30:10) arc (30:150:10cm);
        \draw[dashed] (P) -- (30:10);
        \draw[dashed] (P) -- (150:10);
        \draw[fill=white] (P) circle [radius=20pt];
        \node at (90:10) [below left]{\scriptsize $j$};
    }\ ,\\
\end{align*}    
\end{dfn}
\begin{dfn}[Puncture-framing relation]\label{def:pframing-quantum}
    For each puncture $p\in\punc$,
    \begin{gather*}
        \ \tikz[baseline=-.6ex, scale=.1]{
            \coordinate (A) at (-8,0);
            \coordinate (B) at (8,0);
            \coordinate (P) at (0,0);
            \draw[webline] (P) circle [radius=4cm];
            \draw[wline] (P) -- (90:8);
            \draw[dashed] (P) circle [radius=8cm];
            \draw[fill=white] (P) circle [radius=20pt];
            \node at (P) [left]{\scriptsize $p$};
            \node at (90:8) [below right]{\scriptsize $\sfC_p$};
        }\ 
        =Z_{p}(A+A^{-1})
        \ \tikz[baseline=-.6ex, scale=.1]{
            \coordinate (A) at (-8,0);
            \coordinate (B) at (8,0);
            \coordinate (P) at (0,0);
            \draw[wline] (P) -- (90:8);
            \draw[dashed] (P) circle [radius=8cm];
            \draw[fill=white] (P) circle [radius=20pt];
            \node at (P) [left]{\scriptsize $p$};
            \node at (90:8) [below right]{\scriptsize $\sfC_p$};
        }\ .\label{rel:pframing-quantum}
    \end{gather*}
\end{dfn}

In the local diagram above, the arc with $\sfC_p$ incident to $p$ means a bunch of parallel half-edges incident to $p$, as shown in \cref{fig:puncture}.
Note that $Z_p$ is independent of an arrangement of the 
half-edges at $p$.

\begin{figure}
    \begin{tikzpicture}[scale=.1]
        \coordinate (A) at (-10,0);
        \coordinate (B) at (10,0);
        \coordinate (P) at (0,0);
        \draw[wline] (P) -- (130:10);
        \draw[wline] (P) -- (110:10);
        \node at (90:8) {\scriptsize $\cdots$};
        \draw[wline] (P) -- (70:10);
        \draw[dashed] (P) circle [radius=10cm];
        \draw[fill=white] (P) circle [radius=20pt];
        \node at (P) [below]{\scriptsize $p$};
        \node at (130:12) {\scriptsize $l_1$};
        \node at (110:12) {\scriptsize $l_2$};
        \node at (90:12) {\scriptsize $\cdots$};
        \node at (70:12) {\scriptsize $l_n$};
    \end{tikzpicture}
    \caption{Here is a local diagram of walls at a puncture $p$. Here $\sfC_p=\{l_1,l_2,\dots,l_n\}$ is the multiset of their labels.}
    \label{fig:puncture}
\end{figure}

\begin{dfn}[Roger--Yang skein relation, {\em cf.}~\cite{RogerYang14}]\label{def:RY-skein}
    For any puncture $p\in \punc$,
    \begin{align*}
        \ \tikz[baseline=-.6ex, scale=.1]{
            \coordinate (A) at (-8,0);
            \coordinate (B) at (8,0);
            \coordinate (P) at (0,0);
            \draw[webline] (P) -- (A);
            \draw[webline, shorten <=.2cm] (P) -- (B);
            \draw[wline] (P) -- (90:8);
            \draw[dashed] (P) circle [radius=8cm];
            \draw[fill=white] (P) circle [radius=20pt];
            \node at (P) [below]{\scriptsize $p$};
            \node at (90:8) [below right]{\scriptsize $\sfC_p$};
        }\ 
        &=v_p^{-1}\Big(
        A^{\frac{1}{2}}
        \ \tikz[baseline=-.6ex, scale=.1]{
            \coordinate (A) at (-8,0);
            \coordinate (B) at (8,0);
            \coordinate (P) at (0,0);
            \draw[webline] (A) -- ($(P)+(-3,0)$) to[out=north, in=west] ($(P)+(0,3)$) to[out=east, in=north] ($(P)+(3,0)$) -- (B);
            \draw[wline] (P) -- (90:8);
            \draw[dashed] (P) circle [radius=8cm];
            \draw[fill=white] (P) circle [radius=20pt];
            \node at (P) [below]{\scriptsize $p$};
            \node at (90:8) [below right]{\scriptsize $\sfC_p$};
        }\ 
        +A^{-\frac{1}{2}}Z_{p}
        \ \tikz[baseline=-.6ex, scale=.1]{
            \coordinate (A) at (-8,0);
            \coordinate (B) at (8,0);
            \coordinate (P) at (0,0);
            \draw[webline] (A) -- ($(P)+(-3,0)$) to[out=south, in=west] ($(P)+(0,-3)$) to[out=east, in=south] ($(P)+(3,0)$) -- (B);
            \draw[wline] (P) -- (90:8);
            \draw[dashed] (P) circle [radius=8cm];
            \draw[fill=white] (P) circle [radius=20pt];
            \node at (P) [above left]{\scriptsize $p$};
            \node at (90:8) [below right]{\scriptsize $\sfC_p$};
        }\ 
        \Big)\\
        \ \tikz[baseline=-.6ex, scale=.1]{
            \coordinate (A) at (-8,0);
            \coordinate (B) at (8,0);
            \coordinate (P) at (0,0);
            \draw[webline, shorten <=.2cm] (P) -- (A);
            \draw[webline] (P) -- (B);
            \draw[wline] (P) -- (90:8);
            \draw[dashed] (P) circle [radius=8cm];
            \draw[fill=white] (P) circle [radius=20pt];
            \node at (P) [below]{\scriptsize $p$};
            \node at (90:8) [below right]{\scriptsize $\sfC_p$};
        }\ 
        &=v_p^{-1}\Big(
        A^{-\frac{1}{2}}
        \ \tikz[baseline=-.6ex, scale=.1]{
            \coordinate (A) at (-8,0);
            \coordinate (B) at (8,0);
            \coordinate (P) at (0,0);
            \draw[webline] (A) -- ($(P)+(-3,0)$) to[out=north, in=west] ($(P)+(0,3)$) to[out=east, in=north] ($(P)+(3,0)$) -- (B);
            \draw[wline] (P) -- (90:8);
            \draw[dashed] (P) circle [radius=8cm];
            \draw[fill=white] (P) circle [radius=20pt];
            \node at (P) [below]{\scriptsize $p$};
            \node at (90:8) [below right]{\scriptsize $\sfC_p$};
        }\ 
        +A^{\frac{1}{2}}Z_{p}
        \ \tikz[baseline=-.6ex, scale=.1]{
            \coordinate (A) at (-8,0);
            \coordinate (B) at (8,0);
            \coordinate (P) at (0,0);
            \draw[webline] (A) -- ($(P)+(-3,0)$) to[out=south, in=west] ($(P)+(0,-3)$) to[out=east, in=south] ($(P)+(3,0)$) -- (B);
            \draw[wline] (P) -- (90:8);
            \draw[dashed] (P) circle [radius=8cm];
            \draw[fill=white] (P) circle [radius=20pt];
            \node at (P) [above left]{\scriptsize $p$};
            \node at (90:8) [below right]{\scriptsize $\sfC_p$};
        }\ 
        \Big)
        .
    \end{align*}
\end{dfn}

One can confirm the following relation from the above skein relations.
\begin{lem}[Punctured loop relation]\label{def:pkink}
    For any puncture $p\in\punc$,
    \begin{align*}
        \ \tikz[baseline=-.6ex, scale=.1]{
            \coordinate (A) at (-8,0);
            \coordinate (B) at (8,0);
            \coordinate (P) at (0,0);
            \draw[webline, rounded corners, shorten <=.2cm] (P) -- (-120:5) -- (-90:5) -- (-60:5) -- (P);
            \draw[wline] (P) -- (90:8);
            \draw[dashed] (P) circle [radius=8cm];
            \draw[fill=white] (P) circle [radius=20pt];
            \node at (P) [left]{\scriptsize $p$};
            \node at (90:8) [below right]{\scriptsize $\sfC_p$};
        }\ 
        =-\frac{Z_{p}}{v_p}A^{\frac{3}{2}}(A-A^{-1})
        \ \tikz[baseline=-.6ex, scale=.1]{
            \coordinate (A) at (-8,0);
            \coordinate (B) at (8,0);
            \coordinate (P) at (0,0);
            \draw[wline] (P) -- (90:8);
            \draw[dashed] (P) circle [radius=8cm];
            \draw[fill=white] (P) circle [radius=20pt];
            \node at (P) [left]{\scriptsize $p$};
            \node at (90:8) [below right]{\scriptsize $\sfC_p$};
        }\ ,\quad
        \ \tikz[baseline=-.6ex, scale=.1]{
            \coordinate (A) at (-8,0);
            \coordinate (B) at (8,0);
            \coordinate (P) at (0,0);
            \draw[webline, rounded corners, shorten >=.2cm] (P) -- (-120:5) -- (-90:5) -- (-60:5) -- (P);
            \draw[wline] (P) -- (90:8);
            \draw[dashed] (P) circle [radius=8cm];
            \draw[fill=white] (P) circle [radius=20pt];
            \node at (P) [left]{\scriptsize $p$};
            \node at (90:8) [below right]{\scriptsize $\sfC_p$};
        }\ 
        =\frac{Z_{p}}{v_p}A^{-\frac{3}{2}}(A-A^{-1})
        \ \tikz[baseline=-.6ex, scale=.1]{
            \coordinate (A) at (-8,0);
            \coordinate (B) at (8,0);
            \coordinate (P) at (0,0);
            \draw[wline] (P) -- (90:8);
            \draw[dashed] (P) circle [radius=8cm];
            \draw[fill=white] (P) circle [radius=20pt];
            \node at (P) [left]{\scriptsize $p$};
            \node at (90:8) [below right]{\scriptsize $\sfC_p$};
        }\ .\label{rel:pkink}
    \end{align*}
\end{lem}

\begin{rem}
    The puncture condition is required to combine the Roger--Yang skein relation with the wall-passing relation for both the quantum and classical cases.
    For example, at $A^{\frac{1}{2}}=1$,
    \begin{align*}
        \ \tikz[baseline=-.6ex, scale=.1, yshift=-5cm]{
            \coordinate (A) at (-8,0);
            \coordinate (B) at (8,0);
            \coordinate (P) at (0,0);
            \coordinate (Q) at (90:10);
            \draw[webline, rounded corners] (P) -- (7,5) -- (Q);
            \draw[webline, rounded corners] (P) -- (4,5) -- (Q);
            \draw[wline] (P) -- (Q);
            \bdryline{(A)}{(B)}{2cm}
            \draw[fill=black] (P) circle [radius=20pt];
            \draw[fill=white] (Q) circle [radius=20pt];
            \node at (90:6) [left]{\scriptsize $\sfC_p$};
        }\ 
        =
        \ \tikz[baseline=-.6ex, scale=.1, yshift=-5cm]{
            \coordinate (A) at (-8,0);
            \coordinate (B) at (8,0);
            \coordinate (P) at (0,0);
            \coordinate (Q) at (90:10);
            \draw[webline, rounded corners] (P) -- (-7,5) -- (Q);
            \draw[webline, rounded corners] (P) -- (-4,5) -- (Q);
            \draw[wline] (P) -- (Q);
            \bdryline{(A)}{(B)}{2cm}
            \draw[fill=black] (P) circle [radius=20pt];
            \draw[fill=white] (Q) circle [radius=20pt];
            \node at (90:6) [right]{\scriptsize $\sfC_p$};
        }\ 
    \end{align*}
    holds by wall-passing relation.
    On the other hand, 
    \begin{align*}
        \ \tikz[baseline=-.6ex, scale=.1, yshift=-5cm]{
            \coordinate (A) at (-8,0);
            \coordinate (B) at (8,0);
            \coordinate (P) at (0,0);
            \coordinate (Q) at (90:10);
            \draw[webline, rounded corners] (P) -- (7,5) -- (Q);
            \draw[webline, rounded corners] (P) -- (4,5) -- (Q);
            \draw[wline] (P) -- (Q);
            \bdryline{(A)}{(B)}{2cm}
            \draw[fill=black] (P) circle [radius=20pt];
            \draw[fill=white] (Q) circle [radius=20pt];
            \node at (90:6) [left]{\scriptsize $\sfC_p$};
        }\ 
        =v_p^{-1}\prod_{\gamma\in \sfC_p}z_{\ell(\gamma),+}
        \ \tikz[baseline=-.6ex, scale=.1, yshift=-5cm]{
            \coordinate (A) at (-8,0);
            \coordinate (B) at (8,0);
            \coordinate (P) at (0,0);
            \coordinate (Q) at (90:10);
            \draw[webline, rounded corners] (P) -- (6,4) to[out=north, in=east] (0,14) to[out=west, in=north] (-6,4) -- (P);
            \draw[wline] (P) -- (Q);
            \bdryline{(A)}{(B)}{2cm}
            \draw[fill=black] (P) circle [radius=20pt];
            \draw[fill=white] (Q) circle [radius=20pt];
            \node at (90:6) [left=-1pt]{\scriptsize $\sfC_p$};
        }\ \text{ and }
        \ \tikz[baseline=-.6ex, scale=.1, yshift=-5cm]{
            \coordinate (A) at (-8,0);
            \coordinate (B) at (8,0);
            \coordinate (P) at (0,0);
            \coordinate (Q) at (90:10);
            \draw[webline, rounded corners] (P) -- (-7,5) -- (Q);
            \draw[webline, rounded corners] (P) -- (-4,5) -- (Q);
            \draw[wline] (P) -- (Q);
            \bdryline{(A)}{(B)}{2cm}
            \draw[fill=black] (P) circle [radius=20pt];
            \draw[fill=white] (Q) circle [radius=20pt];
            \node at (90:6) [right]{\scriptsize $\sfC_p$};
        }\ 
        =
        v_p^{-1}\prod_{\gamma\in \sfC_p}z_{\ell(\gamma),-}
        \ \tikz[baseline=-.6ex, scale=.1, yshift=-5cm]{
            \coordinate (A) at (-8,0);
            \coordinate (B) at (8,0);
            \coordinate (P) at (0,0);
            \coordinate (Q) at (90:10);
            \draw[webline, rounded corners] (P) -- (6,4) to[out=north, in=east] (0,14) to[out=west, in=north] (-6,4) -- (P);
            \draw[wline] (P) -- (Q);
            \bdryline{(A)}{(B)}{2cm}
            \draw[fill=black] (P) circle [radius=20pt];
            \draw[fill=white] (Q) circle [radius=20pt];
            \node at (90:6) [left=-1pt]{\scriptsize $\sfC_p$};
        }\ 
    \end{align*}
    hold by Roger-Yang skein relation and wall-passing relation. Hence $\prod_{\gamma\in \sfC_p}z_{\ell(\gamma),+}=\prod_{\gamma\in \sfC_p}z_{\ell(\gamma),-}$ is required.
\end{rem}

\subsection{Branched wall systems}
Let us define a \emph{branched wall system} $\bW=(\sfC,J,\ell)$ on $\Sigma$ by relaxing the condition of $\sfC_{\mathrm{arc}}$ as follows:
\begin{itemize}
    \item $\sfC_{\mathrm{arc}}$ is a set of mutually distinct simple arcs whose endpoints lie in $\Sigma\setminus\partial\Sigma$ or marked points in $M$, and possibly 
    have transverse double points in their interiors. We allow the curves to share the endpoints, while an endpoint on the interior of another curve in $\sfC$ is prohibited.
\end{itemize}
We call a pair $(\gamma,\ell(\gamma))$ a wall, an endpoint of $\sfC_{\mathrm{arc}}$ in $\Sigma\setminus\partial\Sigma$ a \emph{pole} of $\bW$.
We denote the set of poles by $V(\bW)$.
We suppose that a tangle in $(\Sigma,\bW)$ avoids touching the poles of $\bW$.
The coefficient ring is defined to be
$
    \bZ_{A,\bW}:=\bZ[A^{\pm 1/2},z_{j,\pm}^{\pm 1/2} \mid j\in J].
$
We use the notation $a_{j}^{1/2}:=z_{j,+}^{1/2}z_{j,-}^{1/2}$ for any $j\in J$, and $\sfC_p$ denotes the multiset of walls in $\sfC_{\mathrm{arc}}$ whose endpoints share a pole $p\in V(\bW)$ whose multiplicity is given by the number of half-edges incident to $p$.

\begin{dfn}[Pole relation]\label{def:pole}
    For any pole $p\in V(\bW)$,
    \begin{align*}
        \ \tikz[baseline=-.6ex, scale=.1]{
            \coordinate (A) at (-8,0);
            \coordinate (B) at (8,0);
            \coordinate (P) at (0,0);
            \draw[webline] (A) -- ($(P)+(-3,0)$) to[out=north, in=west] ($(P)+(0,3)$) to[out=east, in=north] ($(P)+(3,0)$) -- (B);
            \draw[wline] (P) -- (90:8);
            \draw[dashed] (P) circle [radius=8cm];
            \draw[fill=mygreen] (P) circle [radius=20pt];
            \node at (P) [below]{\scriptsize $p$};
            \node at (90:8) [below right]{\scriptsize $\sfC_p$};
        }\ 
        =\prod_{\gamma\in \sfC_p}a_{\ell(\gamma)}^{\frac{1}{2}}
        \ \tikz[baseline=-.6ex, scale=.1]{
            \coordinate (A) at (-8,0);
            \coordinate (B) at (8,0);
            \coordinate (P) at (0,0);
            \draw[webline] (A) -- ($(P)+(-3,0)$) to[out=south, in=west] ($(P)+(0,-3)$) to[out=east, in=south] ($(P)+(3,0)$) -- (B);
            \draw[wline] (P) -- (90:8);
            \draw[dashed] (P) circle [radius=8cm];
            \draw[fill=mygreen] (P) circle [radius=20pt];
            \node at (P) [above left]{\scriptsize $p$};
            \node at (90:8) [below right]{\scriptsize $\sfC_p$};
        }\ 
    \end{align*}
\end{dfn}

One can confirm the following relation by using \cref{def:pole} and \eqref{rel:wall-pass-int}.

\begin{lem} For any $p\in V(\bW)$ and any partition $\sfC_p=\sfC_{p,1}\sqcup \sfC_{p,2}$,
    \begin{align*}
        \ \tikz[baseline=-.6ex, scale=.1]{
            \coordinate (A) at (-8,0);
            \coordinate (B) at (8,0);
            \coordinate (P) at (0,0);
            \draw[webline] (A) -- ($(P)+(-3,0)$) to[out=north, in=west] ($(P)+(0,3)$) to[out=east, in=north] ($(P)+(3,0)$) -- (B);
            \draw[wline] (P) -- (90:8);
            \draw[wline] (P) -- (-90:8);
            \draw[dashed] (P) circle [radius=8cm];
            \draw[fill=mygreen] (P) circle [radius=20pt];
            \node at (P) [below left]{\scriptsize $p$};
            \node at (90:8) [above]{\scriptsize $\sfC_{p,1}$};
            \node at (-90:8) [below]{\scriptsize $\sfC_{p,2}$};
        }\ 
        =\frac{\prod_{\gamma\in \sfC_{p,1}}a_{\ell(\gamma)}^{\frac{1}{2}}}{\prod_{\eta\in \sfC_{p,2}}a_{\ell(\eta)}^{\frac{1}{2}}}
        \ \tikz[baseline=-.6ex, scale=.1]{
            \coordinate (A) at (-8,0);
            \coordinate (B) at (8,0);
            \coordinate (P) at (0,0);
            \draw[webline] (A) -- ($(P)+(-3,0)$) to[out=south, in=west] ($(P)+(0,-3)$) to[out=east, in=south] ($(P)+(3,0)$) -- (B);
            \draw[wline] (P) -- (90:8);
            \draw[wline] (P) -- (-90:8);
            \draw[dashed] (P) circle [radius=8cm];
            \draw[fill=mygreen] (P) circle [radius=20pt];
            \node at (P) [above left]{\scriptsize $p$};
            \node at (90:8) [above]{\scriptsize $\sfC_{p,1}$};
            \node at (-90:8) [below]{\scriptsize $\sfC_{p,2}$};
        }\ 
    \end{align*}
\end{lem}

\subsection{Oriented wall systems}
Let us define an \emph{oriented wall system} $\overrightarrow{\bW}=(\overrightarrow{\sfC},J,\vec{\ell})$ on an unpunctured marked surface $\Sigma$ (namely, $\bM_\circ=\emptyset$ as in the body text) as a wall system with oriented underlying curves $\overrightarrow{\sfC}$.
One can similarly consider the set of underlying arcs $\overrightarrow{\sfC}_{\mathrm{arc}}$, the set of underlying loops $\overrightarrow{\sfC}_{\mathrm{loop}}$, the labeling $\vec{\ell}\colon\overrightarrow{\sfC}\to J$, and tangles in $(\Sigma, \overrightarrow{\bW})$.
Let us introduce skein relations for a tangle in $(\Sigma,\overrightarrow{\bW})$ to define the skein algebra $\SK{\Sigma,\overrightarrow{\bW}}$.
The coefficient ring $\cR_{\overrightarrow{\bW}}$ of the skein algebra $\SK{\Sigma,\overrightarrow{\bW}}$ is a quotient of the ring
$
    \bZ[A^{\pm 1/2},(z_{j,\pm}^{\uparrow})^{\pm 1},(z_{j,\pm}^{\downarrow})^{\pm 1}\mid j\in J]
$
modulo an ideal generated by $\{\,z_{j,+}^{\uparrow}z_{j,-}^{\uparrow}-z_{j,+}^{\downarrow}z_{j,-}^{\downarrow}\mid j\in J\,\}$. 
The skein relations are defined as follows.
\begin{dfn}
    \begin{align*}
    \ \tikz[baseline=-.6ex, scale=.1, yshift=-5cm]{
        \draw[webline, rounded corners] (-5,0) -- (5,5) -- (-5,10);
        \draw[wline,->-={.5}{black}] (0,-1) -- (0,11);
        \node at (0,11) [below right]{\scriptsize $j$};
    }\ 
    &=a'_j
    \ \tikz[baseline=-.6ex, scale=.1, yshift=-5cm]{
            \draw[webline, rounded corners] (-5,0) -- (-2,5) -- (-5,10);
            \draw[wline,->-={.5}{black}] (0,-1) -- (0,11);
            \node at (0,11) [below right]{\scriptsize $j$};
    }\ ,&
    \ \tikz[baseline=-.6ex, scale=.1, yshift=-5cm]{
        \draw[webline, rounded corners] (-5,0) -- (5,5) -- (-5,10);
        \draw[wline,-<-={.5}{black}] (0,-1) -- (0,11);
        \node at (0,11) [below right]{\scriptsize $j$};
    }\ 
    &=a'_j
    \ \tikz[baseline=-.6ex, scale=.1, yshift=-5cm]{
            \draw[webline, rounded corners] (-5,0) -- (-2,5) -- (-5,10);
            \draw[wline,-<-={.5}{black}] (0,-1) -- (0,11);
            \node at (0,11) [below right]{\scriptsize $j$};
    }\  
    \\
    \ \tikz[baseline=-.6ex, scale=.1, yshift=-2cm]{
        \coordinate (A) at (-10,0);
        \coordinate (B) at (10,0);
        \coordinate (P) at (0,0);
        \draw[webline, rounded corners] (P) -- (60:6) -- (120:10);
        \draw[wline,->-={.5}{black}] (P) -- (90:10);
        \draw[dashed] (30:10) arc (30:150:10cm);
        \draw[dashed] (P) -- (30:10);
        \draw[dashed] (P) -- (150:10);
        \draw[fill=black] (P) circle [radius=20pt];
        \node at (90:10) [below right]{\scriptsize $j$};
    }\ 
    &=z_{j,+}^{\uparrow}
    \ \tikz[baseline=-.6ex, scale=.1, yshift=-2cm]{
        \coordinate (A) at (-10,0);
        \coordinate (B) at (10,0);
        \coordinate (P) at (0,0);
        \draw[webline, rounded corners] (P) -- (120:10);
        \draw[wline,->-={.5}{black}] (P) -- (90:10);
        \draw[dashed] (30:10) arc (30:150:10cm);
        \draw[dashed] (P) -- (30:10);
        \draw[dashed] (P) -- (150:10);
        \draw[fill=black] (P) circle [radius=20pt];
        \node at (90:10) [below right]{\scriptsize $j$};
    }\ ,&
    \ \tikz[baseline=-.6ex, scale=.1, yshift=-2cm]{
        \coordinate (A) at (-10,0);
        \coordinate (B) at (10,0);
        \coordinate (P) at (0,0);
        \draw[webline, rounded corners] (P) -- (120:6) -- (60:10);
        \draw[wline,->-={.5}{black}] (P) -- (90:10);
        \draw[dashed] (30:10) arc (30:150:10cm);
        \draw[dashed] (P) -- (30:10);
        \draw[dashed] (P) -- (150:10);
        \draw[fill=black] (P) circle [radius=20pt];
        \node at (90:10) [below left]{\scriptsize $j$};
    }\ 
    &=z_{j,-}^{\uparrow}
    \ \tikz[baseline=-.6ex, scale=.1, yshift=-2cm]{
        \coordinate (A) at (-10,0);
        \coordinate (B) at (10,0);
        \coordinate (P) at (0,0);
        \draw[webline, rounded corners] (P) -- (60:10);
        \draw[wline,->-={.5}{black}] (P) -- (90:10);
        \draw[dashed] (30:10) arc (30:150:10cm);
        \draw[dashed] (P) -- (30:10);
        \draw[dashed] (P) -- (150:10);
        \draw[fill=black] (P) circle [radius=20pt];
        \node at (90:10) [below left]{\scriptsize $j$};
    }\ ,\\
        \ \tikz[baseline=-.6ex, scale=.1, yshift=-2cm]{
        \coordinate (A) at (-10,0);
        \coordinate (B) at (10,0);
        \coordinate (P) at (0,0);
        \draw[webline, rounded corners] (P) -- (60:6) -- (120:10);
        \draw[wline,,-<-={.5}{black}] (P) -- (90:10);
        \draw[dashed] (30:10) arc (30:150:10cm);
        \draw[dashed] (P) -- (30:10);
        \draw[dashed] (P) -- (150:10);
        \draw[fill=black] (P) circle [radius=20pt];
        \node at (90:10) [below right]{\scriptsize $j$};
    }\ 
    &=z_{j,+}^{\downarrow}
    \ \tikz[baseline=-.6ex, scale=.1, yshift=-2cm]{
        \coordinate (A) at (-10,0);
        \coordinate (B) at (10,0);
        \coordinate (P) at (0,0);
        \draw[webline, rounded corners] (P) -- (120:10);
        \draw[wline,-<-={.5}{black}] (P) -- (90:10);
        \draw[dashed] (30:10) arc (30:150:10cm);
        \draw[dashed] (P) -- (30:10);
        \draw[dashed] (P) -- (150:10);
        \draw[fill=black] (P) circle [radius=20pt];
        \node at (90:10) [below right]{\scriptsize $j$};
    }\ ,&
    \ \tikz[baseline=-.6ex, scale=.1, yshift=-2cm]{
        \coordinate (A) at (-10,0);
        \coordinate (B) at (10,0);
        \coordinate (P) at (0,0);
        \draw[webline, rounded corners] (P) -- (120:6) -- (60:10);
        \draw[wline,-<-={.5}{black}] (P) -- (90:10);
        \draw[dashed] (30:10) arc (30:150:10cm);
        \draw[dashed] (P) -- (30:10);
        \draw[dashed] (P) -- (150:10);
        \draw[fill=black] (P) circle [radius=20pt];
        \node at (90:10) [below left]{\scriptsize $j$};
    }\ 
    &=z_{j,-}^{\downarrow}
    \ \tikz[baseline=-.6ex, scale=.1, yshift=-2cm]{
        \coordinate (A) at (-10,0);
        \coordinate (B) at (10,0);
        \coordinate (P) at (0,0);
        \draw[webline, rounded corners] (P) -- (60:10);
        \draw[wline,-<-={.5}{black}] (P) -- (90:10);
        \draw[dashed] (30:10) arc (30:150:10cm);
        \draw[dashed] (P) -- (30:10);
        \draw[dashed] (P) -- (150:10);
        \draw[fill=black] (P) circle [radius=20pt];
        \node at (90:10) [below left]{\scriptsize $j$};
    }\ ,
\end{align*}
and \labelcref{rel:wall-R3} with any oriented walls, where $a'_{j}:=z_{j,+}^{\uparrow}z_{j,-}^{\uparrow}=z_{j,+}^{\downarrow}z_{j,-}^{\downarrow}$ for $j\in J$.
\end{dfn}

We remark that
\[
        \ \tikz[baseline=-.6ex, scale=.1, yshift=-6cm]{
            \draw[webline] (0,0) to[bend left=80] (90:12);
            \draw[wline,->-={.5}{black}] (0,0) -- (90:12);
            \draw[fill=black] (0,0) circle [radius=20pt];
            \draw[fill=black] (90:12) circle [radius=20pt];
        }\ 
        \neq
        \ \tikz[baseline=-.6ex, scale=.1, yshift=-6cm]{
            \draw[webline] (0,0) to[bend right=80] (90:12);
            \draw[wline,->-={.5}{black}] (0,0) -- (90:12);
            \draw[fill=black] (0,0) circle [radius=20pt];
            \draw[fill=black] (90:12) circle [radius=20pt];
        }\ ,
        \ \tikz[baseline=-.6ex, scale=.1]{
            \coordinate (A) at (-8,0);
            \coordinate (B) at (8,0);
            \coordinate (P) at (0,0);
            \draw[webline] (P) circle [radius=3cm];
            \draw[wline,->-={.5}{black}] (P) circle [radius=5cm];
            \draw[dashed] (P) circle [radius=2cm];
            \draw[dashed] (P) circle [radius=8cm];
        }\ 
        =
        \ \tikz[baseline=-.6ex, scale=.1]{
            \coordinate (A) at (-8,0);
            \coordinate (B) at (8,0);
            \coordinate (P) at (0,0);
            \draw[webline] (P) circle [radius=7cm];
            \draw[wline,->-={.5}{black}] (P) circle [radius=5cm];
            \draw[dashed] (P) circle [radius=2cm];
            \draw[dashed] (P) circle [radius=8cm];
        }\ 
\]
in the skein algebra $\SK{\Sigma,\overrightarrow{\bW}}$.
The skein algebra $\SK{\Sigma,\bW}$ for an (unoriented) wall system $\bW=(\sfC,J,\ell)$ is realized in $\SK{\Sigma,\overrightarrow{\bW}}$ whose oriented wall system $\overrightarrow{\bW}=(\overrightarrow{\sfC},J,\vec{\ell})$ is obtained as follows:
\[
        \ \tikz[baseline=-.6ex, scale=.1, yshift=-6cm]{
            \draw[wline] (0,0) -- (90:12);
            \draw[fill=black] (0,0) circle [radius=20pt];
            \draw[fill=black] (90:12) circle [radius=20pt];
            \node at (0,6) [left]{\scriptsize $j$};
        }\ 
        \rightsquigarrow
        \ \tikz[baseline=-.6ex, scale=.1, yshift=-6cm]{
            \draw[wline,->-={.5}{black}] (0,0) to[bend left] (90:12);
            \draw[wline,-<-={.5}{black}] (0,0) to[bend right] (90:12);
            \draw[fill=black] (0,0) circle [radius=20pt];
            \draw[fill=black] (90:12) circle [radius=20pt];
            \node at (-4,6) [above]{\scriptsize $j$};
            \node at (4,6) [above]{\scriptsize $j$};
        }\ ,
        \ \tikz[baseline=-.6ex, scale=.1]{
            \coordinate (A) at (-8,0);
            \coordinate (B) at (8,0);
            \coordinate (P) at (0,0);
            \draw[wline] (P) circle [radius=5cm];
            \draw[dashed] (P) circle [radius=2cm];
            \draw[dashed] (P) circle [radius=8cm];
            \node at (0:7) {\scriptsize $j'$};
        }\ 
        \rightsquigarrow
        \ \tikz[baseline=-.6ex, scale=.1]{
            \coordinate (A) at (-8,0);
            \coordinate (B) at (8,0);
            \coordinate (P) at (0,0);
            \draw[wline,->-={.5}{black}] (P) circle [radius=4cm];
            \draw[wline,-<-={.5}{black}] (P) circle [radius=6cm];
            \draw[dashed] (P) circle [radius=2cm];
            \draw[dashed] (P) circle [radius=8cm];
            \node at (0:8) {\scriptsize $j'$};
            \node at (0:2) {\scriptsize $j'$};
        }\ 
\]
where $j,j'\in J$.
Each wall of $\bW$ corresponds to a pair of anti-parallel curves with the same label. Then we get a $\bZ_{A,\bW}$-algebra embedding
\begin{align*}
    \SK{\Sigma,\bW} \to \SK{\Sigma, \overrightarrow{\bW}},
\end{align*}
where the $\bZ_{A,\bW}$-algebra structure of the latter is given by the ring homomorphism $\bZ_{A,\bW} \to \cR_{\overrightarrow{\bW}}$ such that 
$z_{j,+}\mapsto z_{j,+}^{\uparrow}z_{j,+}^{\downarrow}$, and $z_{j,-}\mapsto z_{j,-}^{\uparrow}z_{j,-}^{\downarrow}$ for any $j\in J$.



\end{document}